\documentclass[11pt, a4paper, twoside]{article}
\usepackage{bbm}
\usepackage{pifont}
\usepackage{bbding}
\usepackage{mathrsfs}
\usepackage{pgf,tikz}
\usepackage{amsfonts,amssymb,amsmath,indentfirst,amsthm}
\usepackage{epsfig}
\usepackage{fancyhdr}
\usepackage{graphicx}
\usepackage{cite}
\usepackage{url}
\usepackage{xcolor}
\usepackage{cite}
\usepackage{calc}
\usepackage{cases}

\def\ess~inf{\mathop{\rm ess~inf}}

\numberwithin{equation}{section}

\newenvironment{key words}{\emph{\texttt{Keywords}}\mbox{  }}{ }
\newtheorem{theorem}{Theorem}[section]
\newtheorem{lemma}[theorem]{Lemma}

\renewenvironment{proof}{\noindent{\textbf{Proof.}}}{\hfill$\Box$}
\theoremstyle{remark}

\theoremstyle{plain}

\makeatletter

\newcommand{\Rmnum}[1]{\expandafter\@slowromancap\romannumeral #1@}
\makeatother

\begin{document}

\renewcommand{\headrulewidth}{0.5pt}
\renewcommand{\thefootnote}{\fnsymbol {footnote}}
\title
{\textbf{On convergence properties for generalized Schr\"{o}dinger operators along tangential curves} \thanks{This work is supported by Natural Science Foundation of China (No.11871452; No.11601427);
China Postdoctoral Science Foundation (No. 2021M693139); the Fundamental Research Funds for the Central Universities (No. E1E40109).}}

\author{{Wenjuan Li and Huiju Wang }  }
\date{}
 \maketitle

 {\bf Abstract:}  In this paper, we consider convergence properties for  generalized Schr\"{o}dinger operators along tangential curves in $\mathbb{R}^{n} \times \mathbb{R}$ with less smoothness comparing with Lipschitz condition. Firstly, we obtain sharp convergence rate for generalized Schr\"{o}dinger operators with polynomial growth along tangential curves in $\mathbb{R}^{n} \times \mathbb{R}$, $n \ge 1$. Secondly, it was open until now on pointwise convergence of solutions to the Schr\"{o}dinger equation along non-$C^1$ curves in $\mathbb{R}^{n} \times \mathbb{R}$, $n\geq 2$, we obtain the corresponding results along  some tangential curves when  $n=2$ by the broad-narrow argument and polynomial partitioning. Moreover, the corresponding convergence rate will follow. Thirdly, we get the  convergence result along a family of restricted tangential curves  in $\mathbb{R} \times \mathbb{R}$. As a consequence, we obtain the sharp  $L^p$-Schr\"{o}dinger maximal estimates along tangential curves  in $\mathbb{R} \times \mathbb{R}$.



\tableofcontents

\section{\textbf{Introduction}\label{Section 1}}

\noindent
Let $P(\xi)$ be a real continuous function defined on $\mathbb{R}^{n}$, $D=\frac{1}{i}(\frac{\partial}{\partial x_{1}}\frac{\partial}{\partial x_{2}},...,\frac{\partial}{\partial x_{n}}).$ $P(D)$ is defined via its real symbol
\[P(D)f(x)= \int_{\mathbb{R}^{n}}{e^{ix \cdot \xi}P(\xi)\hat{f}(\xi)d\xi},\]
where $\hat{f}(\xi)$ denotes the Fourier transform of $f$.

The solution to the generalized Schr\"{o}dinger equation
\begin{equation}\label{Eq1}
\begin{cases}
 \partial_{t}u(x,t)-iP(D)u(x,t) =0 \:\:\:\ x \in \mathbb{R}^{n}, t \in \mathbb{R}^{+},\\
u(x,0)=f \\
\end{cases}
\end{equation}
can be formally written as
\begin{equation}
e^{itP(D)}f(x):= \int_{\mathbb{R}^{n}}{e^{ix \cdot \xi +itP(\xi)} \hat{f} (\xi)d\xi }.
\end{equation}
For example, when $P(\xi) = |\xi|^{2}$, the operator $e^{itP(D)}$ is the celebrated Schr\"{o}dinger operator $e^{it\Delta}$.  The related pointwise convergence problem is to determine the optimal $s$ for which
\begin{equation}
\mathop{lim}_{t \rightarrow 0^{+}} e^{itP(D)}f(x) = f(x)
\end{equation}
almost everywhere whenever $f \in H^{s}(\mathbb{R}^{n})$. For various $P(\xi)$, many experts have made many valuable contributions on the development of this subject, see the related articles \cite{B1, B2, B}, \cite{C}, \cite{CK}, \cite{DK,DG,DN,DL,DGL,DKWZ,DZ,Guth1,Guth2}, \cite{L,LR,LW,LWY,LR1,LR2,LR3,Miao,MVV}, \cite{RVV}, \cite{SS,S,S3,S2,Trebels,SPS,TV,V,W,WS,Zh}.

There are various interesting generalizations of the pointwise  convergence problem. The readers can see \cite{LW} and references therein.  One of such generalizations is to consider convergence properties for generalized Schr\"{o}dinger operators along curves $(\gamma(x,t) ,t)$ instead of the vertical line $(x,t)$. Here $\gamma(x,t)$ maps $\mathbb{R}^{n} \times \mathbb{R}$ to $\mathbb{R}^{n}$, $\gamma(x,0) =x$.   When $\gamma(x,t)$ is a $C^{1}$ function in $t$, the convergence properties are very similar with that in the vertical case, see \cite{LR}. However, much less is known when $\gamma(x,t)$ is just  $\alpha$-H\"{o}lder continuous in $t$, $0< \alpha <1$. Such  curves ($\gamma(x, t), t$) are called tangential curves since as $t \rightarrow 0$, $(\gamma(x,t),t)$ approaches $(x,0)$ tangentially to the hyperplane $\{(y,t)\in \mathbb{R}^n \times \mathbb{R}: t=0\}$. One can see figure 3 in Subsection \ref{B} below for the geometry.

In this paper, we mainly consider four kinds of convergence properties for (generalized) Schr\"{o}dinger operators along tangential curves:

 \textbf{(A)} a.e. convergence rate for generalized Schr\"{o}dinger operators along tangential  curves in $\mathbb{R}^{n} \times \mathbb{R}$, $n \ge 1$;

\textbf{(B)} a.e. pointwise convergence for Schr\"{o}dinger operator along tangential  curves in $\mathbb{R}^{2} \times \mathbb{R}$;

\textbf{(C)} a.e. convergence for Schr\"{o}dinger operator along a family of restricted tangential curves in $\mathbb{R} \times \mathbb{R}$;

\textbf{(D)} sharp  $L^p$-Schr\"{o}dinger maximal estimates along tangential curves in  $\mathbb{R}\times \mathbb{R}$.

In the rest of this introduction, we will introduce (A), (B), (C), (D) in Subsection \ref{A}, Subsection \ref{B}, Subsection \ref{C}, Subsection \ref{D}, respectively. Our results obtained in Subsection \ref{C} and  Subsection \ref{D} can be extended to more general operators, such as elliptic operators and fractional operators. It will appear in our subsequent articles.

\subsection{Convergence rate for  generalized Schr\"{o}dinger operators and applications}\label{A}

The problems on a.e. convergence rate of important operators (such as Fourier multipliers, certain integral means and summability means of Fourier integrals) were investigated in a lot of works  \cite{Carbery, MW, Trebels, CFW, LW} etc. In \cite{LW}, the authors studied the relationship between smoothness of the functions and the convergence rate for generalized Schr\"{o}dinger operators with polynomial growth along the curves in $\mathbb{R}^{n} \times \mathbb{R}$.

Denote by $B(x_{0},r)$ the ball with center $x_{0} \in \mathbb{R}^{n}$ and radius $r \lesssim 1$. Suppose that $\gamma(x,t)$ satisfies
\begin{equation}\label{Holder}
|\gamma(x,t)-\gamma(x,t^{\prime})| \le C|t-t^{\prime}|^{\alpha}, \:\ 0<\alpha \le 1
\end{equation}
uniformly for $x \in B(x_{0},r)$ and $t,t^{\prime} \in [0,1]$, $\gamma(x,0)=x$.
\begin{theorem}\cite{LW}\label{theorem1.1}
If there exist real numbers $m \ge 1$ and $s_{0} \ge 0$ such that
\begin{equation}\label{phase}
|P(\xi)| \lesssim |\xi|^{m},\hspace{0.2cm} |\xi| \rightarrow +\infty,
\end{equation}
and for each $s > s_{0}$,
\begin{equation}
\biggl\|\mathop{sup}_{0<t<1} |e^{itP(D)}(f)(\gamma(x,t))|\biggl\|_{L^{p}(B(x_{0},r))} \lesssim \|f\|_{H^{s}(\mathbb{R}^{n})},  \:\ p \ge 1,
\end{equation}
then for all $f \in H^{s+\delta}(\mathbb{R}^{n})$, $0 \le \delta <m$,
\begin{equation}\label{classicalresult}
e^{itP(D)}(f)(\gamma(x,t)) - f(x) = o(t^{\alpha \delta /m}), \:\ a.e. \:\:\ x \in B(x_{0},r) \text{\quad as \quad} t \rightarrow 0^{+}.
\end{equation}
\end{theorem}

The notation "$o$" means infinitesimal of high order.

Li and Wang \cite{LW} showed that  in the case of vertical line ($\alpha=1$), $t^{\delta/m}$ in the inequality (\ref{classicalresult}) can not be replaced by $t^{\delta/m^{\prime}}$ for some  $0<m^{\prime} < m$. However, for general case, there is no counterexample in \cite{LW} to show $t^{\alpha\delta/m}$ in the inequality (\ref{classicalresult}) is optimal.

Moreover, from the red line in Figure 1 and Figure 2, we observed that when smoothness $\alpha$ of the curve is fixed, the convergence rate will become close to $t^{\alpha}$ as long as smoothness $\delta$ of the function $f$ tends to $m$.  The problem is whether it can also guarantee the convergence rate will be close to  $t^{\alpha}$, if smoothness $\delta$ of the function is far away from $m$. We will  give a confirmed answer in this article.

\begin{theorem}\label{theorem1.1n}
If there exist real numbers $m \ge 1$ and $s_{0} \ge 0$ such that $P(\xi)$ satisfies (\ref{phase})
and for each $s > s_{0}$,
\begin{equation}\label{Eq1.8n}
\biggl\|\mathop{sup}_{0<t<1} |e^{itP(D)}(f)(\gamma(x,t))|\biggl\|_{L^{p}(B(x_{0},r))} \lesssim \|f\|_{H^{s}(\mathbb{R}^{n})},  \:\ p \ge 1.
\end{equation}
Then we have,

(1) when $1/m \le \alpha <1$, for each $s >s_{0}$, $f \in H^{s+\delta}(\mathbb{R}^{n})$, $0 \le \delta < \alpha m$, it holds
\begin{equation}\label{Eq1.9n}
\biggl\|\mathop{sup}_{0<t<1} \frac{|e^{itP(D)}(f)(\gamma(x,t))-f(x)|}{t^{\delta/m}}\biggl\|_{L^{q}(B(x_{0},r))} \lesssim \|f\|_{H^{s+ \delta}(\mathbb{R}^{n})}, \quad \quad q=\min\{p,2\};
\end{equation}

(2) when $ 0< \alpha <1/m$, for each $s >s_{0}$, $f \in H^{s+\delta}(\mathbb{R}^{n})$, $0 \le \delta < 1$, it holds
\begin{equation}\label{Eq1.91n}
\biggl\|\mathop{sup}_{0<t<1} \frac{|e^{itP(D)}(f)(\gamma(x,t))-f(x)|}{t^{\alpha \delta}}\biggl\|_{L^{q}(B(x_{0},r))} \lesssim \|f\|_{H^{s+ \delta}(\mathbb{R}^{n})}, \quad \quad q=\min\{p,2\}.
\end{equation}
\end{theorem}

By standard arguments, we obtain the following convergence rate result.

\begin{theorem}\label{theorem1.2n}
Under the assumption of Theorem \ref{theorem1.1n}, we have

(1) if $1/m \le \alpha <1$, then for all $f \in H^{s+\delta}(\mathbb{R}^{n})$,
\begin{equation}\label{Eq1.92n}
e^{itP(D)}(f)(\gamma(x,t)) - f(x) = o(t^{h}), \:\ a.e. \:\:\ x \in B(x_{0},r) \text{\quad as \quad} t \rightarrow 0^{+},
\end{equation}
whenever $( \delta, h) \in D_{1}:= \{(x,y): x \ge0, y \ge 0, y \le x/m, y< \alpha \}$;

(2) if $0 < \alpha < 1/m$, then for all $f \in H^{s+\delta}(\mathbb{R}^{n})$,
\begin{equation}\label{Eq1.93n}
e^{itP(D)}(f)(\gamma(x,t)) - f(x) = o(t^{h}), \:\ a.e. \:\:\ x \in B(x_{0},r) \text{\quad as \quad} t \rightarrow 0^{+},
\end{equation}
whenever $( \delta, h) \in D_{2}:= \{(x,y): x \ge0, y \ge 0, y \le \alpha x, y< \alpha \}$.
\end{theorem}

In order to compare the results of Theorem \ref{theorem1.1} (Red line) and Theorem  \ref{theorem1.2n} (Green line), we will give two figures as follows.
\begin{center}
\includegraphics[height=5cm]{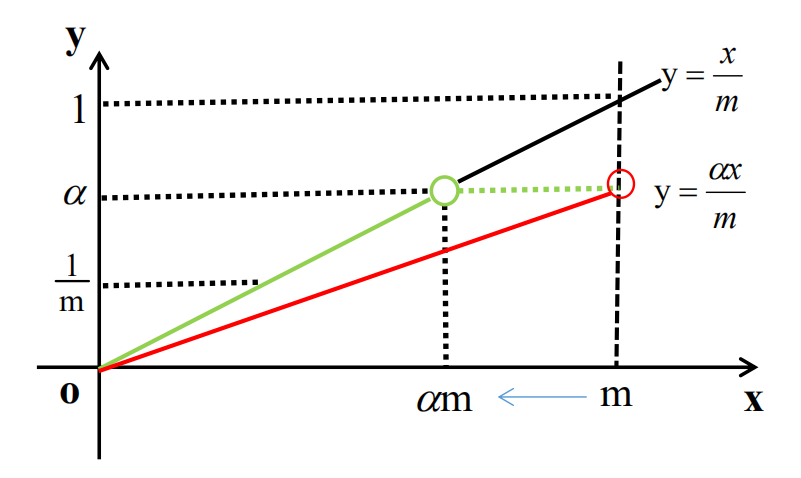}
\end{center}
\begin{center}
Figure 1. $1/m\leq \alpha<1$.

$x$ means  smoothness $\delta$ of the function $f$, $y$ means the convergence rate
\end{center}

\begin{center}
\includegraphics[height=5cm]{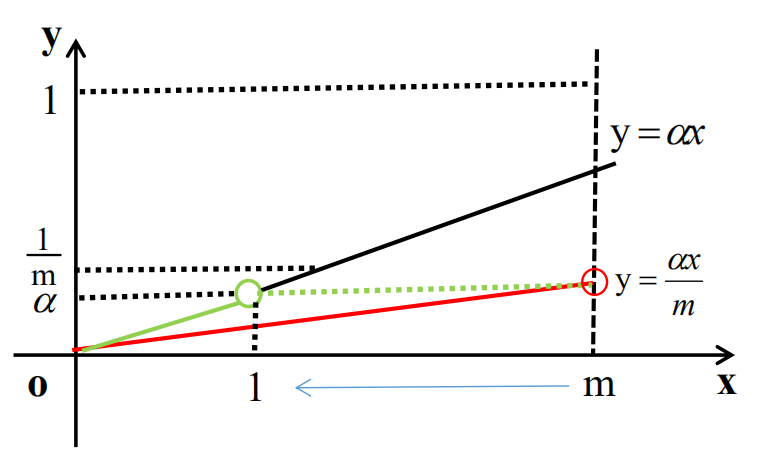}
\end{center}
\begin{center}
Figure 2. $0 < \alpha<1/m$.

$x$ means  smoothness $\delta$ of the function $f$, $y$ means the convergence rate
\end{center}

We  notice that it is difficult for us to extend  the regions  obtained in Theorem \ref{theorem1.2n} under such a general assumption on $\gamma$,  $P$ and $s_{0}$.
On one hand, for non-zero Schwartz functions, the convergence rate seems no faster than $t^{\alpha}$ as $t$ tends to $0$ along some curves $(\gamma(x,t),t)$, where functions $\gamma(x,t)$ are chosen as in Theorem \ref{theorem2.1n} below.
\begin{theorem} \label{theorem2.1n}\cite{LW}
 There exists
\[\gamma(x,t) = x - e_{1}t^{\alpha}, \:\ e_{1} = (1, 0,...,0),\]
such that for each Schwartz function $f$ and $0 < \alpha < 1$, if
\begin{equation}\label{Eq2.19n}
\mathop{lim}_{t \rightarrow 0^{+}}{\frac{e^{itP(D)}(f)(\gamma(x,t))-f(x)}{t^{\alpha}}}=0, \:\ a.e. \:\ x \in \mathbb{R}^{n},
\end{equation}
then $f\equiv 0,$  where $P(\xi)$ satisfies (\ref{phase}).
\end{theorem}

On the other hand, a stronger version of Theorem \ref{theorem1.1n} can be obtained from its proof.

\begin{theorem}\label{thsharp}
For any $\varepsilon >0$, if a  function $f$ defined on $\mathbb{R}^{n}$ with supp $\hat{f}
\subset \{\xi: |\xi| \sim R\}$ for some $R \gg 1$ satisfies
\[\biggl\|\mathop{sup}_{0<t<1} |e^{itP(D)}f(\gamma(x,t))|\biggl\|_{L^{p}(B(x_{0},r))} \le C_{\varepsilon} R^{\varepsilon }\|f\|_{L^{2}(\mathbb{R}^{n})},  \:\ p \ge 1,\]
where $C_{\varepsilon} >0$, $P(\xi)$ satisfies (\ref{phase}) for some $m \ge 1$, then we have

(1) when $1/m \le \alpha <1$,
\begin{equation}\label{sharp1}
\biggl\|\mathop{sup}_{0<t<1} \frac{|e^{itP(D)}(f)(\gamma(x,t))-f(x)|}{t^{\delta/m}}\biggl\|_{L^{q}(B(x_{0},r))} \le C_{\varepsilon} R^{\varepsilon +\delta}\|f\|_{L^{2}(\mathbb{R}^{n})}, \:\ q=\min\{p,2\};
\end{equation}

(2) when $0 < \alpha <1/m$,
\begin{equation}\label{sharp2}
\biggl\|\mathop{sup}_{0<t<1} \frac{|e^{itP(D)}(f)(\gamma(x,t))-f(x)|}{t^{\alpha \delta}}\biggl\|_{L^{q}(B(x_{0},r))} \le C_{\varepsilon} R^{\varepsilon +\delta}\|f\|_{L^{2}(\mathbb{R}^{n})}, \:\ q=\min\{p,2\}.
\end{equation}
\end{theorem}

Under the conditions of Theorem \ref{thsharp}, $t^{\delta/m}$ on the left hand side of inequality (\ref{sharp1}) can not be replaced by $t^{\delta/m^{\prime}}$ for some  $0<m^{\prime} < m$. Also,  $t^{\alpha \delta}$ on the left hand side of inequality (\ref{sharp2}) can not be replaced by $t^{ \alpha^{\prime} \delta}$ for some  $\alpha^{\prime}>\alpha $. Two counterexamples will be given in Subsection \ref{necessary} below.

Here we will give some applications. Let's first recall some results  on convergence problem for (fractional) Schr\"{o}dinger operators along tangential curves.

In $\mathbb{R} \times \mathbb{R}$,  \cite{CLV} established convergence result for Schr\"{o}dinger operator along the curve $(\gamma(x,t),t)$, where $\gamma(x,t)$ is given in Theorem \ref{theorem CLV} below. The authors obtained the following maximal  estimate from which the pointwise convergence result follows.
 \begin{theorem}\label{theorem CLV}\cite{CLV}
Let $n=1$ and $0<\alpha\leq 1$. Denote by $B_r(x_0)$ the interval which has center at $x_{0}$ with length $2r$, and  by $I_T(t_0)$ the interval which has center at $t_{0}$ with length $2T$. Suppose that a function $\gamma$ satisfies H\"{o}lder condition of order $\alpha$, $0<\alpha\leq 1$ in $t$
\begin{equation*}
|\gamma(x,t)-\gamma(x,t')|\leq C|t-t'|^{\alpha},
\end{equation*}
and is bilipschitz in $x$
\begin{equation*}
C_1|x-y|\leq |\gamma(x,t)-\gamma(y,t)|\leq C_2|x-y|,
\end{equation*}
for arbitrary $x$, $y \in B_r(x_0)$ and $t$, $t'\in I_T(t_0)$, $\gamma(x,0)=x$. Then
\begin{equation*}
\biggl\|\sup_{t\in I_T(t_0)}|e^{it\triangle}f(\gamma(x,t))|\biggl\|_{L^2(B_r(x_0))}\leq C\|f\|_{H^s(\mathbb{R})},
\end{equation*}
if $s>\max\{1/2-\alpha, 1/4\}$.
\end{theorem}

 Based on the result from Theorem \ref{theorem CLV}, we will give an application for Theorem \ref{theorem1.2n}.

\begin{theorem}
Under the assumption of Theorem \ref{theorem CLV}, we have,

(1) if $1/2 \le \alpha <1$, then for each $s > \frac{1}{4}$ and all $f \in H^{s+\delta}(\mathbb{R})$,
\begin{equation}\label{Eq1.94n}
e^{it\Delta}(f)(\gamma(x,t)) - f(x) = o(t^{h}), \:\ a.e. \:\:\ x \in B(x_{0},r) \text{\quad as \quad} t \rightarrow t_{0}^{+},
\end{equation}
whenever $( \delta, h) \in D_{1}:= \{(x,y): x \ge0, y \ge 0, y \le x/2, y< \alpha \}$.

(2) if $0 < \alpha < 1/2$, then for each $s > \max\{\frac{1}{2}-\alpha,\frac{1}{4}\}$ and all $f \in H^{s+\delta}(\mathbb{R})$,
\begin{equation}\label{Eq1.95n}
e^{it\Delta}(f)(\gamma(x,t)) - f(x) = o(t^{h}), \:\ a.e. \:\:\ x \in B(x_{0},r) \text{\quad as \quad} t \rightarrow 0^{+},
\end{equation}
whenever $( \delta, h) \in D_{2}:= \{(x,y): x \ge0, y \ge 0, y \le \alpha x, y< \alpha \}$.
\end{theorem}

In \cite{CLV}, the authors adopted the $TT^*$ method, the time localizing lemma and the van der Corput's lemma to establish Theorem \ref{theorem CLV}. Cho-Lee extended this result in \cite{CL} where they obtained the capacity dimension of the divergence set. Recently, geting around of using the time localizing lemma, the corresponding result was obtained for fractional Schr\"{o}dinger operators by Cho-Shiraki \cite{CS}. Combing with Theorem \ref{theorem1.2n}, we can get the convergence rate results. But we omit the results here.

\subsection{Convergence results along tangential curves in $\mathbb{R}^{2} \times \mathbb{R}$}\label{B}
Comparing with the  case in $\mathbb{R} \times \mathbb{R}$, much less is known about the convergence problem for Schr\"{o}dinger operator along tangential curves in higher dimensional case $\mathbb{R}^{n} \times \mathbb{R}, n \ge2$. It follows from  \cite[Proposition 4.3]{CLV} that if $P(\xi)$ satisfies
 \[|D_{\xi}^{\beta}P(\xi)| \lesssim |\xi|^{m-|\beta|-1}, \quad \quad |P(\xi)| \sim |\xi|^{m-1} \]
 for $|\xi| \gg 1$ and $m \ge 2$, $\gamma(x,t)$ is bi-lipschitz in $x$ and satisfies H\"{o}lder condition of order $\frac{1}{m-1}$ in $t$. Then the convergence of  $e^{itP(D)}f$ along the curve $(\gamma(x,t),t)$ follows from the convergence of $e^{itP(D)}f$ along the vertical line $(x,t)$. It is clear that some convergence result along tangential curves for generalized Schr\"{o}dinger operators can be obtained when $m >2$.

For $m=2$ and $n=1$, Cho-Lee-Vargas \cite{CLV}  showed the pointwise convergence  along the curve $(\gamma(x,t), t)$ as required by Theorem \ref{theorem CLV}.
 Ding and Niu \cite{DN} improved the above theorem, i.e. $f\in H^s(\mathbb{R})$ for $s\geq 1/4$, if $1/2\leq \alpha\leq 1$.

However, the convergence problem  along tangential curve  in  $\mathbb{R}^{n} \times \mathbb{R}$ ($n \ge2$) was open until now  when $m=2$.

It is now well-known that the method of restriction estimates for the paraboloid adopted by the references \cite{Guth1, Guth2} can be applied to get the sharp convergence result for $e^{it\Delta}f$ along the vertical line $(x,t)$, one can see the articles \cite{DGL} for $n=2$ and \cite{DZ} for $n \ge 3$. So it is interesting to seek if these methods can be applied to get convergence result for $e^{it\Delta}f$ along tangential curves. In this Subsection, we give a convergence result for Schr\"{o}dinger operator along tangential curves in  $\mathbb{R}^2 \times \mathbb{R}$ by the broad-narrow argument and polynomial partitioning. Moreover, we obtain the corresponding convergence rate.

Let $\Gamma_{\alpha}:=\{\gamma:[0,1]\rightarrow\mathbb{R}^{2}: \text{for each } t,t^{\prime} \in [0, 1], | \gamma(t)-\gamma(t^{\prime})| \le C_{\alpha}|t-t^{\prime}|^{\alpha}\}, C_{\alpha} \ge 1.$
We consider the convergence problem of the Schr\"{o}dinger operator along the curves $(x+\gamma(t), t)$ with $\gamma \in \Gamma_{\alpha}$ for some $\alpha \in [1/2,1)$. For example, $\gamma(t)=t^{\alpha}\mu, \alpha \in [1/2,1)$ and $\mu$ is a bounded vector in $\mathbb{R}^2$.

We have the maximal estimate below.
\begin{theorem}\label{maximal theorem along tangential}
Let $p=3.2$. For all $\varepsilon >0$ and $f \in H^{3/8+\varepsilon}(\mathbb{R}^2)$, it holds
\begin{equation}\label{maximal estimate along tangential}
  \biggl\|\sup_{t\in (0,1)}|e^{it\triangle}f(x+\gamma(t))|\biggl\|_{L^p(B(0,1))}\leq C \|f\|_{H^s(\mathbb{R}^{2})}.
\end{equation}
Here, the constant $C$ depends only on $\varepsilon$ and $C_{\alpha}$, but does not depend on the choice of $\gamma$.
\end{theorem}

\begin{center}
\includegraphics[height=7cm]{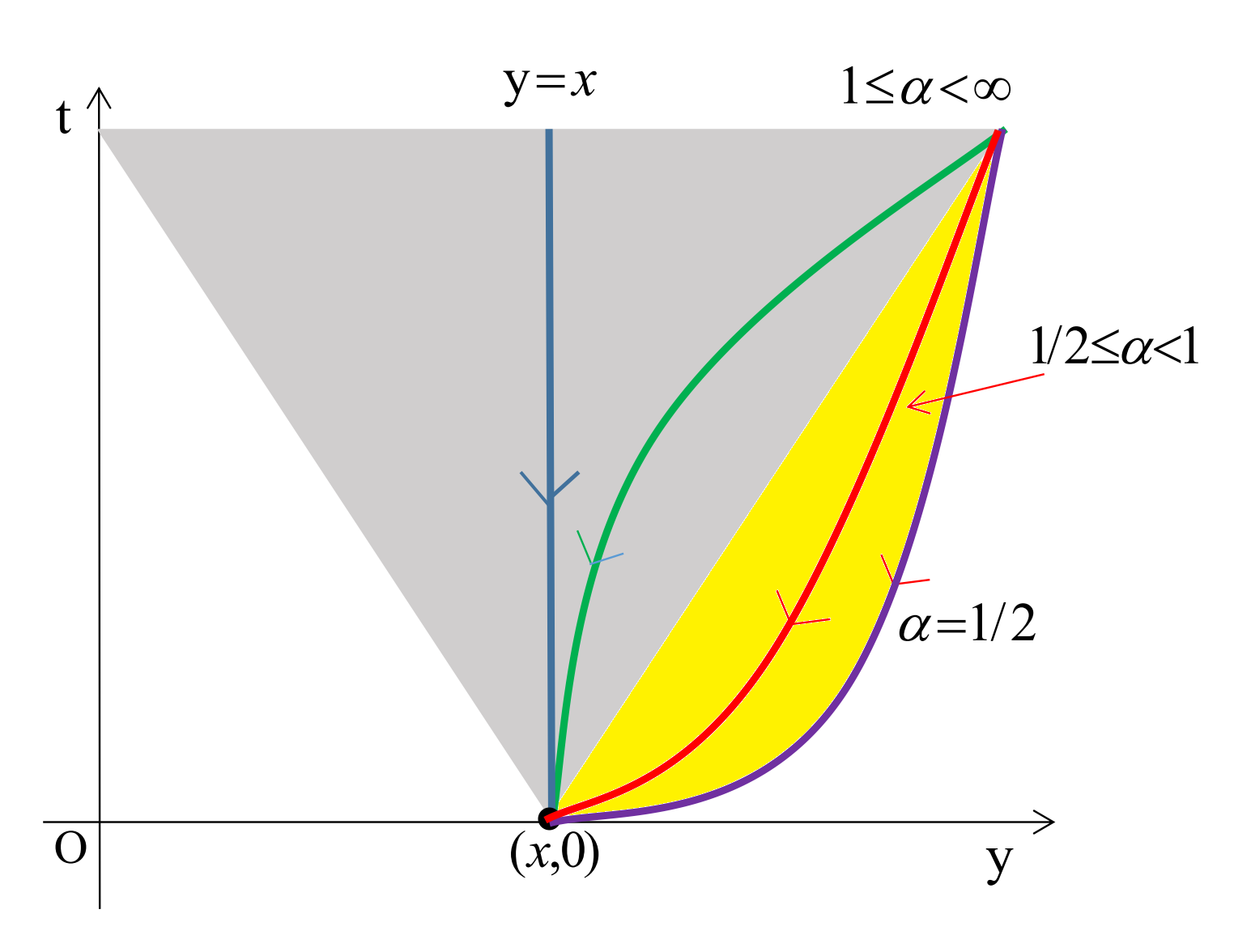}
\end{center}
\begin{center}
Figure 3. The blue line is convergence result along vertical lines from \cite{DGL}; The green line is convergence result along $C^1$-curves from \cite{LR}; The red line is convergence result along tangential curves from Theorem \ref{maximal theorem along tangential} above.
\end{center}

By Theorem \ref{maximal theorem along tangential}, if $\gamma \in \Gamma_{\alpha}$  and $\gamma(0)=(0,0)$, then the convergence result
 \begin{equation}\label{Eq13}
\mathop {\lim }\limits_{t \to 0^{+}} {e^{it\Delta }}f\left( {x + \gamma \left( {t} \right)} \right) = f(x) \:\ a.e. \:\ x \in \mathbb{R}^2
 \end{equation}
holds whenever $f \in H^{s}(\mathbb{R}^2)$, $s > 3/8.$ Hence, the corresponding convergence rate follows from Theorem \ref{maximal theorem along tangential} and Theorem \ref{theorem1.2n}.

\begin{theorem}
For each $s > \frac{3}{8}$ and all $f \in H^{s+\delta}(\mathbb{R}^{2})$. If $\gamma \in \Gamma_{\alpha}$, $\gamma(0) =(0,0)$, then
\begin{equation}\label{Eq1.96n}
e^{it\Delta}f(x + \gamma(t)) - f(x) = o(t^{h}), \:\ a.e. \:\:\ x \in \mathbb{R}^{2} \text{\quad as \quad} t \rightarrow 0^{+},
\end{equation}
whenever $( \delta, h) \in D_{1}:= \{(x,y): x \ge0, y \ge 0, y \le x/2, y< \alpha \}$.
\end{theorem}

Before the proof of Theorem \ref{maximal theorem along tangential}, we give the following remark. If Theorem \ref{maximal theorem along tangential} holds true, then for each $(x_{0}, t_{0} ) \in \mathbb{R}^{2}\times \mathbb{R}$,   and $\gamma \text{  which maps } [t_{0}, t_{0} +1]$ to $\mathbb{R}^{2}$ such that for each $t,t^{\prime} \in [t_{0}, t_{0}+1]$,
\[|\gamma(t) - \gamma(t^{\prime})| \le C_{\alpha} |t-t^{\prime}|^{\alpha},\]
  it holds
\begin{equation}\label{Eq1.21}
 \biggl\|\sup_{t\in (t_{0},t_{0}+1)}|e^{it\triangle}f(x+\gamma(t))|\biggl\|_{L^p(B(x_{0},1))}\leq C \|f\|_{H^s(\mathbb{R}^{2})},\hspace{0.2cm}s>3/8,
\end{equation}
where the constant $C$ depends only on $\varepsilon$ and $C_{\alpha}$.
Indeed,  changing variables implies that
\[\biggl\|\sup_{t\in (t_{0},t_{0}+1)}|e^{it\triangle}f(x+\gamma(t))|\biggl\|_{L^p(B(x_{0},1))} = \biggl\|\sup_{t\in (0,1)}|e^{it\triangle}g(x+\gamma(t+t_{0}))|\biggl\|_{L^p(B(0,1))}, \]
in which $\hat{g}(\xi) = e^{ix_{0} \cdot \xi + it_{0}|\xi|^{2}}\hat{f}(\xi)$. It is obvious that $\gamma(\cdot + t_{0}) \in \Gamma_{\alpha}$, then by Theorem \ref{maximal theorem along tangential}, we have
\[ \biggl\|\sup_{t\in (0,1)}|e^{it\triangle}g(x+\gamma(t+t_{0}))|\biggl\|_{L^p(B(0,1))} \le C \|g\|_{H^s(\mathbb{R}^{2})}, \hspace{0.2cm}s>3/8. \]
Here the constant $C$ depends only on $\varepsilon$ and $C_{\alpha}$, but does not depend on the choice of $\gamma$. Since $\|f\|_{H^s} = \|g\|_{H^s}$,
inequality (\ref{Eq1.21}) follows. Moreover,  it follows from inequality (\ref{Eq1.21}) that
\[\mathop {\lim }\limits_{t \to t_{0}^{+}} {e^{it\Delta }}f\left( {x + \gamma \left( {t} \right)} \right) = e^{it_{0}\Delta}f(x + \gamma(t_{0})) \:\ a.e. \:\ x\in \mathbb{R}^2.\]

Now let's turn to reduce the proof of Theorem \ref{maximal theorem along tangential}. By Littlewood-Paley decomposition,  it suffices to show that
\begin{equation}\label{annulas maximal estimate along tangential}
 \biggl\|\sup_{t\in (0,1)}|e^{it\triangle}f(x+\gamma(t))|\biggl\|_{L^p(B(0,1))}\leq C R^{3/8+
\varepsilon}\|f\|_{L^2},
\end{equation}
whenever supp $\hat{f} \subset \{{\xi \in \mathbb{R}^{2}: |\xi| \sim R}\}$, $R \gg 1$, the constant $C$ depends on $\varepsilon$ and $C_{\alpha}$, but does not depend on the choice of $\gamma$. For each $\gamma \in \Gamma_{\alpha}$, according to the time localizing lemma, the global estimate
\begin{equation}\label{global estimate}
 \biggl\|\sup_{t\in (0,1)}|e^{it\triangle}f(x+\gamma(t))|\biggl\|_{L^p(B(0,1))}\leq C^{\prime} R^{3/8+ \varepsilon}\|f\|_{L^2}
 \end{equation}
follows from the local estimate on each interval $(t_{0},t_{0}+R^{-1}) \subset (0,1)$,
 \begin{equation}\label{local estimate}
 \biggl\|\sup_{t\in (t_{0},t_{0}+R^{-1})}|e^{it\triangle}f(x+\gamma(t))|\biggl\|_{L^p(B(0,1))}\leq C^{\prime \prime} R^{3/8+ \varepsilon}\|f\|_{L^2}.
 \end{equation}
Notice that the constant $C^{\prime}$ in the global estimate (\ref{global estimate}) depends on $C_{\alpha}$ and the constant $C^{\prime \prime}$ in the local estimate (\ref{local estimate}). Therefore, in order to show inequality (\ref{annulas maximal estimate along tangential}), we just need to  prove the theorem below.
\begin{theorem}\label{local in time estimate}
Let  $p=3.2$, $R \gg 1$ and
\[\Gamma_{\alpha,R^{-1}}:=\{\gamma:[0,R^{-1}]\rightarrow\mathbb{R}^{2}: \text{for each } t,t^{\prime} \in [0, R^{-1}], | \gamma(t)-\gamma(t^{\prime})| \le C_{\alpha}|t-t^{\prime}|^{\alpha}\}.\]
For any $\varepsilon >0$, we have
\begin{equation}\label{local annulas maximal estimate along tangential curves}
 \biggl\|\sup_{t\in (0,R^{-1})}|e^{it\triangle}f(x+\gamma(t))|\biggl\|_{L^p(B(0,1))}\leq C R^{3/8+
\varepsilon}\|f\|_{L^2},
\end{equation}
for all $f$ with supp $\hat{f} \subset \{{\xi \in \mathbb{R}^{2}: |\xi| \sim R}\}$, where the constant $C$ depends only on $\varepsilon$ and $C_{\alpha}$.
\end{theorem}

If Theorem \ref{local in time estimate} holds true, then for any interval $(t_{0},t_{0}+R^{-1}) \subset (0,1)$ and arbitrary $\gamma \in \Gamma_{\alpha}$,
\[ \biggl\|\sup_{t\in (t_{0},t_{0}+R^{-1})}|e^{it\triangle}f(x+\gamma(t))|\biggl\|_{L^p(B(0,1))} = \biggl\|\sup_{t\in (0, R^{-1})}|e^{it\triangle}g(x+\gamma(t+t_{0}))|\biggl\|_{L^p(B(0,1))},\]
where $\hat{g}(\xi) = e^{it_{0}|\xi|^{2}}\hat{f}(\xi)$. It is clear that $\gamma(t +t_{0}) \in \Gamma_{\alpha,R^{-1}}$ and Theorem \ref{local in time estimate} implies that
\[\biggl\|\sup_{t\in (0, R^{-1})}|e^{it\triangle}g(x+\gamma(t+t_{0}))|\biggl\|_{L^p(B(0,1))} \leq C R^{3/8+
\varepsilon}\|g\|_{L^2}, \]
where the constant $C$ depends only on $\varepsilon$ and $C_{\alpha}$. Then we have the local estimate (\ref{local estimate})   with the constant $C^{\prime \prime}$  depends only on $\varepsilon$ and $C_{\alpha}$, so does  the constant $C^{\prime}$ in  (\ref{global estimate}). Finally we arrive at inequality (\ref{annulas maximal estimate along tangential}).

By parabolic rescaling, Theorem \ref{local in time estimate} can be reduced to show that for each $\gamma \in \Gamma_{\alpha, R^{-1}}$,
\[ \biggl\|\sup_{t\in (0,R)}|e^{it\triangle}f(x+R\gamma (\frac{t}{R^{2}}))|\biggl\|_{L^p(B(0,R))}\leq C R^{2/p - 5/8 + \varepsilon} \|f\|_{L^2},\]
for all $f$ with supp $\hat{f} \subset B(0,1)$. But in order to apply the induction argument in the frequency space of $f$, we will prove the following theorem.

\begin{theorem}\label{reduction main theorem}
Let $p=3.2$. For arbitrary $\gamma \in \Gamma_{\alpha,R^{-1}}$, any $\varepsilon >0$, all balls $B(\xi_{0}, M^{-1}) \subset B(0,1) $ and $f$ with supp $\hat{f} \subset B(\xi_{0},M^{-1})$, it holds
\begin{equation}\label{reduction main estimate}
 \biggl\|\sup_{t\in (0,R)}|e^{it\triangle}f(x+R\gamma (\frac{t}{R^{2}}))|\biggl\|_{L^p(B(0,R))}\leq C M^{-\varepsilon^{2}}R^{2/p - 5/8 + \varepsilon} \|f\|_{L^2}.
\end{equation}
Here the constant $C$ depends only on $\varepsilon$ and $C_{\alpha}$, but does not depend on the choice of $\gamma$.
\end{theorem}

We notice that the following feature plays key role in the proof of Theorem \ref{reduction main theorem}: $\gamma \in \Gamma_{\alpha, R^{-1}} $ implies that $R\gamma(\frac{t}{R^{2}})$ is $\alpha$-H\"{o}lder continuous for $t \in [0,R]$ when $\alpha \in [1/2, 1)$. However this feature is no longer true for $\alpha \in (0, 1/2)$, then the method in this paper does not work. We leave the case $\alpha \in (0, 1/2)$ for further consideration.

We put the detailed proof of Theorem \ref{reduction main theorem} into Section \ref{proof for reduction main theorem}, and briefly discuss the method here.
We would like to prove Theorem \ref{reduction main theorem} by the broad-narrow argument and polynomial partitioning. Such methods work very well in the study of Fourier restriction operators,  see \cite{Guth1} and  \cite{Guth2} for instance. Du-Li  \cite{DL} firstly applied the method in \cite{Guth2} to study convergence problem for Schr\"{o}dinger operator along vertical lines. Further, with the help of decoupling methods, the result of \cite{DL} was improved to sharp convergence in \cite{DGL}. The main difference between our problem and the vertical case is that the support for the  Fourier transform of $e^{it\Delta}f(x+R\gamma(\frac{t}{R^{2}}))$ is not clear, since $\gamma(t)$ is not smooth. This leads to  the failure of many nice properties which are very important in the research of the works \cite{Guth1,DGL, Guth2,DL}. In order to overcome these difficulties, we prove Lemma \ref{main lemma} below. Lemma \ref{main lemma} works  as a substitution for the locally constant property. However,  we still do not know if the decoupling method can help to improve the result obtained by Theorem \ref{reduction main theorem}, also by Theorem \ref{maximal theorem along tangential}.

The convergence result obtained by Theorem \ref{maximal theorem along tangential} may not be sharp. In fact, by Lemma \ref{equivalence} below and Bourgain's counterexample in \cite{B2}, we can get the following necessary condition.
\begin{theorem}\label{necessary condition}
Let  $\gamma(x,t)= x - e_{1}t^{\alpha}$, $e_{1}= (1,0)$, $\alpha \in [1/2,1)$. Then
\begin{equation}\label{tangential global}
\biggl\|\sup_{0<t<1}|e^{it\triangle}f(\gamma(x,t))|\biggl\|_{L^1(B(0,1))} \leq C \|f\|_{H^s(\mathbb{R}^{2})}
\end{equation}
holds for all $f \in H^{s}(\mathbb{R}^2)$ only if $s \ge 1/3$.
\end{theorem}

It is clear that there is a gap between $1/3$ and $3/8$. The next theorem implies that if one wants to improve the convergence result obtained by Theorem \ref{maximal theorem along tangential}, then the range of $p$ should be chosen very carefully. More concretely, if $s \rightarrow 1/3^+$, then $p$ should not be larger than $3$.

\begin{theorem}\label{upper bound1}
Let  $\gamma(x,t)= x - e_{1}t^{\alpha}$, $e_{1}= (1,0)$, $\alpha \in [1/2,1)$. If $0< s < 1/2$ and
\begin{equation}\label{tangential large p}
\biggl\|\sup_{0<t<1}|e^{it\triangle}f(\gamma(x,t))|\biggl\|_{L^p(B(0,1))} \leq C \|f\|_{H^s(\mathbb{R}^{2})}
\end{equation}
holds for all $f \in H^{s}(\mathbb{R}^2)$, then $p \le \frac{1}{1-2s}$.
\end{theorem}

Proofs of Theorem \ref{necessary condition} and Theorem \ref{upper bound1} will appear in Section \ref{proof of necessary condition}.

\subsection{Convergence results along a family of restricted tangential curves in $\mathbb{R} \times \mathbb{R}$}\label{C}
In this section, we consider the  convergence problem along a class of restricted tangential curves in $\mathbb{R} \times \mathbb{R}$ given by $\{(y,t): y \in \Gamma_{x,t}\}$ for each $t\in [0,1]$, where
\[\Gamma_{x,t}=\{\gamma(x,t,\theta): \theta\in \Theta\}\]
  for a given compact set $\Theta$  in $\mathbb{R}$. $\gamma$ is a map from $\mathbb{R} \times [0,1] \times \Theta$ to $\mathbb{R}$,  which satisfies $\gamma(x,0,\theta) =x$ for all $x \in \mathbb{R}$, $\theta \in \Theta$, and the following conditions (C1)-(C3) hold:\\
\textbf{(C1)} for fixed $t \in [0,1]$, $\theta \in \Theta$, $\gamma$ has at least $C^{1}$-regularity in $x$, and there exists a constant $C_{1} \ge 1$
such that for each $x, x^{\prime}  \in \mathbb{R}$, $\theta \in \Theta$, $t \in [0,1]$,
\begin{equation}\label{Eq1.3}
 C_{1}^{-1}|x- x^{\prime}| \le |\gamma(x,t, \theta)-\gamma(x^{\prime},t, \theta)| \le C_{1}|x- x^{\prime}|;
\end{equation}
\textbf{(C2)} there exists  a constant $C_{2} >0$ and $\alpha \in (0,1)$ such that for each $x \in \mathbb{R}$, $\theta \in \Theta$, $t,t^{\prime} \in [0,1]$,
\begin{equation}\label{Eq1.4}
|\gamma(x,t, \theta)-\gamma(x,t^{\prime}, \theta)| \le C_{2}|t- t^{\prime}|^{\alpha};
\end{equation}
\textbf{(C3)} there exists a constant $C_{3} >0$ such that for each $x \in \mathbb{R}$, $t \in [0,1]$, $\theta, \theta^{\prime} \in \Theta$,
\begin{equation}\label{Eq1.5}
|\gamma(x,t, \theta)-\gamma(x,t, \theta^{\prime})| \le C_{3}|\theta- \theta^{\prime}|.
\end{equation}
Let's study the relationship between the dimension of $\Theta$ and the optimal $s$ for which
 \begin{equation}\label{nontangential convergence}
 \lim_{\substack{(y,t) \rightarrow (x,0) \\ y\in \Gamma_{x,t}}}e^{it\Delta}f(y) = f(x) \hspace{0.3cm} a.e. \hspace{0.1cm} x\in \mathbb{R}
 \end{equation}
whenever $f \in {H^s}\left( {{\mathbb{R}}} \right)$. In order to understand such convergence problem better, we give the exact example below, see Figure 4 and Figure 5. Here we introduce the so-called logarithmic density or upper Minkowski dimension of $\Theta$ to characterize its size, which is defined by
\[\beta(\Theta)= \limsup_{\delta \rightarrow 0^{+}} \frac{logN(\delta)}{-log\delta},\]
where $N(\delta)$ is the minimum number of closed balls of diameter $\delta$ to cover $\Theta$. Apparently, when $\Theta$ is a single point, $\beta(\Theta) =0$; when $\Theta$ is a compact subset of $\mathbb{R}^{n}$ with positive Lebesgue measure, $\beta(\Theta) =n$.

\begin{center}
\includegraphics[height=2cm]{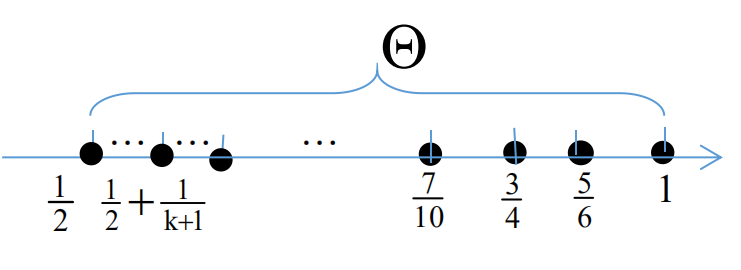}
\end{center}
\begin{center}
Figure 4.  $\Theta=\{1, \frac{5}{6}, \cdot \cdot \cdot, \frac{1}{2}+ \frac{1}{k+1}, \cdot \cdot \cdot,\frac{1}{2}  : k \in \mathbb{N}^{+}\}$.
\end{center}

\begin{center}
\includegraphics[height=8cm]{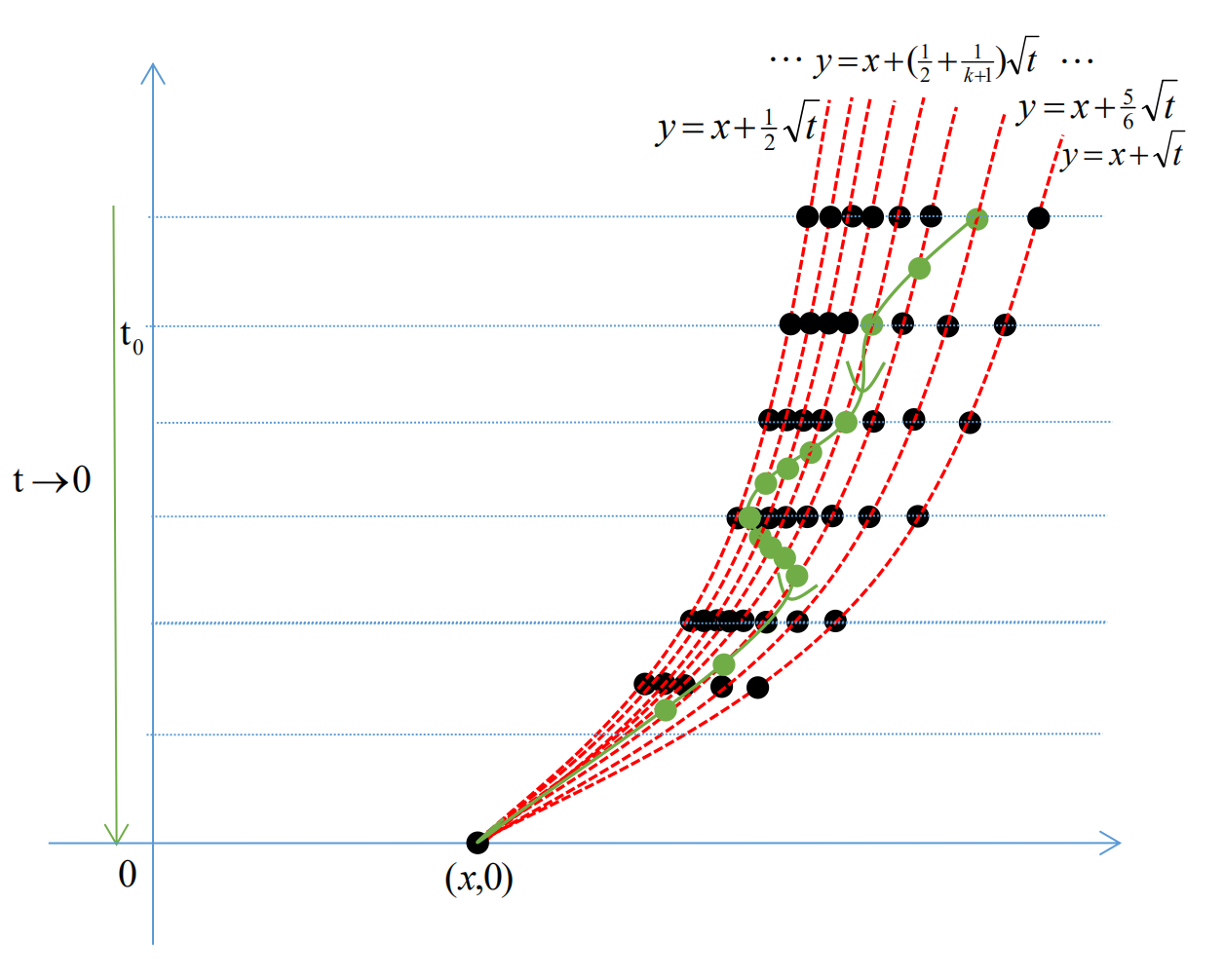}
\end{center}
\begin{center}
Figure 5. The convergence path (green curve) is determined by the set $\bigcup_{t\in [0,1/2]}\{(y,t):y\in \Gamma_{x,t}\}$ whose element consists of all black points where $\Gamma_{x,t}=\{x+\theta\sqrt{t}, \theta\in \Theta\}$.
\end{center}

In \cite{CLV}, this question is  considered  for a family of restricted straight lines  in $\mathbb{R }\times \mathbb{R}$. Exactly, for $t\in [-1,1]$, let $\Gamma_{x,t}=\{x+t\theta: $ and $\theta\in \Theta\}$, where $\Theta$ is a given compact set in $\mathbb{R}$. In \cite{CLV}, they proved that the corresponding non-tangential convergence result holds for $s>\frac{\beta(\Theta)+1}{4}$. Then Shiraki \cite{Shi} generalized this result to a wider class of equations which includes the fractional Schr\"{o}dinger equation. Very recently, Li-Wang-Yan \cite{LWY} obtained the corresponding non-tangential convergence result in any dimensions and extend the straight lines to more general curves with Lipschitz regularity in time variable. The problem is what will happen if the curves satisfy just $\alpha$-H\"{o}lder regularity ($0<\alpha<1$). Next, we will give an answer about it.


By standard argument, the convergence result follows from the maximal estimate below.
\begin{theorem}\label{nontangential maximal theorem}
When $\gamma$ satisfies (C1)-(C3), considering the maximal estimate
\begin{equation}\label{nontangential maximal estimate}
\biggl\|\sup_{t\in (0,1), \theta \in \Theta }|e^{it\triangle}f(\gamma(x,t, \theta))|\biggl\|_{L^p(B(x_0,r))}\leq C \|f\|_{H^s(\mathbb{R})}, \quad f \in H^{s}(\mathbb{R}),
\end{equation}
where $B(x_{0}, r) \subset \mathbb{R}$. We have,

(1) for   each $\alpha \in [1/2,1) $,  inequality (\ref{nontangential maximal estimate}) holds if  $s > s_{0}= \frac{ \beta (\Theta) +1}{4}$ and $p = 4$;

(2) for each  $\alpha \in (1/4,1/2) $,  inequality (\ref{nontangential maximal estimate}) holds if  $s >s_{0}= \frac{ \beta (\Theta) +1}{4}$ and $p = 8\alpha$;

(3) for each  $\alpha \in (0,1/4] $,  inequality (\ref{nontangential maximal estimate}) holds if  $s > s_{0}= \alpha \beta(\Theta) +\frac{1}{2}- \alpha$ and $p = 2$. \\
Moreover, we notice that the constant on the right hand side of inequality (\ref{nontangential maximal estimate}) depends only on $C_{1},C_{2}, C_{3}$, $\Theta$ and the choice of $B(x_{0},r)$, but does not depend on $f$.
\end{theorem}

Then we have the following convergence result for Schr\"{o}dinger  operator along a family of restricted tangential curves.
\begin{theorem}\label{nontangential convergence theorem}
Let $\gamma$ satisfy (C1)-(C3). The convergence result (\ref{nontangential convergence}) holds almost everywhere if

(1)   $\alpha \in [1/2,1) $, $f \in H^{s}(\mathbb{R})$ and $s > s_{0}= \frac{ \beta (\Theta) +1}{4}$;

(2)  $\alpha \in (1/4,1/2) $, $f \in H^{s}(\mathbb{R})$ and $s >s_{0}= \frac{ \beta (\Theta) +1}{4}$;

(3)   $\alpha \in (0,1/4] $, $f \in H^{s}(\mathbb{R})$ and $s > s_{0}= \alpha \beta(\Theta) +\frac{1}{2}- \alpha$.
\end{theorem}

Theorem \ref{nontangential convergence theorem} is sharp when $\beta(\Theta) =0 $ (see \cite{CLV} or Theorem \ref{theorem CLV} in this paper) and $\beta(\Theta) = 1$ (see \cite{SS}). We leave necessity of the case  $0< \beta(\Theta) <1$ for further consideration.

Here we briefly sketch the main idea for the proof of Theorem \ref{nontangential maximal theorem} and leave the details to Section
\ref{proof of nontangential}. By Littlewood-Paley decomposition, we only need to consider $f$ whose Fourier
transform is supported in $\{\xi \in \mathbb{R}: |\xi| \sim \lambda \}$, $\lambda \gg 1$. Next we decompose $\Theta$ into small subsets $\{\Theta_{k}\}$ such that $\Theta= \cup_{k}\Theta_{k}$ with bounded overlap, where each $\Theta_{k}$ is contained in a closed ball with  diameter $\lambda^{-\mu}$, $\mu >0$. Then the number of $\Theta_{k}$ is dominated by $\lambda^{\mu \beta(\Theta)  + \varepsilon }$, for any $\varepsilon >0$.
Theorem \ref{nontangential maximal theorem} is reduced to consider the
estimate
\begin{equation}\label{local nontangential estimate}
\biggl\|\mathop{\rm{sup}}_{t \in (0,1), \theta \in \Theta_{k}} |e^{it\Delta}f(\gamma(x,t,\theta))|\biggl\|_{L^{p}(B(x_{0},r))} \le C \lambda^{\nu} \|f\|_{L^{2}},
\end{equation}
where the constant $C$ does not depend on $k$. By \cite{CLV} or Theorem \ref{theorem CLV} in this paper,  $\nu$ can not be smaller than $ \max\{1/2 - \alpha, 1/4\}$.
Then we have to solve two problems:

\textbf{(P1)} what is the smallest possible value for $\mu$ such that (\ref{local nontangential estimate}) holds for $\nu= \max\{1/2 - \alpha, 1/4\}$;

\textbf{(P2)} what is the largest possible value for $p$ such that (\ref{local nontangential estimate}) holds for $\nu= \max\{1/2 - \alpha, 1/4\}$.
These two problems are solved by Lemma \ref{lemma2.3} below.  What's more, we will further discuss the problem \textbf{(P2)} in the next Subsection.
\subsection{Sharp $L^p-$Schr\"{o}dinger maximal estimates along tangential curves  in $\mathbb{R} \times \mathbb{R}$ }\label{D}
We notice that the  problem \textbf{(P2)} is of independent interest in the study of Schr\"{o}dinger maximal function.
For celebrated $L^p$-Schr\"{o}dinger maximal estimate, one would find the optimal $p$ and $s$ such that the maximal estimate holds. When spatial dimension $n=2$, Du-Guth-Li \cite{DGL} proved the sharp $L^p$-estimates for all $p\le 3$ and $s>1/3$. When spatial dimension $n\geq 3$, Du-Zhang \cite{DZ} proved the sharp $L^2$-estimate with $s>n/2(n+1)$, but  the sharp $L^p$-estimate of Schr\"{o}dinger maximal function is still unknown for $p >2$. Partial results on this problem are obtained by using polynomial partitioning and refined Strichartz estimates  in \cite{CMM,DKWZ,WS}.

As a consequence of  Theorem \ref{nontangential maximal theorem}, we achieve the sharp $L^p$-Schr\"{o}dinger maximal estimates along tangential curves in $\mathbb{R} \times \mathbb{R}$. In fact, we take $\Theta$ to be the set only consisting of a single point $\theta_0$, and rewrite the conditions (C1)-(C3) as follows. Here we abuse the notation a bit and replace $\gamma(x,t,\theta_0)$ by $\gamma(x,t)$. Let $\gamma$ be a map from $\mathbb{R} \times [0,1]$ to $\mathbb{R}$,  which satisfies $\gamma(x,0) =x$ for all $x \in \mathbb{R}$ and the following conditions $(C1)^{\prime}$-$(C2)^{\prime}$ hold:\\
$\boldsymbol{(C1)^{\prime}}$ for fixed $t \in [0,1]$,  $\gamma$ has at least $C^{1}$-regularity in $x$, and there exists a constant $C_{1} \ge 1$
such that for each $x, x^{\prime}  \in \mathbb{R}$, $t \in [0,1]$,
\begin{equation}\label{Eq1.3}
 C_{1}^{-1}|x- x^{\prime}| \le |\gamma(x,t)-\gamma(x^{\prime},t)| \le C_{1}|x- x^{\prime}|;
\end{equation}
$\boldsymbol{(C2)^{\prime}}$ there exists  a constant $C_{2} >0$ and $\alpha \in (0,1)$ such that for each $x \in \mathbb{R}$, $t,t^{\prime} \in [0,1]$,
\begin{equation}\label{Eq1.4}
|\gamma(x,t)-\gamma(x,t^{\prime})| \le C_{2}|t- t^{\prime}|^{\alpha}.
\end{equation}

Specially, Sj\"{o}lin \cite{S3} has studied  $L^p$-Schr\"{o}dinger maximal estimates for the case $\gamma(x,t)=x$.

 \begin{theorem}\label{upper bound}
Let $n=1$ and $ 0 <\alpha < 1$.  Suppose that a function $\gamma$ satisfies conditions $(C1)^{\prime}$, $(C2)^{\prime}$
for arbitrary $x$, $y \in B(x_0,r) \subset \mathbb{R}$ and $t$, $t'\in [0,1]$, $\gamma(x,0)=x$. Considering the $L^{p}$-maximal estimate,
\begin{equation}\label{Lp estimate}
\biggl\|\sup_{t\in (0,1)}|e^{it\triangle}f(\gamma(x,t))|\biggl\|_{L^p(B(x_0,r))}\leq C \|f\|_{H^s(\mathbb{R})},
\end{equation}
we have,

(1) for each $s >1/4$ and $\alpha \in [1/2,1) $,  inequality (\ref{Lp estimate}) holds if  $p \le 4$;

(2) for each $s >1/4$ and $\alpha \in (1/4,1/2) $,  inequality (\ref{Lp estimate}) holds if  $p \le 8\alpha$;

(3) for each $s >1/2-\alpha$ and $\alpha \in (0,1/4] $,  inequality (\ref{Lp estimate}) holds if  $p \le 2$. \\
Moreover, the constant $C$ on the right hand side of inequality (\ref{Lp estimate}) depends only on $C_{1},C_{2}$ and the choice of $B(x_{0},r)$.
\end{theorem}

\begin{center}
\includegraphics[height=4cm]{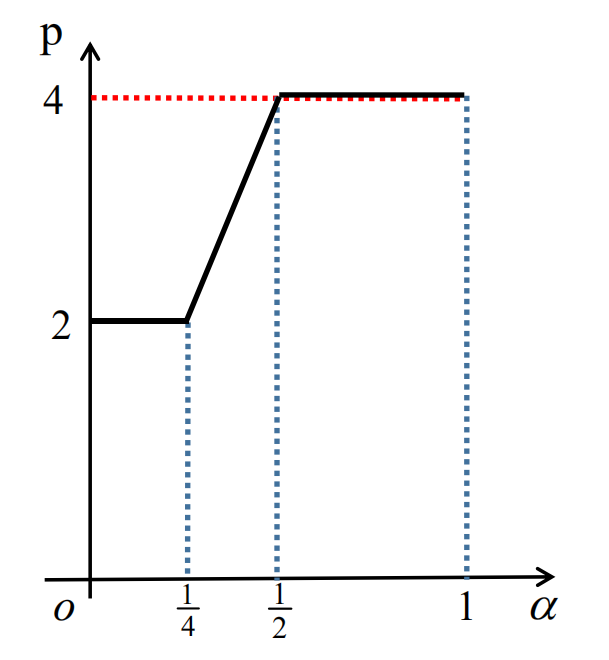}
\end{center}
\begin{center}
Figure 6. Relationship between $p$ and $\alpha$ when $s$ is fixed.
\end{center}

\begin{center}
\includegraphics[height=4cm]{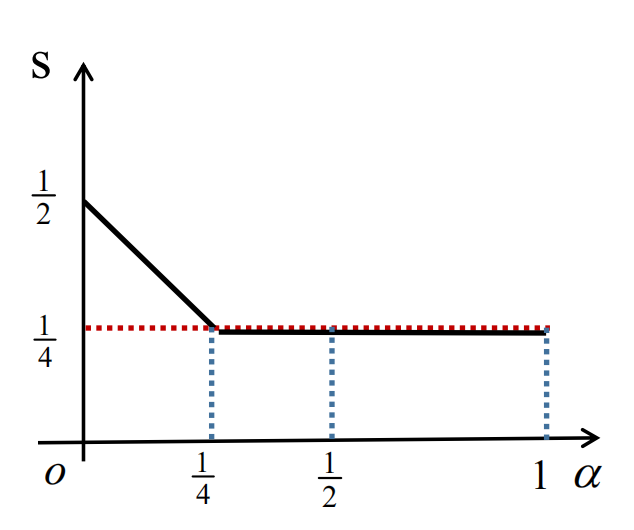}
\end{center}
\begin{center}
Figure 7. Relationship between $s$ and $\alpha$ when $p$ is fixed.
\end{center}

It is clear that Theorem \ref{upper bound} improves the previous $L^2$-Schr\"{o}dinger maximal estimates  of \cite{CLV}, see also Theorem \ref{theorem CLV} in this paper.
 Finally, we will show that the upper bound for $p$ obtained by Theorem \ref{upper bound} can not be improved when $\gamma(x,t)$ are chosen as in Theorem \ref{necessary bound for p} below.
 \begin{theorem}\label{necessary bound for p}
Taking $\gamma(x,t)=x+t^{\alpha}$, we consider the $L^{p}$-maximal estimate
\begin{equation}\label{special Lp estimate}
\biggl\|\sup_{t\in (0,1)}|e^{it\triangle}f(\gamma(x,t))|\biggl\|_{L^p(B(0,1))}\leq C \|f\|_{H^s(\mathbb{R})},
\end{equation}
we have,

(1) inequality (\ref{special Lp estimate}) holds for each $s >1/4$ and $\alpha \in [1/2,1) $  only if  $p \le 4$;

(2)  inequality (\ref{special Lp estimate}) holds for each $s >1/4$ and $\alpha \in (1/4,1/2) $  only if $p \le 8\alpha$;

(3)  inequality (\ref{special Lp estimate}) holds for each $s >1/2-\alpha$ and $\alpha \in (0,1/4] $  only if  $p \le 2$.
 \end{theorem}

We will show the proof for Theorem  \ref{necessary bound for p} in Section \ref{Necessary}.

\textbf{Conventions}: Throughout this article, we shall use the notation $A\ll B$, which means that  there is a sufficiently large constant $G$, which does not depend on the relevant parameters arising in the context in which
the quantities $A$ and $B$ appear, such that $G A\leq B$.  We write
$A\sim B$, and mean that $A$ and $B$ are comparable. By
$A\lesssim B$ we mean that $A \le CB $ for some constant $C$ independent of the parameters related to  $A$ and $B$. We write
$A\wedge B=\min\{A,B\}$. Given $\mathbb{R}^{n}$, we write $B(0,1)$ instead of the unit ball $B^{n}(0,1)$ in $\mathbb{R}^{n}$ centered at the origin for short, and the same notation is valid for $B(x_{0},r)$. We denote $RapDec(R)$ for terms rapidly decaying in $R$ which can be ignored in our estimate.


\section{Proof of Theorem \ref{theorem1.1n}}\label{Section 2n}

We will prove the sufficiency and necessity of Theorem \ref{theorem1.1n} in Subsection \ref{sufficient} and Subsection
\ref{necessary} respectively. We notice that some of
the details during the proofs have been appeared in \cite{LW}, but we write most of them for the reader's convenience.

\subsection{Sufficiency}\label{sufficient}
\begin{lemma}\label{lemma2.1n} (\cite{LW})
Assume that $g$ is a Schwartz function whose Fourier transform is supported in the annulus $A(\lambda)=\{\xi \in \mathbb{R}^{n}: |\xi| \sim \lambda\}$. $\gamma(x,t)$ satisfies
\[|\gamma(x,t)- x| \lesssim t^{\alpha},\hspace{0.8cm} \gamma(x,0)=x\]
for all $x \in B(x_{0},r)$ and $t \in (0, \lambda^{-\frac{1}{\alpha}})$. Then for each $x \in B(x_{0},r)$ and $t \in (0, \lambda^{-\frac{1}{\alpha}})$,
\begin{align}
|e^{itP(D)}g(\gamma(x,t))| \le \sum_{\mathfrak{l} \in \mathbb{Z}^{n}}{\frac{C_{n}}{(1+|\mathfrak{l}|)^{n+1}}\biggl|\int_{\mathbb{R}^{n}}{e^{i(x+ \frac{\mathfrak{l}}{\lambda})\cdot\xi+itP(\xi)}\hat{g}}(\xi)d\xi \biggl|}.
\end{align}
\end{lemma}

\textbf{Proof of Theorem \ref{theorem1.1n}.}
\textbf{(1)} When $1/m \le \alpha <1$, it is sufficient to show that for arbitrary $\varepsilon >0$,  $s_{1} = s_{0} + \varepsilon$,
\begin{equation}\label{Eq2.3n}
\biggl\|\mathop{sup}_{0<t<1} \frac{|e^{itP(D)}(f)(\gamma(x,t))-f(x)|}{t^{\delta/m}}\biggl\|_{L^{q}(B(x_{0},r))} \lesssim \|f\|_{H^{s_{1}+ \delta}(\mathbb{R}^{n})}
\end{equation}
where $0<\delta< \alpha m $.

Now we decompose $f$ as
\[f=\sum_{k=0}^{\infty}{f_{k}},\]
where supp$ \hat{f_{0}} \subset B(0,1)$, supp$ \hat{f_{k}} \subset \{\xi: |\xi| \sim 2^{k}\}$, $k \ge 1$. It follows that
\begin{align}\label{Eq2.10n}
& \biggl\|\mathop{sup}_{0<t<1} \frac{|e^{itP(D)}(f)(\gamma(x,t))-f(x)|}{t^{ \delta/m}}\biggl\|_{L^{q}(B(x_{0},r))} \nonumber\\
&\le \sum_{k=0}^{\infty}{\biggl\|\mathop{sup}_{0<t<1} \frac{|e^{itP(D)}(f_{k})(\gamma(x,t))-f_{k}(x)|}{t^{ \delta/m}}\biggl\|_{L^{q}(B(x_{0},r))}}.
\end{align}

For $k \lesssim 1$, since $P(\xi)$ is continuous,
\begin{align}\label{Eq2.11n}
&\biggl\|\mathop{sup}_{0<t<1} \frac{|e^{itP(D)}(f_{k})(\gamma(x,t))-f_{k}(x)|}{t^{ \delta/m}}\biggl\|_{L^{q}(B(x_{0},r))} \nonumber\\
&\le \biggl\|\mathop{sup}_{0<t<1} \frac{|e^{itP(D)}(f_{k})(\gamma(x,t))-f_{k}(\gamma(x,t))|}{t^{ \delta/m}}\biggl\|_{L^{q}(B(x_{0},r))} \nonumber\\
&  +\biggl\|\mathop{sup}_{0<t<1} \frac{|f_{k}(\gamma(x,t))-f_{k}(x)|}{t^{ \delta/m}}\biggl\|_{L^{q}(B(x_{0},r))} \nonumber\\
&\lesssim \|f\|_{H^{s_{1}+ \delta}(\mathbb{R}^{n})},
\end{align}
where we used
\[  \frac{|e^{itP(D)}(f_{k})(\gamma(x,t))-f_{k}(\gamma(x,t))|}{t^{ \delta/m}} \le t^{1- \delta/m}\int_{\mathbb{R}^{n}}{|P(\xi)||\hat{f_{k}}(\xi)|d\xi},\]
and
\[  \frac{|f_{k}(\gamma(x,t))-f_{k}(x)|}{t^{ \delta/m}} \le t^{\alpha- \delta/m}\int_{\mathbb{R}^{n}}{|\xi||\hat{f_{k}}(\xi)|d\xi}.\]

For $k \gg 1$,
\begin{align}\label{Eq2.12n}
&\biggl\|\mathop{sup}_{0<t<1} \frac{|e^{itP(D)}(f_{k})(\gamma(x,t))-f_{k}(x)|}{t^{ \delta/m}}\biggl\|_{L^{q}(B(x_{0},r))} \nonumber\\
&\le \biggl\|\mathop{sup}_{2^{-mk} \le t<1} \frac{|e^{itP(D)}(f_{k})(\gamma(x,t))-f_{k}(x)|}{t^{ \delta/m}}\biggl\|_{L^{q}(B(x_{0},r))} \nonumber
\end{align}
\begin{align}
& \:\:\ + \biggl\|\mathop{sup}_{0<t<2^{-mk}} \frac{|e^{itP(D)}(f_{k})(\gamma(x,t))-f_{k}(x)|}{t^{ \delta/m}}\biggl\|_{L^{q}(B(x_{0},r))} \nonumber\\
&:= I + II.
\end{align}

We first estimate $I$, inequality (\ref{Eq1.8n}) implies
\begin{equation}\label{Eq2.13n}
\biggl\|\mathop{sup}_{2^{-mk}  \le t<1} |e^{itP(D)}(f_{k})(\gamma(x,t))|\biggl\|_{L^{p}(B(x_{0},r))} \lesssim 2^{(s_{0}+\frac{\varepsilon}{2})k}\|f_{k}\|_{L^{2}(\mathbb{R}^{n})},
\end{equation}
therefore,
\begin{align}\label{Eq2.15n}
 I  &\le 2^{\delta k} \biggl\|\mathop{sup}_{2^{-mk}  \le t < 1} |e^{itP(D)}(f_{k})(\gamma(x,t))-f_{k}(x)|\biggl\|_{L^{q}(B(x_{0},r))} \nonumber\\
&\le 2^{\delta k} \biggl\{\biggl\|\mathop{sup}_{2^{-mk}  \le t < 1} |e^{itP(D)}(f_{k})(\gamma(x,t))|\biggl\|_{L^{q}(B(x_{0},r))} + \|f_{k}\|_{L^{q}(B(x_{0},r))} \biggl\}\nonumber\\
&\lesssim 2^{\delta k} \biggl\{\biggl\|\mathop{sup}_{2^{-mk} \le t < 1} |e^{itP(D)}(f_{k})(\gamma(x,t))|\biggl\|_{L^{p}(B(x_{0},r))} + \|f_{k}\|_{L^{2}(B(x_{0},r))} \biggl\} \nonumber\\
&\lesssim 2^{\delta k}2^{(s_{0}+\frac{\varepsilon}{2})k}\|f_{k}\|_{L^{2}(\mathbb{R}^{n})} \nonumber\\
&\lesssim 2^{-\frac{\varepsilon k}{2}}\|f\|_{H^{s_{1}+ \delta}(\mathbb{R}^{n})}.
\end{align}

For $II$, by the triangle inequality,
\begin{align}
II &\le \biggl\|\mathop{sup}_{0<t<2^{-mk}} \frac{|e^{itP(D)}(f_{k})(\gamma(x,t))-f_{k}(\gamma(x,t))|}{t^{ \delta/m}}\biggl\|_{L^{q}(B(x_{0},r))} \nonumber\\
& \:\:\ + \biggl\|\mathop{sup}_{0<t<2^{-mk}} \frac{|f_{k}(\gamma(x,t))-f_{k}(x)|}{t^{ \delta/m}}\biggl\|_{L^{q}(B(x_{0},r))}.
\end{align}

Using Taylor's formula, we have
\begin{align}\label{Eq2.16n}
&\biggl\|\mathop{sup}_{0<t<2^{-mk}} \frac{|f_{k}(\gamma(x,t))-f_{k}(x)|}{t^{ \delta/m}}\biggl\|_{L^{q}(B(x_{0},r))} \nonumber\\
&\lesssim \sum_{j \ge 1} \frac{1}{j!} \sum_{h_{1}, h_{2},...,h_{j} \in \{1,2,...,n\}} \biggl\|\mathop{sup}_{0<t<2^{-mk}}  \frac{ \Pi_{d=1}^{j}|\gamma_{h_{d}}(x,t)-x_{h_{d}}| }{t^{\delta/m}} \biggl|\int_{\mathbb{R}^{n}}{e^{ix\cdot\xi} \Pi_{d=1}^{j}\xi_{h_{d}} \hat{f_{k}}}(\xi)d\xi \biggl| \biggl\|_{L^{2}(B(x_{0},r))} \nonumber\\
&\lesssim \sum_{j \ge 1} \frac{2^{-\alpha mkj+ \delta k}}{j!} \sum_{h_{1}, h_{2},...,h_{j} \in \{1,2,...,n\}} \biggl\|  \int_{\mathbb{R}^{n}}{e^{ix\cdot\xi} \Pi_{d=1}^{j}\xi_{h_{d}} \hat{f_{k}}}(\xi)d\xi \biggl\|_{L^{2}(B(x_{0},r))} \nonumber
\end{align}
\begin{align}
&\lesssim \sum_{j \ge 1} \frac{2^{-\alpha mkj+ \delta k} 2^{kj} n^{j}}{j! }  \| f_{k} \|_{L^{2}(\mathbb{R}^{n})} \nonumber\\
&\lesssim 2^{\delta k}  \|\hat{f}_{k}\|_{L^{2}(\mathbb{R}^{n})} \nonumber\\
&\lesssim 2^{-s_{1}k}\|f\|_{H^{s_{1}+ \delta}(\mathbb{R}^{n})}.
\end{align}

By Taylor's formula and Lemma \ref{lemma2.1n} (notice that $2^{-mk} \le 2^{-k/\alpha}$, since $m \ge 1/\alpha$), we get
\begin{align}\label{Eq2.17n}
&\biggl\|\mathop{sup}_{0<t<2^{-mk}} \frac{|e^{itP(D)}(f_{k})(\gamma(x,t))-f_{k}(\gamma(x,t))|}{t^{ \delta/m}}\biggl\|_{L^{q}(B(x_{0},r))} \nonumber\\
&\le \sum_{j=1}^{\infty}{\frac{2^{-mkj+\delta k}}{j!} \biggl\|\mathop{sup}_{0<t<2^{-mk}} \biggl|\int_{\mathbb{R}^{n}}{e^{i\gamma(x,t)\cdot\xi}P(\xi)^{j}\hat{f_{k}}}(\xi)d\xi \biggl|\biggl\|_{L^{q}(B(x_{0},r))}} \nonumber\\
&\le \sum_{j=1}^{\infty}{\frac{2^{-mkj+\delta k}}{j!} \sum_{\mathfrak{l} \in \mathbb{Z}^{n}}{\frac{C_{n}}{(1+|\mathfrak{l}|)^{n+1}}\biggl\|\int_{\mathbb{R}^{n}}{e^{i(x+ \frac{\mathfrak{l}}{2^{k}})\cdot\xi}P(\xi)^{j}\hat{f_{k}}}(\xi)d\xi \biggl\|_{L^{q}(B(x_{0},r))}}} \nonumber\\
&\le \sum_{j=1}^{\infty}{\frac{2^{-mkj+\delta k}}{j!} \sum_{\mathfrak{l} \in \mathbb{Z}^{n}}{\frac{C_{n}}{(1+|\mathfrak{l}|)^{n+1}}\|P(\xi)^{j}\hat{f_{k}}(\xi)\|_{L^{2}(\mathbb{R}^{n})}}} \nonumber\\
&\lesssim \sum_{j=1}^{\infty}{\frac{2^{-mkj+\delta k}2^{mkj}}{j!} \|\hat{f_{k}}}(\xi)\|_{L^{2}(\mathbb{R}^{n})} \nonumber\\
&\lesssim 2^{-s_{1}k}\|f\|_{H^{s_{1}+ \delta}(\mathbb{R}^{n})}.
\end{align}

Inequalities (\ref{Eq2.15n}), (\ref{Eq2.16n}) and (\ref{Eq2.17n}) yield for $k \gg 1$,
\begin{align}\label{Eq2.18n}
\biggl\|\mathop{sup}_{0<t<1} \frac{|e^{itP(D)}(f_{k})(\gamma(x,t))-f_{k}(x)|}{t^{ \delta/m}}\biggl\|_{L^{q}(B(x_{0},r))} &\lesssim 2^{-\frac{\varepsilon k}{2}}\|f\|_{H^{s_{1}+ \delta}(\mathbb{R}^{n})}.
\end{align}

Obviously, inequality (\ref{Eq2.3n}) follows from inequalities (\ref{Eq2.10n}), (\ref{Eq2.11n}) and (\ref{Eq2.18n}).

\textbf{(2)} When $0 < \alpha <1/m$, we just need to show that for  arbitrary $\varepsilon >0$,  $s_{1} = s_{0} + \varepsilon$,
\begin{equation}
\biggl\|\mathop{sup}_{0<t<1} \frac{|e^{itP(D)}(f)(\gamma(x,t))-f(x)|}{t^{\alpha \delta}}\biggl\|_{L^{q}(B(x_{0},r))} \lesssim \|f\|_{H^{s_{1}+ \delta}(\mathbb{R}^{n})}
\end{equation}
where $0<\delta< 1 $. The proof is very similar with that of part (1), we write the details only for completeness.

For this goal, we decompose $f$ as
\[f=\sum_{k=0}^{\infty}{f_{k}},\]
where supp$ \hat{f_{0}} \subset B(0,1)$, supp$ \hat{f_{k}} \subset \{\xi: |\xi| \sim 2^{k}\}$, $k \ge 1$.

For $k \lesssim 1$,  just as the similar argument in the part (1), we have
\begin{align}
\biggl\|\mathop{sup}_{0<t<1} \frac{|e^{itP(D)}(f_{k})(\gamma(x,t))-f_{k}(x)|}{t^{\alpha \delta}}\biggl\|_{L^{q}(B(x_{0},r))}
\lesssim \|f\|_{H^{s_{1}+ \delta}(\mathbb{R}^{n})},
\end{align}
which follows from two inequalities below
\[  \frac{|e^{itP(D)}(f_{k})(\gamma(x,t))-f_{k}(\gamma(x,t))|}{t^{\alpha \delta}} \le t^{1- \alpha \delta}\int_{\mathbb{R}^{n}}{|P(\xi)||\hat{f_{k}}(\xi)|d\xi},\]
and
\[  \frac{|f_{k}(\gamma(x,t))-f_{k}(x)|}{t^{ \alpha \delta}} \le t^{\alpha- \alpha \delta}\int_{\mathbb{R}^{n}}{|\xi||\hat{f_{k}}(\xi)|d\xi}.\]

For $k \gg 1$,
\begin{align}
&\biggl\|\mathop{sup}_{0<t<1} \frac{|e^{itP(D)}(f_{k})(\gamma(x,t))-f_{k}(x)|}{t^{\alpha \delta}}\biggl\|_{L^{q}(B(x_{0},r))} \nonumber\\
&\le \biggl\|\mathop{sup}_{2^{-k/ \alpha} \le t<1} \frac{|e^{itP(D)}(f_{k})(\gamma(x,t))-f_{k}(x)|}{t^{ \alpha \delta}}\biggl\|_{L^{q}(B(x_{0},r))} \nonumber\\
& \:\:\ + \biggl\|\mathop{sup}_{0<t<2^{-k / \alpha}} \frac{|e^{itP(D)}(f_{k})(\gamma(x,t))-f_{k}(x)|}{t^{\alpha \delta}}\biggl\|_{L^{q}(B(x_{0},r))} \nonumber\\
&:= I^{\prime} + II^{\prime}.
\end{align}

From inequality (\ref{Eq1.8n}) we obtain
\begin{equation}
\biggl\|\mathop{sup}_{2^{-k/ \alpha}  \le t<1} |e^{itP(D)}(f_{k})(\gamma(x,t))|\biggl\|_{L^{p}(B(x_{0},r))} \lesssim 2^{(s_{0}+\frac{\varepsilon}{2})k}\|f_{k}\|_{L^{2}(\mathbb{R}^{n})}.
\end{equation}
Thus,
\begin{align}
 I^{\prime}  &\le 2^{\delta k} \biggl\|\mathop{sup}_{2^{-k / \alpha }  \le t < 1} |e^{itP(D)}(f_{k})(\gamma(x,t))-f_{k}(x)|\biggl\|_{L^{q}(B(x_{0},r))} \nonumber\\
&\le 2^{\delta k} \biggl\{\biggl\|\mathop{sup}_{2^{-k / \alpha}  \le t < 1} |e^{itP(D)}(f_{k})(\gamma(x,t))|\biggl\|_{L^{q}(B(x_{0},r))} + \|f_{k}\|_{L^{q}(B(x_{0},r))} \biggl\}\nonumber\\
&\lesssim 2^{\delta k} \biggl\{\biggl\|\mathop{sup}_{2^{-k / \alpha} \le t < 1} |e^{itP(D)}(f_{k})(\gamma(x,t))|\biggl\|_{L^{p}(B(x_{0},r))} + \|f_{k}\|_{L^{2}(B(x_{0},r))} \biggl\} \nonumber\\
&\lesssim 2^{\delta k}2^{(s_{0}+\frac{\varepsilon}{2})k}\|f_{k}\|_{L^{2}(\mathbb{R}^{n})} \nonumber\\
&\lesssim 2^{-\frac{\varepsilon k}{2}}\|f\|_{H^{s_{1}+ \delta}(\mathbb{R}^{n})}.
\end{align}

For $II^{\prime}$, by the triangle inequality,
\begin{align}
II^{\prime} &\le \biggl\|\mathop{sup}_{0<t<2^{-k/ \alpha}} \frac{|e^{itP(D)}(f_{k})(\gamma(x,t))-f_{k}(\gamma(x,t))|}{t^{\alpha \delta}}\biggl\|_{L^{q}(B(x_{0},r))} \nonumber\\
& \:\:\ + \biggl\|\mathop{sup}_{0<t<2^{-k / \alpha}} \frac{|f_{k}(\gamma(x,t))-f_{k}(x)|}{t^{\alpha \delta}}\biggl\|_{L^{q}(B(x_{0},r))}.
\end{align}

Applying Taylor's formula, we have
\begin{align}
&\biggl\|\mathop{sup}_{0<t<2^{-k / \alpha}} \frac{|f_{k}(\gamma(x,t))-f_{k}(x)|}{t^{\alpha \delta}}\biggl\|_{L^{q}(B(x_{0},r))} \nonumber\\
&\lesssim \sum_{j \ge 1} \frac{1}{j!} \sum_{h_{1}, h_{2},...,h_{j} \in \{1,2,...,n\}} \biggl\|\mathop{sup}_{0<t<2^{-k/ \alpha}}  \frac{ \Pi_{d=1}^{j}|\gamma_{h_{d}}(x,t)-x_{h_{d}}| }{t^{\alpha \delta}} \biggl|\int_{\mathbb{R}^{n}}{e^{ix\cdot\xi} \Pi_{d=1}^{j}\xi_{h_{d}} \hat{f_{k}}}(\xi)d\xi \biggl| \biggl\|_{L^{2}(B(x_{0},r))} \nonumber\\
&\lesssim \sum_{j \ge 1} \frac{2^{- kj+ \delta k}}{j!} \sum_{h_{1}, h_{2},...,h_{j} \in \{1,2,...,n\}} \biggl\|  \int_{\mathbb{R}^{n}}{e^{ix\cdot\xi} \Pi_{d=1}^{j}\xi_{h_{d}} \hat{f_{k}}}(\xi)d\xi \biggl\|_{L^{2}(B(x_{0},r))} \nonumber\\
&\lesssim \sum_{j \ge 1} \frac{2^{- kj+ \delta k} 2^{kj} n^{j}}{j! }  \| f_{k} \|_{L^{2}(\mathbb{R}^{n})} \nonumber\\
&\lesssim 2^{\delta k}  \|\hat{f}_{k}\|_{L^{2}(\mathbb{R}^{n})} \nonumber\\
&\lesssim 2^{-s_{1}k}\|f\|_{H^{s_{1}+ \delta}(\mathbb{R}^{n})}.
\end{align}

Using Taylor's formula and Lemma \ref{lemma2.1n}, we get
\begin{align}
&\biggl\|\mathop{sup}_{0<t<2^{-k / \alpha}} \frac{|e^{itP(D)}(f_{k})(\gamma(x,t))-f_{k}(\gamma(x,t))|}{t^{\alpha \delta}}\biggl\|_{L^{q}(B(x_{0},r))} \nonumber\\
&\le \sum_{j=1}^{\infty}{\frac{2^{-kj/ \alpha+\delta k}}{j!} \biggl\|\mathop{sup}_{0<t<2^{-k / \alpha}} \biggl|\int_{\mathbb{R}^{n}}{e^{i\gamma(x,t)\cdot\xi}P(\xi)^{j}\hat{f_{k}}}(\xi)d\xi \biggl|\biggl\|_{L^{q}(B(x_{0},r))}} \nonumber\\
&\le \sum_{j=1}^{\infty}{\frac{2^{-kj / \alpha+\delta k}}{j!} \sum_{\mathfrak{l} \in \mathbb{Z}^{n}}{\frac{C_{n}}{(1+|\mathfrak{l}|)^{n+1}}\biggl\|\int_{\mathbb{R}^{n}}{e^{i(x+ \frac{\mathfrak{l}}{2^{k}})\cdot\xi}P(\xi)^{j}\hat{f_{k}}}(\xi)d\xi \biggl\|_{L^{q}(B(x_{0},r))}}} \nonumber\\
&\le \sum_{j=1}^{\infty}{\frac{2^{-kj / \alpha+\delta k}}{j!} \sum_{\mathfrak{l} \in \mathbb{Z}^{n}}{\frac{C_{n}}{(1+|\mathfrak{l}|)^{n+1}}\|P(\xi)^{j}\hat{f_{k}}(\xi)\|_{L^{2}(\mathbb{R}^{n})}}} \nonumber\\
&\lesssim \sum_{j=1}^{\infty}{\frac{2^{-kj / \alpha +\delta k}2^{mkj}}{j!} \|\hat{f_{k}}}(\xi)\|_{L^{2}(\mathbb{R}^{n})} \nonumber\\
&\lesssim 2^{-s_{1}k}\|f\|_{H^{s_{1}+ \delta}(\mathbb{R}^{n})}.
\end{align}
So we arrive at
\begin{align}
\biggl\|\mathop{sup}_{0<t<1} \frac{|e^{itP(D)}(f_{k})(\gamma(x,t))-f_{k}(x)|}{t^{ \alpha \delta}}\biggl\|_{L^{q}(B(x_{0},r))} &\lesssim 2^{-\frac{\varepsilon k}{2}}\|f\|_{H^{s_{1}+ \delta}(\mathbb{R}^{n})}
\end{align}
for $k \gg 1$, then  finish the proof of part (2).   \hfill$\Box$


\subsection{Necessity}\label{necessary}
In this subsection, we give two counterexamples to show the sharpness of the exponent of the variable $t$ obtained by Theorem \ref{thsharp}.

 Set $f_{R}$ with
\[\widehat{f_{R}}(\xi) =\chi_{Q_{R}}(\xi),\]
where $\chi_{Q_{R}}$ is the characteristic function on the cube $Q_{R}:=\{\xi: R \le \xi_{i} \le R+1, 1 \le i \le n \}$, $R \ge 1$. Let
\[P(\xi) = \xi_{1}^{m}.\]
Choose $\gamma(x,t)= x + \mu t^{\alpha}$, $0<\alpha <1$, $\mu = (1/1000,0,...,0)$.
H\"{o}lder's inequality implies
\[\biggl\|\mathop{sup}_{0<t<1} |e^{itP(D)}f_{R}(\gamma(x,t))|\biggl\|_{L^{p}(B(0,1))} \lesssim \|f_{R}\|_{L^{2}(\mathbb{R}^{n})},  \:\ p \ge 1,\]
obviously,
\begin{equation}\label{Eq7.12+}
\biggl\|\mathop{sup}_{0<t<1} |e^{itP(D)}f_{R}(\gamma(x,t))|\biggl\|_{L^{p}(B(0,1))} \lesssim R^{\varepsilon}\|f_{R}\|_{L^{2}(\mathbb{R}^{n})},  \:\ p \ge 1,
\end{equation}
for arbitrary $\varepsilon > 0$. The convergent rate result corresponding to inequality (\ref{Eq7.12+}) is
\begin{equation}\label{Eq7.12}
\biggl\|\mathop{sup}_{0<t<1} \frac{|e^{itP(D)}(f_{R})(\gamma(x,t))-f_{R}(x)|}{t^{\delta_{1}}}\biggl\|_{L^{1}(B(0,1))} \lesssim R^{\varepsilon + \delta_{2}} \|f_{R}\|_{L^{2}(\mathbb{R}^{n})}.
\end{equation}
We observed that the implied constant in inequality (\ref{Eq7.12}) does not depend on $R$.  It is not hard to see that
\begin{equation}\label{Eq7.13}
R^{\varepsilon + \delta_{2}}\|f_{R}\|_{L^{2}(\mathbb{R}^{n})} \sim R^{\varepsilon + \delta_{2}}.
\end{equation}
By Taylor's formula, for arbitrary $t \in (0,1)$, we get
\begin{align}\label{Eq7.14}
&t^{-\delta_{1}}|e^{itP(D)}(f_{R})(\gamma(x,t))-f_{R}(x)| \nonumber\\
&=t^{-\delta_{1}}\biggl|\int_{Q_{R}}  e^{i x \cdot \xi} (e^{i \frac{t^{\alpha}}{1000}\xi_{1}+ itP(\xi)}-1)d\xi \biggl| \nonumber\\
&=t^{-\delta_{1}}\biggl|\int_{Q_{R}}  e^{i x \cdot \xi } (i \frac{t^{\alpha}}{1000}\xi_{1}+ it\xi_{1}^{m}) d\xi
+ \sum_{j\ge2} \frac{1}{j!}\int_{Q_{R}}  e^{i x \cdot \xi } (i \frac{t^{\alpha}}{1000}\xi_{1}+  it\xi_{1}^{m})^{j} d\xi \biggl| \nonumber\\
&\ge  t^{-\delta_{1}}\biggl| \biggl| \int_{Q_{R}}  e^{i x \cdot \xi } (i \frac{t^{\alpha}}{1000}\xi_{1}+ it\xi_{1}^{m}) d\xi\biggl| - \biggl| \sum_{j\ge2} \frac{1}{j!}\int_{Q_{R}}  e^{i x \cdot \xi } (i \frac{t^{\alpha}}{1000}\xi_{1}+  it\xi_{1}^{m})^{j} d\xi \biggl| \biggl|.
\end{align}

When $1/m \le \alpha <1$, we take $t_{0}= \frac{R^{-m}}{100}$, for each $\xi \in Q_{R}$,
\[ \frac{1}{100} \le  \frac{t_{0}^{\alpha}}{1000}\xi_{1}+ t_{0}\xi_{1}^{m}   \le \frac{1}{50} \frac{(R+1)^{m}}{R^{m}}. \]
Then for any $x \in B(0,\frac{1}{1000})$,
\[\biggl|\int_{Q_{R}}  e^{i x \cdot \xi } (i \frac{t_{0}^{\alpha}}{1000}\xi_{1}+ it_{0}\xi_{1}^{m}) d\xi\biggl|
= \biggl|\int_{\{\eta: 0 \le \eta_{i} \le 1, 1 \le i \le n\}}  e^{i x \cdot \eta} [ \frac{t_{0}^{\alpha}}{1000} (\eta_{1}+R )+ t_{0} (\eta_{1}+R )^{m}] d\eta \biggl|,\]
where the phase function satisfies
\[|x \cdot \eta | \le \frac{1}{100}. \]
So we have
\begin{equation}\label{Eq7.14+}
\biggl|\int_{Q_{R}}  e^{i x \cdot \xi } (i \frac{t_{0}^{\alpha}}{1000}\xi_{1}+ it_{0}\xi_{1}^{m}) d\xi\biggl| \ge  \frac{1}{200}
\end{equation}
for each $x \in B(0,\frac{1}{1000})$.
When $R$ is sufficiently large so that $(R+1)^{m}/R^{m} \le 2$,  we obtain
\begin{align}\label{Eq7.14++++}
&\biggl|\sum_{j\ge2} \frac{1}{j!}\int_{Q_{R}}  e^{i x \cdot \xi } (i \frac{t_{0}^{\alpha}}{1000}\xi_{1}+ it_{0}\xi_{1}^{m})^{j} d\xi \biggl| \nonumber\\
&\le \sum_{j\ge2} \frac{1}{j!}\int_{Q_{R}}   \biggl| \frac{t_{0}^{\alpha}}{1000}\xi_{1}+ t_{0}\xi_{1}^{m} \biggl|^{j} d\xi  \nonumber\\
&\le \sum_{j\ge2} \frac{1}{j!} \biggl(\frac{1}{50}\biggl)^{j} \biggl( \frac{(R+1)^{m}}{R^{m}}  \biggl)^{j} \nonumber\\
&\le \frac{e-2}{625}
\end{align}
for any $x \in B(0,\frac{1}{1000})$. Therefore, for each $R$ which satisfies  $(R+1)^{m}/R^{m} \le 2$, inequalities (\ref{Eq7.14}), (\ref{Eq7.14+}) and (\ref{Eq7.14++++}) imply
\begin{align}\label{Eq7.15}
&\biggl\|\mathop{sup}_{0<t<1} \frac{|e^{itP(D)}(f_{R})(\gamma(x,t))-f_{R}(x)|}{t^{\delta_{1}}}\biggl\|_{L^{1}(B(0,1))}  \nonumber\\
&\ge \biggl\| \frac{|e^{it_{0}P(D)}(f_{R})(\gamma(x,t_{0}))-f_{R}(x)|}{t_{0}^{\delta_{1}}}\biggl\|_{L^{1}(B(0,\frac{1}{1000}))} \nonumber\\
&\gtrsim R^{\delta_{1}m}.
\end{align}
If inequality (\ref{Eq7.12}) holds true for each $\varepsilon > 0$, inequalities (\ref{Eq7.13}) and (\ref{Eq7.15}) yield
\[R^{\delta_{1}m-\delta_{2}-\varepsilon} \lesssim 1.\]
When $R$ tends to infinity, this is possible only if
\[\delta_{1} \le \delta_{2}/m,\]
since $\varepsilon$ is arbitrarily small.

When $0 <\alpha < 1/m$, we choose $t_{0}= R^{-1/\alpha}$, then for each $\xi \in Q_{R}$,
\[ \frac{1}{1000} \le  \frac{t_{0}^{\alpha}}{1000}\xi_{1}+ t_{0}\xi_{1}^{m}    \le \frac{1}{1000} \frac{(R+1)}{R} + \frac{(R+1)^{m}}{R^{1/ \alpha}}. \]
Therefore, we have
\begin{equation}\label{Eq7.14++}
\biggl|\int_{Q_{R}}  e^{i x \cdot \xi } (i \frac{t_{0}^{\alpha}}{1000}\xi_{1}+ it_{0}\xi_{1}^{m}) d\xi\biggl| \ge  \frac{1}{2000}
\end{equation}
for any $x \in B(0,\frac{1}{1000})$.
When $R$ is sufficiently large so that $(R+1)/R \le 3/2$ and $(R+1)^{m} / R^{1/ \alpha} < 1/2000$,  we obtain
\begin{align}\label{Eq7.14+++}
&\biggl|\sum_{j\ge2} \frac{1}{j!}\int_{Q_{R}}  e^{i x \cdot \xi } (i \frac{t_{0}^{\alpha}}{1000}\xi_{1}+ it_{0}\xi_{1}^{m})^{j} d\xi \biggl| \nonumber\\
&\le \sum_{j\ge2} \frac{1}{j!}\int_{Q_{R}}   \biggl| \frac{t_{0}^{\alpha}}{1000}\xi_{1}+ t_{0}\xi_{1}^{m} \biggl|^{j} d\xi \nonumber\\
&\le \sum_{j\ge2} \frac{1}{j!} \biggl(\frac{1}{500}\biggl)^{j}  \nonumber\\
&\le \frac{e-2}{250000}
\end{align}
for each $x \in B(0,\frac{1}{1000})$. Therefore, for large enough $R$, inequalities (\ref{Eq7.14}), (\ref{Eq7.14++}) and (\ref{Eq7.14+++}) imply
\begin{align}\label{Eq7.15+}
&\biggl\|\mathop{sup}_{0<t<1} \frac{|e^{itP(D)}(f_{R})(\gamma(x,t))-f_{R}(x)|}{t^{\delta_{1}}}\biggl\|_{L^{1}(B(0,1))}
\nonumber\\
&\ge \biggl\| \frac{|e^{it_{0}P(D)}(f_{R})(\gamma(x,t_{0}))-f_{R}(x)|}{t_{0}^{\delta_{1}}}\biggl\|_{L^{1}(B(0,\frac{1}{1000}))} \nonumber\\
&\gtrsim R^{\delta_{1}/\alpha}.
\end{align}
If inequality (\ref{Eq7.12}) holds true for each $\varepsilon > 0$, inequality (\ref{Eq7.13}) and (\ref{Eq7.15+}) yield
\[R^{\delta_{1}/\alpha-\delta_{2}-\varepsilon} \lesssim 1.\]
When $R$ tends to infinity, since $\varepsilon$ is arbitrarily small,  this is possible only if
\[\delta_{1} \le \alpha \delta_{2}.\]

\section{Proof of Theorem \ref{reduction main theorem}}\label{proof for reduction main theorem}
Before the proof of Theorem \ref{reduction main theorem}, we first prove the following lemma.
\begin{lemma}\label{main lemma}
Let $\gamma \in \Gamma_{\alpha, R^{-1}}$, $\alpha \in [1/2,1)$. suppose that $\widetilde{\gamma}$ is a measurable function from $[0, R]$ to $\mathbb{R}^{2}$.

\textbf{(1)} If supp $\hat{f} \subset B(\xi_{0}, \rho^{-\alpha})$ for some $1 \le \rho \le R$, then for each $(x,t) \in B(x_{0}, \rho) \times (t_{0},t_{0}+\rho) \subset B(0,R)\times [0,R]$, we have
\begin{equation}\label{expansion 1}
\biggl|e^{it\Delta}f(x+\widetilde{\gamma}(t) + R\gamma(\frac{t}{R^{2}})) \biggl| \lesssim \sum_{\mathfrak{l} \in \mathbb{Z}^{2}}
\frac{1}{(1+|\mathfrak{l}|)^{100}} \biggl|e^{it\Delta}f_{\mathfrak{l}}(x+ \widetilde{\gamma}(t) + R\gamma(\frac{t_{0}}{R^{2}}) )\biggl|,
\end{equation}
where $\hat{f_{\mathfrak{l}}}(\xi) = e^{i\rho^{\alpha} \mathfrak{l} \cdot \xi} \hat{f}(\xi)$, the implied constant does not depend on the choice of $\gamma$ and $\widetilde{\gamma}$.

 \textbf{(2)} If $\hat{f}$ is supported in a ball of radius $\rho^{-1}$, then for each $(x,t) \in B(x_{0}, \rho) \times (t_{0},t_{0}+\rho)\subset B(0,R)\times [0,R]$, we have
\begin{align}\label{expansion 2}
&\biggl|e^{it\Delta}f(x+ R\gamma(\frac{t}{R^{2}})) \biggl| \nonumber\\
&\lesssim \sum_{\mathfrak{l, m} \in \mathbb{Z}^{2} \times \mathbb{Z}^{2}}
\frac{\rho^{-3}}{(1+|\mathfrak{l}|)^{100}(1+|\mathfrak{m}|)^{100}}  \int_{B(x_{0},\rho)} \int_{t_{0}}^{t_{0} + \rho} \biggl|e^{it\Delta}f_{\mathfrak{l,m}}(y + R\gamma(\frac{t}{R^{2}}) )\biggl|dtdy,
\end{align}
where $\widehat{f_{\mathfrak{l,m}}}(\xi) = e^{i\rho \mathfrak{l} \cdot \xi+ i\rho \mathfrak{m} \cdot \xi} \hat{f}(\xi)$, the implied constant does not depend on the choice of $\gamma$.
\end{lemma}
\begin{proof}
\textbf{(1)} We choose a cutoff function $\phi$ which equals $1$ on $B(0,1)$ and  is supported in $[-\pi, \pi]^{2}$, then
\begin{align}
&\biggl|e^{it\Delta}f(x+\widetilde{\gamma}(t) + R\gamma(\frac{t}{R^{2}})) \biggl| \nonumber\\
&= \biggl| \int_{\mathbb{R}^{2}}e^{ix \cdot \xi + i \widetilde{\gamma }(t) \cdot \xi + i R\gamma(\frac{t}{R^{2}}) \cdot \xi + it |\xi|^{2} } \phi \biggl(\frac{\xi -\xi_{0}}{\rho^{-\alpha}} \biggl) \hat{f}(\xi)d\xi \biggl| \nonumber\\
& = \rho^{-2\alpha} \biggl| \int_{\mathbb{R}^{2}}e^{ix \cdot (\rho^{-\alpha} \eta + \xi_{0}) + i \widetilde{\gamma }(t) \cdot (\rho^{-\alpha} \eta + \xi_{0}) + i [R\gamma(\frac{t}{R^{2}}) - R\gamma(\frac{t_{0}}{R^{2}}) ] \cdot (\rho^{-\alpha} \eta + \xi_{0}) + i  R\gamma(\frac{t_{0}}{R^{2}})  \cdot (\rho^{-\alpha} \eta + \xi_{0}) + it |\rho^{-\alpha} \eta + \xi_{0}|^{2} } \nonumber\\
& \quad \quad  \times \phi(\eta) \hat{f}(\rho^{-\alpha} \eta + \xi_{0})d\eta \biggl| \nonumber\\
& = \rho^{-2\alpha} \biggl| \int_{\mathbb{R}^{2}}e^{ix \cdot (\rho^{-\alpha} \eta + \xi_{0}) + i \widetilde{\gamma }(t) \cdot (\rho^{-\alpha} \eta + \xi_{0}) + i [R\gamma(\frac{t}{R^{2}}) - R\gamma(\frac{t_{0}}{R^{2}}) ] \cdot \rho^{-\alpha} \eta  + i  R\gamma(\frac{t_{0}}{R^{2}})  \cdot (\rho^{-\alpha} \eta + \xi_{0}) + it |\rho^{-\alpha} \eta + \xi_{0}|^{2} } \nonumber\\
& \quad \quad  \times \phi(\eta) \hat{f}(\rho^{-\alpha} \eta + \xi_{0})d\eta \biggl|. \nonumber
\end{align}
By Fourier expansion,
\[ e^{i [R\gamma(\frac{t}{R^{2}}) - R\gamma(\frac{t_{0}}{R^{2}}) ] \cdot \rho^{-\alpha} \eta  }\phi(\eta) = \sum_{\mathfrak{l} \in \mathbb{Z}^{2}} c_{\mathfrak{l}}(t) e^{i\mathfrak{l} \cdot \eta}.\]
Since for each $t \in (t_{0},t_{0} +\rho)$, we get
\[ \biggl| R\gamma(\frac{t}{R^{2}}) - R\gamma(\frac{t_{0}}{R^{2}}) \biggl| \rho^{-\alpha} \le C_{\alpha} R^{1-2\alpha} \rho^{\alpha}  \rho^{-\alpha} \le C_{\alpha}.\]
Integration by parts implies that,
\[|c_{\mathfrak{l}}(t)| \lesssim \frac{1}{(1+|\mathfrak{l}|)^{100}},\]
where the implied constant does not depend on the choice of $\gamma$. Therefore, we obtain
\begin{align}
&\biggl|e^{it\Delta}f(x+\widetilde{\gamma}(t) + R\gamma(\frac{t}{R^{2}})) \biggl| \nonumber\\
&\le \sum_{\mathfrak{l} \in \mathbb{Z}^{2}} \rho^{-2\alpha} |c_{\mathfrak{l}}(t)| \biggl| \int_{\mathbb{R}^{2}}e^{ix \cdot (\rho^{-\alpha} \eta + \xi_{0}) + i \widetilde{\gamma }(t) \cdot (\rho^{-\alpha} \eta + \xi_{0}) + i\mathfrak{l} \cdot  \eta  + i  R\gamma(\frac{t_{0}}{R^{2}})  \cdot (\rho^{-\alpha} \eta + \xi_{0}) + it |\rho^{-\alpha} \eta + \xi_{0}|^{2} } \nonumber\\
& \quad \quad  \times   \hat{f}(\rho^{-\alpha} \eta + \xi_{0})d\eta \biggl| \nonumber\\
&= \sum_{\mathfrak{l} \in \mathbb{Z}^{2}}  |c_{\mathfrak{l}}(t)| \biggl| \int_{\mathbb{R}^{2}}e^{ix \cdot \xi + i \widetilde{\gamma }(t) \cdot \xi + i\mathfrak{l} \cdot  (\rho^{\alpha} \xi -\rho^{\alpha} \xi_{0})  + i  R\gamma(\frac{t_{0}}{R^{2}})  \cdot \xi + it |\xi|^{2} }  \hat{f}(\xi)d\xi \biggl| \nonumber\\
& \lesssim \sum_{\mathfrak{l} \in \mathbb{Z}^{2}}
\frac{1}{(1+|\mathfrak{l}|)^{100}} \biggl|e^{it\Delta}f_{\mathfrak{l}}(x+ \widetilde{\gamma}(t) + R\gamma(\frac{t_{0}}{R^{2}}) )\biggl|.
\end{align}

\textbf{(2)} By inequality (\ref{expansion 1}), taking $\tilde{\gamma}\equiv0$, we have
\begin{equation}
\biggl|e^{it\Delta}f(x  + R\gamma(\frac{t}{R^{2}})) \biggl| \lesssim \sum_{\mathfrak{l} \in \mathbb{Z}^{2}}
\frac{1}{(1+|\mathfrak{l}|)^{100}} \biggl|e^{it\Delta}f_{\mathfrak{l}}(x + R\gamma(\frac{t_{0}}{R^{2}}) )\biggl|.
\end{equation}
Since the Fourier transform of $e^{it\Delta}f_{\mathfrak{l}}(x + R\gamma(\frac{t_{0}}{R^{2}}) )$ is supported in a ball of radius $\rho^{-1}$, by locally constant property, for any $(x,t) \in B(x_{0}, \rho) \times (t_{0},t_{0}+\rho)$, we get
\[ \biggl |e^{it\Delta}f_{\mathfrak{l}}(x + R\gamma(\frac{t_{0}}{R^{2}}) ) \biggl |  \lesssim \rho^{-3} \int_{B(x_{0},\rho)} \int_{t_{0}}^{t_{0} + \rho} \biggl|e^{it\Delta}f_{\mathfrak{l}}(y + R\gamma(\frac{t_{0}}{R^{2}}) )\biggl|dtdy,\]
where the implied constant does not depend on $\mathfrak{l}$.
Taking $\tilde{\gamma}(t) = R\gamma(\frac{t}{R^{2}})$, and applying inequality (\ref{expansion 1}) with   $-\gamma(t)+ \gamma(\frac{t_{0}}{R^{3}})$ instead of $\gamma(t)$, then for each $(y,t) \in B(x_{0}, \rho) \times (t_{0},t_{0}+\rho)$,
\begin{align}
  &\biggl|e^{it\Delta}f_{\mathfrak{l}}(y + R\gamma(\frac{t_{0}}{R^{2}}) )\biggl| \nonumber\\
  &=\biggl|e^{it\Delta}f_{\mathfrak{l}}( y +R\gamma(\frac{t}{R^{2}}) -R\gamma(\frac{t}{R^{2}})+  R\gamma(\frac{t_{0}}{R^{2}}) )\biggl| \nonumber\\
  &\lesssim \sum_{\mathfrak{m} \in \mathbb{Z}^{2}} \frac{1}{(1+|\mathfrak{m}|)^{100}} \biggl|e^{it\Delta}f_{\mathfrak{l,m}}(y +R\gamma(\frac{t}{R^{2}}) )\biggl| \nonumber
\end{align}
and inequality (\ref{expansion 2}) is proved.
\end{proof}

Now we are ready to prove Theorem \ref{reduction main theorem} by induction on $M$. Let's first check the base of the induction.
When $  M^{-1} \leq R^{-1/2} $ and  supp $\hat{f} \subset B(\xi_{0}, R^{-1/2})$,  we can apply the same method as described in the proof of Lemma \ref{main lemma} (1) to obtain that
 for each $(x,t) \in B(0,R) \times (0,R)$,
 \begin{equation}
\biggl|e^{it\Delta}f(x + R\gamma(\frac{t}{R^{2}})) \biggl| \lesssim \sum_{\mathfrak{l} \in \mathbb{Z}^{2}}
\frac{1}{(1+|\mathfrak{l}|)^{100}} \biggl|e^{it\Delta}f_{\mathfrak{l}}(x  + R\gamma(0) )\biggl|,
\end{equation}
 where $\hat{f_{\mathfrak{l}}}(\xi) = e^{iR^{1/2} \mathfrak{l} \cdot \xi} \hat{f}(\xi)$, then
 \begin{align}
 &\biggl\|\sup_{t\in (0,R)}|e^{it\triangle}f(x+R\gamma (\frac{t}{R^{2}}))|\biggl\|_{L^p(B(0,R))} \nonumber\\
 &\lesssim \sum_{\mathfrak{l} \in \mathbb{Z}^{2}} \frac{1}{(1+|\mathfrak{l}|)^{100}} \biggl\|\sup_{t\in (0,R)}|e^{it\triangle}f_{\mathfrak{l}}(x+R\gamma (0))|\biggl\|_{L^p(B(0,R))} \nonumber\\
 &\le \sum_{\mathfrak{l} \in \mathbb{Z}^{2}} \frac{1}{(1+|\mathfrak{l}|)^{100}} C M^{-\varepsilon^{2}}R^{2/p - 5/8 + \varepsilon} \biggl \|e^{iR^{1/2}\mathfrak{l} \cdot \xi + i R\gamma(0) \cdot \xi}\hat{f}(\xi) \biggl \|_{L^2} \nonumber\\
&\le C M^{-\varepsilon^{2}}R^{2/p - 5/8 + \varepsilon}  \|f\|_{L^{2}},
 \end{align}
 and the constant $C$ does not depend on the choice of $\gamma$. Hence inequality (\ref{reduction main estimate}) holds when $M^{-1} \leq R^{-1/2}$. Note that we have used the result obtained by \cite{DL} for Schr\"{o}dinger operator along vertical lines here.

Now we assume that inequality (\ref{reduction main estimate}) holds for  $f$ whose Fourier transform is supported in a ball of radius $(KM)^{-1}$, $K \gg 1$.

We   decompose $B\left( {0,R} \right)$ into balls ${B_K}$ of  radius $K$, and interval $\left( {0,R} \right)$ into intervals $I_K^j$ of length $K$, then
\begin{align}
&\left\| {\mathop {\sup }\limits_{t \in \left( {0,R} \right]} \left| {{e^{it\Delta}}f(x+R\gamma(\frac{t}{R^{2}}))} \right| } \right\|_{{L^p}\left( {B\left( {0,R} \right)} \right)} \nonumber\\ &= \biggl(\sum\limits_{{B_K} \subset B\left( {0,R} \right)} {\int_{{B_K}} {\mathop {\sup }\limits_{I_K^j \subset \left[ {0,R} \right]} } } \mathop {\sup }\limits_{t \in I_K^j} {\left| {{e^{it\Delta}}f(x+R\gamma(\frac{t}{R^{2}}))} \right|^p}dx\biggl)^{1/p}.
\end{align}
We divide $B\left( {{\xi _0},{M^{ - 1}}} \right)$ into balls $\tau $ of radius ${\left( {KM} \right)^{ - 1}},$ $f = \sum\nolimits_\tau  {{f_\tau }} ,\widehat{{f_\tau }} = {\left. {\hat{f}} \right|_\tau }.$ By Lemma \ref{main lemma} (2), for each $(x,t) \in B_{K} \times I_{K}^{j}$,
\begin{align}
&\biggl|e^{it\Delta}f(x+ R\gamma(\frac{t}{R^{2}})) \biggl| \le \sum_{\tau}\biggl|e^{it\Delta}f_{\tau}(x+ R\gamma(\frac{t}{R^{2}})) \biggl| \nonumber\\
&\lesssim \sum_{\mathfrak{l, m} \in \mathbb{Z}^{2} \times \mathbb{Z}^{2}}
\frac{K^{-3}}{(1+|\mathfrak{l}|)^{100}(1+|\mathfrak{m}|)^{100}} \sum_{\tau}  \int_{B_{K}} \int_{I_{K}^{j}} \biggl|e^{it\Delta}f_{\tau, \mathfrak{l,m}}(y + R\gamma(\frac{t}{R^{2}}) )\biggl|dtdy.
\end{align}
Therefore,
\begin{align}\label{Eq29}
&\biggl\|\sup_{t\in (0,R)}|e^{it\triangle}f(x+R\gamma(\frac{t}{R^{2}}))|\biggl\|_{L^p(B(0,R))} \nonumber\\
&\lesssim \sum_{\mathfrak{l, m} \in \mathbb{Z}^{2} \times \mathbb{Z}^{2}}
\frac{1}{(1+|\mathfrak{l}|)^{100}(1+|\mathfrak{m}|)^{100}} \nonumber\\
& \quad \times \biggl\{ \sum_{B_{K} \subset B(0,R)} \int_{B_{K}} \mathop{sup}_{ I_{K}^{j} \subset (0,R)} \biggl( \sum_{\tau}  K^{-3} \int_{B_{K}} \int_{I_{K}^{j}} \biggl|e^{it\Delta}f_{\tau, \mathfrak{l,m}}(y + R\gamma(\frac{t}{R^{2}}) )\biggl| dtdy \biggl)^{p} dx\biggl\}^{1/p} \nonumber
\end{align}
\begin{align}
&\lesssim \sum_{\mathfrak{l, m} \in \mathbb{Z}^{2} \times \mathbb{Z}^{2}}
\frac{1}{(1+|\mathfrak{l}|)^{100}(1+|\mathfrak{m}|)^{100}} \nonumber\\
& \quad \times \biggl\{ \sum_{B_{K} \subset B(0,R)}  \mathop{sup}_{ I_{K}^{j} \subset (0,R)} \biggl[ \sum_{\tau}  K^{-1/p} \biggl( \int_{B_{K}} \int_{I_{K}^{j}} \biggl|e^{it\Delta}f_{\tau, \mathfrak{l,m}}(y + R\gamma(\frac{t}{R^{2}}) )\biggl|^{p} dtdy \biggl)^{1/p} \biggl]^{p} \biggl\}^{1/p}.
\end{align}

Fix $\mathfrak{l,m}$, for each ${B_K} \times I_K^j$ and a parameter $A \in \mathbb{Z}^{+}$, we choose $(k-1)$-dimensional subspaces $V_1^{0},V_2^{0},...,V_A^{0}$ such that
\begin{align}\label{Eq210}
&{\mu _{{e^{it\Delta}}f_{\mathfrak{l, m}}(x+R\gamma(\frac{t}{R^{2}}))}}\left( {{B_K} \times I_K^j} \right) \nonumber\\
&:= \mathop {\min }\limits_{{V_1},{V_2},...,{V_A}} \biggl\{ {\mathop {\max }\limits_{\tau  \notin {V_\alpha },\alpha  = 1,2,...,A} \int \int_{{B_K} \times I_K^j} {{{\left| {{e^{it\Delta}}{f_{\tau,\mathfrak{l,m}}(y+R\gamma(\frac{t}{R^{2}})) }} \right|}^p}dydt} \biggl\}}
\end{align}
achieves the minimum, in our case $k=2$. We say that $\tau  \in {V_\alpha }$ if
\[\mathop {\inf }\limits_{\xi  \in \tau } Angle\left( {\frac{{\left( { - 2\xi ,1} \right)}}{{\left| {\left( { - 2\xi ,1} \right)} \right|}},{V_\alpha }} \right) \le {\left( {KM} \right)^{ - 1}}.\]
Then from inequality (\ref{Eq29}),
\begin{align}
 &\biggl\|\sup_{t\in (0,R)}|e^{it\triangle}f(x+R\gamma(\frac{t}{R^{2}}))|\biggl\|_{L^p(B(0,R))} \nonumber\\
&\lesssim \sum_{\mathfrak{l, m} \in \mathbb{Z}^{2} \times \mathbb{Z}^{2}}
\frac{1}{(1+|\mathfrak{l}|)^{100}(1+|\mathfrak{m}|)^{100}} \nonumber\\
& \times \biggl\{ \sum_{B_{K} \subset B(0,R)}  \mathop{sup}_{ I_{K}^{j} \subset (0,R)} \biggl[ \sum_{\tau: \tau \notin V_{\alpha}^{0}: 1 \le \alpha \le A}  \frac{1}{K^{1/p}} \biggl( \int_{B_{K}} \int_{I_{K}^{j}} \biggl|e^{it\Delta}f_{\tau, \mathfrak{l,m}}(y + R\gamma(\frac{t}{R^{2}}) )\biggl|^{p} dtdy \biggl)^{1/p} \biggl]^{p} \biggl\}^{1/p} \nonumber\\
&\quad \quad +\sum_{\mathfrak{l, m} \in \mathbb{Z}^{2} \times \mathbb{Z}^{2}}
\frac{1}{(1+|\mathfrak{l}|)^{100}(1+|\mathfrak{m}|)^{100}} \nonumber\\
& \times \biggl\{ \sum_{B_{K} \subset B(0,R)}  \mathop{sup}_{ I_{K}^{j} \subset (0,R)} \biggl[ \sum_{\tau: \tau \in \text{ some } V_{\alpha}^{0}}  \frac{1}{K^{1/p}} \biggl( \int_{B_{K}} \int_{I_{K}^{j}} \biggl|e^{it\Delta}f_{\tau, \mathfrak{l,m}}(y + R\gamma(\frac{t}{R^{2}}) )\biggl|^{p} dtdy \biggl)^{1/p} \biggl]^{p} \biggl\}^{1/p} \nonumber\\
&\le K^{10} \sum_{\mathfrak{l, m} \in \mathbb{Z}^{2} \times \mathbb{Z}^{2}}
\frac{1}{(1+|\mathfrak{l}|)^{100}(1+|\mathfrak{m}|)^{100}} \biggl( \sum_{B_{K} \subset B(0,R)}  \mathop{sup}_{ I_{K}^{j} \subset (0,R)}  \mu_{e^{it\Delta}f_{\mathfrak{l,m}}(x+R\gamma(\frac{t}{R^{2}}))}(B_{K} \times I_{K}^{j})  \biggl)^{1/p} \nonumber\\
&\quad \quad +A \sum_{\mathfrak{l, m} \in \mathbb{Z}^{2} \times \mathbb{Z}^{2}}
\frac{1}{(1+|\mathfrak{l}|)^{100}(1+|\mathfrak{m}|)^{100}} \nonumber\\
& \times \biggl(\sum_{\tau: \tau \in \text{ some } V_{\alpha}^{0}}  \sum_{B_{K} \subset B(0,R)}  \mathop{sup}_{ I_{K}^{j} \subset (0,R)}   \frac{1}{K}  \int_{B_{K}} \int_{I_{K}^{j}} \biggl|e^{it\Delta}f_{\tau, \mathfrak{l,m}}(y + R\gamma(\frac{t}{R^{2}}) )\biggl|^{p} dtdy   \biggl)^{1/p} \nonumber\\
&:= I_{1} + I_{2}.
\end{align}

 If $I_{1}$ dominates, we define
\begin{align}\label{Eq211}
&\left\| {{e^{it\Delta}}f_{\mathfrak{l,m}}(x+R\gamma(\frac{t}{R^{2}}))} \right\|_{BL_{k,A}^p{L^\infty }\left( {B\left( {0,R} \right) \times \left[ {0,R} \right]} \right)} \nonumber\\
&: = \biggl( \sum_{B_{K} \subset B(0,R)}  \mathop{sup}_{ I_{K}^{j} \subset (0,R)}  \mu_{e^{it\Delta}f_{\mathfrak{l,m}}(x+R\gamma(\frac{t}{R^{2}}))}(B_{K} \times I_{K}^{j})  \biggl)^{1/p}.
  \end{align}
If Theorem \ref{Theorem 2.1} below holds true, we have
\[{I_1} \le {K^{10}} \sum_{\mathfrak{l, m} \in \mathbb{Z}^{2} \times \mathbb{Z}^{2}}
\frac{1}{(1+|\mathfrak{l}|)^{100}(1+|\mathfrak{m}|)^{100}}  { {C\left( {K,\frac{\varepsilon }{2}, C_{\alpha}} \right)R^{\frac{\varepsilon^{2}}{4}} {M^{ - {\varepsilon ^2}}}{R^{\frac{2}{p} - \frac{5}{8} + \frac{\varepsilon }{2}}}{{\left\| f_{\mathfrak{l,m}} \right\|}_{{L^2}}}} }.\]
Theorem \ref{reduction main theorem} follows from the fact that $R$ is sufficiently large and $\|f_{\mathfrak{l,m}}\|_{L^{2}}=\|f\|_{L^{2}} $.

If $I_{2}$ dominates, then we will prove (\ref{reduction main estimate}) by the induction on the frequency radius $1/M$. Notice that (\ref{reduction main estimate}) holds for frequency radius $1/(KM)$ by previous assumption. We get
\begin{align}
 {I_2} &\le A \sum_{\mathfrak{l, m} \in \mathbb{Z}^{2} \times \mathbb{Z}^{2}}
\frac{1}{(1+|\mathfrak{l}|)^{100}(1+|\mathfrak{m}|)^{100}} \nonumber\\
& \times \biggl(\sum_{\tau: \tau \in \text{ some } V_{\alpha}^{0}}  \sum_{B_{K} \subset B(0,R)}    \int_{B_{K}} \mathop{sup}_{ I_{K}^{j} \subset (0,R)} \frac{1}{K}  \int_{I_{K}^{j}} \biggl|e^{it\Delta}f_{\tau, \mathfrak{l,m}}(y + R\gamma(\frac{t}{R^{2}}) )\biggl|^{p} dtdy   \biggl)^{1/p} \nonumber\\
 &\lesssim A \sum_{\mathfrak{l, m} \in \mathbb{Z}^{2} \times \mathbb{Z}^{2}}
\frac{1}{(1+|\mathfrak{l}|)^{100}(1+|\mathfrak{m}|)^{100}} \nonumber\\
& \times \biggl(\sum_{\tau: \tau \in \text{ some } V_{\alpha}^{0}}  \sum_{B_{K} \subset B(0,R)}    \int_{B_{K}} \mathop{sup}_{ I_{K}^{j} \subset (0,R)} \mathop{sup}_{t \in I_{K}^{j}}  \biggl|e^{it\Delta}f_{\tau, \mathfrak{l,m}}(y + R\gamma(\frac{t}{R^{2}}) )\biggl|^{p} dy   \biggl)^{1/p} \nonumber\\
 &= A \sum_{\mathfrak{l, m} \in \mathbb{Z}^{2} \times \mathbb{Z}^{2}}
\frac{1}{(1+|\mathfrak{l}|)^{100}(1+|\mathfrak{m}|)^{100}} \nonumber\\
& \times \biggl(\sum_{\tau: \tau \in \text{ some } V_{\alpha}^{0}}     \int_{B(0,R)} \mathop{sup}_{ t \in (0,R)} \biggl|e^{it\Delta}f_{\tau, \mathfrak{l,m}}(y + R\gamma(\frac{t}{R^{2}}) )\biggl|^{p} dy   \biggl)^{1/p} \nonumber\\
&\le A \sum_{\mathfrak{l, m} \in \mathbb{Z}^{2} \times \mathbb{Z}^{2}}
\frac{1}{(1+|\mathfrak{l}|)^{100}(1+|\mathfrak{m}|)^{100}} \biggl(\sum_{\tau} (KM)^{-\varepsilon^{2}p}R^{2-\frac{5p}{8}+p\varepsilon} \|f_{\tau,\mathfrak{l,m}}\|_{L^{2}}^{p}  \biggl)^{1/p} \nonumber\\
  &\le C{A}{K^{ - {\varepsilon ^2}}}  {{C_\varepsilon }{M^{ - {\varepsilon ^2}}}{R^{\frac{2}{p} - \frac{5}{8} + \varepsilon }}{{\left\| {{f }} \right\|}_{{L^2}}}} .\nonumber
  \end{align}
Inequality (\ref{reduction main estimate}) follows if we choose $K$ sufficiently large such that $A{K^{ - {\varepsilon ^2}}} \ll 1.$ Notice the constants through our the proof are all independent of the choice of $\gamma$.

In the proof of Theorem \ref{reduction main theorem}, we used the following theorem.
\begin{theorem}\label{Theorem 2.1}
 For $2 \le p \le 3.2$ and $k=2,$ for each $\gamma \in \Gamma_{\alpha, R^{-1}}$ and any $\varepsilon  > 0,$ there exist positive constants $A = A\left( \varepsilon  \right)$ and $C\left( {K,\varepsilon, C_{\alpha} } \right)$ such that
\begin{equation}\label{Eq212}
  \left\| {{e^{it\Delta}}f(x+R\gamma(\frac{t}{R^{2}}))} \right\|_{BL_{k,A}^p{L^\infty }\left( {B\left( {0,R} \right) \times \left[ {0,R} \right]} \right)}^{} \le C\left( {K,\varepsilon, C_{\alpha} } \right){R^{\frac{2}{p} - \frac{5}{8} + \varepsilon }}{\left\| f \right\|_{{L^2}}},
\end{equation}
for all $R \ge 1,$ ${\xi _0} \in B\left( {0,1} \right), M^{-1}\geq R^{-1/2},$ all $f$ with $supp\hat{f} \subset B\left( {{\xi _0},{M^{ - 1}}} \right).$ The constant $C(K, \varepsilon, C_{\alpha})$ does not depend on the choice of $\gamma$.
\end{theorem}

As the similar argument in \cite{DL}, for any subset $U \subset B\left( {0,R} \right) \times \left[ {0,R} \right],$ we define
\begin{align}
&\left\| {{e^{it\Delta}}f(x+R\gamma(\frac{t}{R^{2}}))} \right\|_{BL_{k,A}^p{L^\infty }\left( U \right)}^{} \nonumber\\
&:= {\left( {\sum\limits_{{B_K} \subset B\left( {0,R} \right)} {\mathop {\sup }\limits_{I_K^j \subset \left[ {0,R} \right]} \frac{{\left| {U \cap \left( {{B_K} \times I_K^j} \right)} \right|}}{{\left| {{B_K} \times I_K^j} \right|}}{\mu _{{e^{it\Delta}}{f(x+R\gamma(\frac{t}{R^{2}}))}}}\left( {{B_K} \times I_K^j} \right)} } \right)^{1/p}}, \nonumber
\end{align}
which can be approximated by
\begin{align}
&\left\| {{e^{it\Delta}}f(x+R\gamma(\frac{t}{R^{2}}))} \right\|_{BL_{k,A}^p{L^q}\left( U \right)} \nonumber\\
:&= {\left( {{{\sum\limits_{{B_K} \subset B\left( {0,R} \right)} {\left[ {\sum\limits_{I_K^j \subset \left[ {0,R} \right]}^{} {{{\left( {\frac{{\left| {U \cap \left( {{B_K} \times I_K^j} \right)} \right|}}{{\left| {{B_K} \times I_K^j} \right|}}{\mu _{{e^{it\Delta}}{f(x+R\gamma(\frac{t}{R^{2}}))}}}\left( {{B_K} \times I_K^j} \right)} \right)}^q}} } \right]} }^{1/q}}} \right)^{1/p}}.\nonumber
\end{align}
This means,
\[\left\| {{e^{it\Delta}}f(x+R\gamma(\frac{t}{R^{2}}))} \right\|_{BL_{k,A}^p{L^\infty }\left( U \right)}^{} = \mathop {\lim }\limits_{q \to  + \infty } \left\| {{e^{it\Delta}}f(x+R\gamma(\frac{t}{R^{2}}))} \right\|_{BL_{k,A}^p{L^q }\left( U \right)}^{},\]
hence Theorem \ref{Theorem 2.1} can be deduced to prove Theorem \ref{Theorem3.1} below.

\begin{theorem}\label{Theorem3.1}
Let $2 \le p \le 3.2$ and $k=2$. For each $\gamma \in \Gamma_{\alpha, R^{-1}}$ and any $\varepsilon  > 0,$ $1 \le q <  + \infty,$ there exist positive constants $\bar{A} = \bar{A}\left( \varepsilon  \right)$ and $C\left( {K,\varepsilon, C_{\alpha} } \right)$ such that
\begin{equation}\label{main reduction estimate}
 \left\| {{e^{it\Delta}}f(x+R\gamma(\frac{t}{R^{2}}))} \right\|_{BL_{k,A}^p{L^q}\left( {B\left( {0,R} \right) \times \left[ {0,R} \right]} \right)}^{} \le C\left( {K,\varepsilon, C_{\alpha} } \right) R^{\delta(Log \bar{A} - Log A)}{R^{\frac{2}{p} - \frac{5}{8} + \varepsilon }}{\left\| f \right\|_{{L^2}}},
 \end{equation}
for all $1 \le A \le \bar{A}$, $R \ge 1,$ ${\xi _0} \in B\left( {0,1} \right),$ all $f$ with $supp \widehat{f} \subset B\left( {{\xi _0},{M^{ - 1}}} \right), M^{-1}\geq R^{-1/2}.$ The constants here are all independent of the choice of $\gamma$.
\end{theorem}

We  notice that the result in Theorem \ref{Theorem3.1} is translation invariance in both $x$ and $t$. This is very important for the induction on the scale $R$. The translation invariance property also plays a key role in the proof of the corresponding  result in the vertical case. We leave the proof of Theorem \ref{Theorem3.1} to Appendix,
because it is quite long and technical.


\section{Proofs of Theorem \ref{necessary condition} and Theorem \ref{upper bound1}}\label{proof of necessary condition}
In order to prove Theorem \ref{necessary condition} and Theorem \ref{upper bound1}, we first show the following lemma.
\begin{lemma}\label{equivalence}
Let  $\gamma(x,t)= x - e_{1}t^{\alpha}$, $e_{1}= (1,0)$, $\alpha \in [1/2,1)$. Then for $p \ge 1$ and $s \ge 0$,
\begin{equation}\label{tangential}
\biggl\|\sup_{0<t<R^{-1}}|e^{it\triangle}f(\gamma(x,t))|\biggl\|_{L^p(B(0,1))} \leq C \|f\|_{H^s}
\end{equation}
holds whenever $\hat{f}$ has compact support and  supp $\hat{f} \subset \{(\xi_{1},\xi_{2}): |\xi_{1}-\xi_{1}^{0}| \lesssim R^{1/2}\}$ if and only if
\begin{equation}\label{vertical}
\biggl\|\sup_{0<t<R^{-1}}|e^{it\triangle}f(x)|\biggl\|_{L^p(B(0,1))}\leq C \|f\|_{H^s},
\end{equation}
holds no matter when $\hat{f}$ has compact support and  supp $\hat{f} \subset \{(\xi_{1},\xi_{2}): |\xi_{1}-\xi_{1}^{0}| \lesssim R^{1/2}\}$.
\end{lemma}
\begin{proof}
We first show that if inequality (\ref{vertical}) holds whenever $\hat{f}$ has compact support and  $\text{supp} \hat{f} \subset \{(\xi_{1},\xi_{2}): |\xi_{1}-\xi_{1}^{0}| \lesssim R^{1/2}\}$, then inequality (\ref{tangential}) holds for such $f$. As the similar argument in \cite{LR}, we introduce a cut-off function $\phi$ which is real-valued, smooth and equal to $1$ on $(-2,2)$ and supported on $(-\pi, \pi)$.
Making a change of variables,
\[\xi_{1}\rightarrow R^{1/2}\eta_{1} + \xi_{1}^{0}, \quad \quad  \xi_{2} \rightarrow \eta_{2}, \]
one gets
\begin{align}
&e^{it\Delta}f(\gamma(x,t)) = \int_{\mathbb{R}^{2}}e^{ix \cdot \xi - it^{\alpha} \xi_{1} + it|\xi|^{2}}\hat{f}(\xi)d\xi = R^{1/2}e^{-it^{\alpha} \xi_{1}^{0}} \nonumber\\
 &\quad \times \int_{\mathbb{R}^{2}}e^{i x \cdot (R^{1/2}\eta_{1} + \xi_{1}^{0}, \eta_{2}) - iR^{1/2}t^{\alpha} \eta_{1} + it|(R^{1/2}\eta_{1} + \xi_{1}^{0}, \eta_{2})|^{2}} \phi(\eta_{1}) \hat{f}(R^{1/2}\eta_{1} + \xi_{1}^{0}, \eta_{2})d\eta_{1}d\eta_{2}.
\end{align}
Since for each $t \in (0,R^{-1})$ and $\alpha \in [1/2,1)$,
\[R^{1/2} t^{\alpha} \le 1,\]
then by Fourier expansion,
\[\phi(\eta_{1})e^{iR^{1/2}t^{\alpha} \eta_{1}} = \sum_{\mathfrak{l} \in \mathbb{Z}}{c_{\mathfrak{l}}( t)e^{i\mathfrak{l} \eta_{1}}}.\]
Integration by parts shows that
\[|c_{\mathfrak{l}}( t)| \le \frac{C}{(1+|\mathfrak{l}|)^{2}}\]
uniformly for each $\mathfrak{l} \in \mathbb{Z}$ and $t \in (0, R^{-1})$. Then we have
\begin{align}
&|e^{it\Delta}f(\gamma(x,t))|  \nonumber\\
&= \biggl| \sum_{\mathfrak{l} \in \mathbb{Z}}c_{\mathfrak{l}}( t) R^{1/2}e^{-it^{\alpha}\xi_{1}^{0}} \int_{\mathbb{R}^{2}}e^{i x \cdot (R^{1/2}\eta_{1} + \xi_{1}^{0}, \eta_{2}) - i\mathfrak{l} \eta_{1} + it|(R^{1/2}\eta_{1} + \xi_{1}^{0}, \eta_{2})|^{2}}  \hat{f}(R^{1/2}\eta_{1} + \xi_{1}^{0}, \eta_{2})d\eta_{1}d\eta_{2}\biggl|
\nonumber\\
&= \biggl| \sum_{\mathfrak{l} \in \mathbb{Z}}c_{\mathfrak{l}}( t) e^{-it^{\alpha} \xi_{1}^{0}} \int_{\mathbb{R}^{2}}e^{i x \cdot \xi - i\mathfrak{l} (\xi_{1}-\xi_{1}^{0})/R^{1/2} + it|\xi|^{2}}  \hat{f}(\xi_{1}, \xi_{2})d\xi_{1}d\xi_{2} \biggl|
\nonumber\\
&= \biggl| \sum_{\mathfrak{l} \in \mathbb{Z}}c_{\mathfrak{l}}( t) e^{-it^{\alpha} \xi_{1}^{0} + i\mathfrak{l}  \xi_{1}^{0}/R^{1/2}}  \int_{\mathbb{R}^{2}}e^{i x \cdot \xi + it|\xi|^{2}} e^{- i\mathfrak{l} \xi_{1}/R^{1/2} } \hat{f}(\xi_{1}, \xi_{2})d\xi_{1}d\xi_{2} \biggl|
\nonumber\\
&\le \sum_{\mathfrak{l} \in \mathbb{Z}}{\frac{C}{(1+|\mathfrak{l}|)^{2}} | e^{it\Delta}f_{R}^{\mathfrak{l}}(x) |}, \nonumber
\end{align}
where $\widehat{f_{R}^{\mathfrak{l}}}(\xi_{1}, \xi_{2})=e^{- i\mathfrak{l} \xi_{1}/R^{1/2} } \hat{f}(\xi_{1}, \xi_{2})$ and supp $\widehat{f_{R}^{\mathfrak{l}}} \subset \{(\xi_{1},\xi_{2}): |\xi_{1}-\xi_{1}^{0}| \lesssim R^{1/2}\}$. Therefore, applying inequality (\ref{vertical}) to get
\begin{align}
\biggl\|\sup_{0<t<R^{-1}}|e^{it\triangle}f(\gamma(x,t))|\biggl\|_{L^p(B(0,1))} &\le \sum_{\mathfrak{l} \in \mathbb{Z}}{\frac{C}{(1+|\mathfrak{l}|)^{2}}  \biggl\|\sup_{0<t<R^{-1}}|e^{it\triangle}f_{R}^{\mathfrak{l}}(x)|\biggl\|_{L^{p}(B(0,1))} } \nonumber\\
&\le \sum_{\mathfrak{l} \in \mathbb{Z}}{\frac{C}{(1+|\mathfrak{l}|)^{2}}  \|f_{R}^{\mathfrak{l}} \|_{H^{s}} } \nonumber\\
&\le C \|f\|_{H^{s}},
\end{align}
then we arrive at inequality (\ref{tangential}). By the same method, we can also prove  if inequality (\ref{tangential}) holds whenever $\hat{f}$ has compact support and  supp $\hat{f} \subset \{(\xi_{1},\xi_{2}): |\xi_{1}-\xi_{1}^{0}| \lesssim R^{1/2}\}$ for some $\xi_{1}^{0}$, then inequality (\ref{vertical}) holds also for such $f$. We omit its proof here.
\end{proof}

\textbf{Proof of Theorem \ref{necessary condition}.}
If inequality (\ref{tangential global}) holds for all $f \in H^{s}(\mathbb{R}^2)$, by Lemma \ref{equivalence}, inequality (\ref{vertical}) holds whenever $\hat{f}$ has compact support and  supp $\hat{f} \subset \{(\xi_{1},\xi_{2}): |\xi_{1}-\xi_{1}^{0}| \lesssim R^{1/2}\}$. Bourgain \cite{B} actually showed that there exists $f_{R}$,
\[\hat{f_{R}}(\xi) = \chi_{A_{R}}(\xi),\]
where $A_{R}$ is the subset of $\{\xi \in \mathbb{R}^{2}: |\xi| \sim R\}$ defined by
\[A_{R}= \bigcup_{l \in \mathbb{N}^{+}, l \sim R^{1/3}}A_{R,l},\]
\[A_{R,l}= [R-R^{1/2}, R+R^{1/2}] \times [R^{2/3}l, R^{2/3}l+1].\]
Here $\chi_{A_{R}}(\xi)$ is the characteristic function of the set $A_{R}$.
One can find a set $S$ with positive measure such that for each $x \in S$, there exists $t$ with $0< t < R^{-1}$,
\begin{equation}
\biggl| e^{it\Delta}f_{R}(x) \biggl| \ge R^{3/4}.
\end{equation}
Hence,
\begin{equation}\label{Eq2.2.3}
\mathop{sup}_{0<t<R^{-1}} \biggl| e^{it\Delta}f(x)\biggl| \ge R^{3/4}.
\end{equation}
It is obvious that supp $\hat{f_{R}}  \subset \{(\xi_{1},\xi_{2}): |\xi_{1}-R| \lesssim R^{1/2}\}$, applying inequality (\ref{vertical}) to $f_{R}$, and we obtain that
\[R^{3/4} \lesssim R^{s} R^{5/12}.\]
Finally we get $s \ge 1/3$ since $R$ can be sufficiently large.
\hfill$\Box$

\textbf{Proof of Theorem \ref{upper bound1}.}
 If inequality (\ref{tangential large p}) holds for some $0 <s <1/2$, we get inequality (\ref{vertical})
whenever  $\hat{f}$ has compact support and  supp $\hat{f} \subset \{(\xi_{1},\xi_{2}): |\xi_{1}-\xi_{1}^{0}| \lesssim R^{1/2}\}$. According to the proof of \cite [Theorem 5]{S3}, there is $f_{R}$ with supp $\hat{f_{R}} \subset (-R-R^{1/2}, -R+R^{1/2})\times (-R-R^{1/2}, -R+R^{1/2})$ and
\[\|f_{R}\|_{H^{s}} \le C R^{s-1/2}.\]
For each $x \in \{(x_{1},x_{2}): x_{1} \in I, |x_{1}-x_{2}| \le \delta R^{-1/2} \}$, there exists $t \in (0,R^{-1})$ such that
\[|e^{it\triangle}f(x)| \ge c,\]
here $I$ is a small interval around the origin and $\delta$, $c$ are small positive numbers. Combining this argument with inequality (\ref{vertical}), we have
\[R^{-1/2p} \le C R^{s-1/2},\]
which implies
\[p \le \frac{1}{1-2s},\]
since $R$ can be sufficiently large.

\section{Proof of Theorem \ref{nontangential maximal theorem}}\label{proof of nontangential}
\begin{proof}
 Using Littlewood-Paley decomposition, we only need to show that for $f$ with supp$\hat{f} \subset \{\xi \in \mathbb{R}: |\xi| \sim \lambda\}$, $\lambda \gg 1$,
 \begin{equation}\label{Eq2.1}
\biggl\|\mathop{\rm{sup}}_{(t,\theta) \in (0,1) \times \Theta} |e^{it\Delta}f(\gamma(x,t,\theta))|\biggl\|_{L^{p}(B(x_{0},r))} \le C\lambda^{s_{0}+\varepsilon} \|f\|_{L^{2}}, \:\ \forall \varepsilon >0,
\end{equation}
where $s_{0}$ and $p$ are chosen as in Theorem \ref{nontangential maximal theorem}.

We decompose $\Theta$ into small subsets $\{\Theta_{k}\}$ such that $\Theta= \cup_{k}\Theta_{k}$ with bounded overlap, where each $\Theta_{k}$ is contained in a closed ball with  diameter $\lambda^{-\mu}, \mu = \min \{1, 2\alpha\}$.
 Due to the definition of $\beta(\Theta)$, we have
\begin{equation}\label{Eq2.2}
1 \le k \le \lambda^{ \mu \beta(\Theta) + \varepsilon}.
\end{equation}
We claim that
 \begin{equation}\label{uniform estimate}
\mathop{\rm{sup}}_{k}  \biggl\|\mathop{\rm{sup}}_{(t,\theta) \in (0,1) \times \Theta_{k}} |e^{it\Delta}f(\gamma(x,t,\theta))|\biggl\|_{L^{p}(B(x_{0},r))} \le C\lambda^{\nu+\frac{(p-1)\varepsilon}{p}} \|f\|_{L^{2}},
\end{equation}
where $\nu=\max\{\frac{1}{2}-\alpha, \frac{1}{4}\}$. Then
\begin{align}
\biggl\|\mathop{\rm{sup}}_{(t,\theta) \in (0,1) \times \Theta} |e^{it\Delta}f(\gamma(x,t,\theta))|\biggl\|_{L^{p}(B(x_{0},r))}
&\le \biggl(\sum_{k} \biggl\|\mathop{\rm{sup}}_{(t,\theta) \in (0,1) \times \Theta_{k}} |e^{it\Delta}f(\gamma(x,t,\theta))|\biggl\|_{L^{p}(B(x_{0},r))}^{p}  \biggl)^{1/p} \nonumber\\
&\le C\biggl(\sum_{k} \lambda^{p\nu+(p-1)\varepsilon} \|f\|_{L^{2}}^{p} \biggl)^{1/p} \nonumber\\
&\le  C\lambda^{\frac{ \mu \beta{(\Theta)}}{p} + \nu +\varepsilon} \|f\|_{L^{2}}, \nonumber
\end{align}
which implies inequality (\ref{Eq2.1}).

Now let's turn to prove inequality (\ref{uniform estimate}). In fact, inequality  (\ref{uniform estimate}) comes from the following lemma.

\begin{lemma}\label{lemma2.3}
Under the assumption of Theorem \ref{nontangential maximal theorem},  and $f$ is a Schwartz function whose Fourier transform is supported in the annulus $A_{\lambda}=\{\xi \in \mathbb{R}: |\xi| \sim \lambda\}$. Then for each $k$,
\begin{align}\label{Eq2.1n}
\biggl\|\mathop{\rm{sup}}_{t \in (0,1), \theta \in \Theta_{k}} |e^{it\Delta}f(\gamma(x,t,\theta))|\biggl\|_{L^{p}(B(x_{0},r))} \le C \lambda^{\nu} \|f\|_{L^{2}},
\end{align}
where $p$ is chosen as in Theorem \ref{nontangential maximal theorem} and $\nu = \max \{1/2-\alpha, 1/4\}$.
Moreover, the constant $C$ in inequality (\ref{Eq2.1n}) depends on   $C_{1},C_{2}, C_{3}$, $\Theta$ and the choice of $B(x_{0},r)$, but does not depend on $f$ and $k$.
\end{lemma}

Next we will prove Lemma \ref{lemma2.3}. By the Kolmogorov-Seliverstov-Plessner linearization, we choose $t(x) \in (0,1)$, $\theta(x) \in \Theta_{k} $ to be  measurable functions defined on $B(x_{0},r)$, such that
\[\sup_{t\in (0,1), \theta \in \Theta_{k}} \biggl|e^{it\triangle}f(\gamma(x,t, \theta))\biggl| \le 2 \biggl|e^{it(x)\triangle}f(\gamma(x,t(x), \theta(x)))\biggl| . \]
Set
\[Tf(x):= \int_{A_{\lambda}}e^{i\gamma(x,t(x),\theta(x))\xi+it(x)\xi^{2}}f(\xi)d\xi .\]
It is sufficient to show that
\begin{equation*}
\|Tf\|_{L^p(B(x_{0},r))}\leq C \lambda^{\nu} \|f\|_{L^2(A_{\lambda})}
\end{equation*}
holds for all $f$ with supp $f \subset A_{\lambda}$. Notice that  we used the Plancherel's theorem here to replace $\hat{f}$ by $f$.  By duality, this is equivalent to show
\begin{equation*}
\|T^*g\|_{L^2(A_{\lambda})}\leq C \lambda^{\nu} \|g\|_{L^{p^{\prime}}(B(x_{0},r))}, \quad \quad 1/p + 1/p^{\prime}=1,
\end{equation*}
holds for all $g \in L^{p^{\prime}}(B(x_{0},r))$, where
\[T^*g(\xi):= \int_{B(x_{0},r)}e^{-i\gamma(x,t(x), \theta(x))\xi-it(x)\xi^{2}}g(x)dx.\]
We choose a real-valued cutoff function $\phi$ which is equal to $1$ on $A_{1}$ and rapidly decay outside.
Then
\begin{align}
&\|T^*g\|_{L^2(A_{\lambda})}^{2}\nonumber\\
&=\int_{B(x_{0},r)}\int_{B(x_{0},r)}g(x) \bar{g}(y)  \int_{A_{\lambda}} e^{i\gamma(y,t(y),\theta(y)) \xi-i\gamma(x,t(x),\theta(x))\xi+it(y) \xi^{2}-it(x)\xi^{2}} \phi(\xi/\lambda) d\xi dx dy \nonumber\\
&=\int_{B(x_{0},r)}\int_{B(x_{0},r)}g(x) \bar{g}(y) K(x,y) dx dy. \nonumber
\end{align}
Here
\[K(x,y):=\int_{A_{\lambda}} e^{i\gamma(y,t(y), \theta(y)) \xi-i\gamma(x,t(x), \theta(x))\xi+it(y) \xi^{2}-it(x)\xi^{2}} \phi(\xi/\lambda) d\xi.\]

We have the following kernel estimates: \\
\textbf{(E1)} for each $x,y \in B(x_{0},r)$, $ |K(x,y)| \le \lambda.$ \\
\textbf{(E2)} for each $x,y \in B(x_{0},r)$ and $x \neq y$, if  $|t(x)-t(y)| \ge 5 (C_{1}r+ C_{2} +C_{3} \text{diam}(\Theta)) \lambda^{-1} $, then
\[|K(x,y)| \le C \lambda^{-100}.\]
\textbf{(E3)} for each $x,y \in B(x_{0},r)$ and $x \neq y$, if  $|t(x)-t(y)| < 5 (C_{1}r+ C_{2} +C_{3} \text{diam} (\Theta)) \lambda^{-1} $ and $|x-y| \ge 2C_{1}C_{3} \text{diam}(\Theta_{k})$, then
\[|K(x,y)| \le C \max\{\frac{\lambda^{1/2}}{|x-y|^{1/2}},|x-y|^{-1/2\alpha}\}.\]
We also remark that the constant $C$ in (E2) and (E3) depends only on $C_{1}, C_{2}, C_{3}$, $r$ and $\Theta$.

(E1) is trivial so we will only prove (E2) and (E3) by  stationary phase method. By rescaling,
\[K(x,y) = \lambda\int_{A_{1}} e^{i\lambda [\gamma(y,t(y), \theta(y)) \eta-\gamma(x,t(x), \theta(x))\eta+ \lambda t(y) \eta^{2}-\lambda t(x)\eta^{2}]} \phi(\eta) d\eta.\]
Denote
\[\psi (x,y, \eta): = \gamma(y,t(y),\theta(y)) \eta-\gamma(x,t(x),\theta(x))\eta+ \lambda t(y) \eta^{2}-\lambda t(x)\eta^{2}.\]
Then
\[\frac{\partial}{\partial \eta} \psi (x,y,\eta)= \gamma(y,t(y),\theta(y)) -\gamma(x,t(x),\theta(x)) + 2\lambda t(y) \eta-2\lambda t(x)\eta,\]
\[\frac{\partial^{2}}{\partial \eta^{2}} \psi (x,y,\eta)= 2\lambda t(y) -2\lambda t(x).\]

We first prove (E2). Note that
\begin{align}
&\biggl|\gamma(y,t(y),\theta(y)) -\gamma(x,t(x),\theta(x))\biggl| \nonumber\\
&=\biggl|\gamma(y,t(y),\theta(y)) -\gamma(x,t(y),\theta(y)) +\gamma(x,t(y),\theta(y))-\gamma(x,t(x),\theta(y)) \nonumber\\
& \quad \quad +\gamma(x,t(x),\theta(y))-\gamma(x,t(x),\theta(x)) \biggl| \nonumber\\
&\le C_{1} |x-y| +C_{2} |t(x)-t(y)| +C_{3} |\theta(x) -\theta(y)| \nonumber\\
&\le 2C_{1}r+ C_{2} +C_{3} \text{diam}(\Theta).
\end{align}
Therefore, if $|t(x)-t(y)| \ge 5 (C_{1}r+ C_{2} +C_{3} \text{diam} (\Theta)) \lambda^{-1} $, then integration by parts implies
(E2), and the constant $C$ depends only on $C_{1}, C_{2}, C_{3}$, $r$ and $\Theta$.

Next we prove (E3). By the triangle inequality,
\begin{align}
&\Biggl|\gamma(y,t(y),\theta(y)) -\gamma(x,t(x),\theta(x))\biggl| \nonumber\\
& \ge \Biggl| \biggl| \gamma(y,t(y),\theta(y)) -\gamma(x,t(y),\theta(y))
+\gamma(x,t(x),\theta(y))-\gamma(x,t(x),\theta(x)) \biggl| \nonumber\\
 & \quad \quad -\biggl| \gamma(x,t(y),\theta(y))-\gamma(x,t(x),\theta(y))\biggl| \Biggl|. \nonumber
\end{align}
According to the assumption of (E3), $|x-y| \ge 2C_{1}C_{3} \text{diam} (\Theta_{k})$, then
\begin{align}
&\biggl| \gamma(y,t(y),\theta(y)) -\gamma(x,t(y),\theta(y)) +\gamma(x,t(x),\theta(y))-\gamma(x,t(x),\theta(x)) \biggl| \nonumber\\
&\le C_{1}|x-y|+ C_{3} \text{diam}(\Theta_{k}) \nonumber\\
&\le (C_{1}+\frac{1}{2C_{1}})|x-y|, \nonumber
\end{align}
and
\begin{align}
&\biggl| \gamma(y,t(y),\theta(y)) -\gamma(x,t(y),\theta(y)) +\gamma(x,t(x),\theta(y))-\gamma(x,t(x),\theta(x)) \biggl| \nonumber\\
&\ge C_{1}^{-1}|x-y|- C_{3} \text{diam}(\Theta_{k}) \nonumber\\
&\ge (2C_{1})^{-1}|x-y|. \nonumber
\end{align}

If  $ |x-y| \le 100 C_{1}C_{2} |t(x)-t(y)|^{\alpha}$,
\[ \biggl|\frac{\partial^{2}}{\partial \eta^{2}} \psi (x, y, \eta) \biggl| = 2 \lambda |t(x)-t(y)| \ge 2\lambda (100C_{1}C_{2})^{-1/\alpha} |x-y|^{1/\alpha}.\]
Using Van der Corput's lemma, we have
\[|K(x,y)| \le 2^{-1/2} ( 100C_{1}C_{2})^{1/2\alpha}|x-y|^{-1/2\alpha}. \]
If $|x-y| > 100 C_{1}C_{2} |t(x)-t(y)|^{\alpha}$ and $|x-y| > 100 C_{1} \lambda |t(x)-t(y)|$,
\[ \biggl|\frac{\partial}{\partial \eta} \psi (x, y, \eta)\biggl|  \ge  (10 C_{1} )^{-1} |x-y|,\]
and
\[ \biggl|\frac{\partial^{2}}{\partial \eta^{2}} \psi (x, y, \eta) \biggl| = 2 \lambda |t(x)-t(y)| \le 10 (C_{1}r+ C_{2} +C_{3} \text{diam}(\Theta)).\]
Integration by parts implies
\[|K(x,y)| \le \frac{C_{N} \lambda}{(1+ \lambda (10 C_{1} )^{-1} |x-y|)^{N}}.\]
If $|x-y| > 100 C_{1}C_{2} |t(x)-t(y)|^{\alpha}$ and $|x-y| \le 100 C_{1} \lambda |t(x)-t(y)|$,
\[ \biggl|\frac{\partial^{2}}{\partial \eta^{2}} \psi (x, y, \eta) \biggl| = 2 \lambda |t(x)-t(y)| \ge  (50C_{1})^{-1} |x-y|.\]
Applying Van der Corput's lemma to get
\[|K(x,y)| \le (50C_{1})^{1/2} \frac{\lambda^{1/2}}{|x-y|^{1/2}}.\]
We finish the estimate of (E3).

According to the kernel estimate, we break $B(x_{0},r) \times B(x_{0},r)$ into $ \Omega_{1}, \Omega_{2}$, where
\[\Omega_{1}:=\{(x,y) \in B(x_{0},r) \times B(x_{0},r): |t(x) -t(y)| \le 5 (C_{1}r+ C_{2} +C_{3} \text{diam} (\Theta)) \lambda^{-1} \},\]
\[\Omega_{2}:=\{(x,y) \in  B(x_{0},r) \times B(x_{0},r): |t(x) -t(y)| > 5 (C_{1}r+ C_{2} +C_{3} \text{diam} (\Theta)) \lambda^{-1} \}.\]
By (E2), we have
\begin{align}
& \biggl| \int \int_{ \Omega_{2}}g(x) \bar{g}(y) K(x,y) dx dy \biggl|  \le C \lambda^{-100} \|g\|_{L^{p^{\prime}}(B(x_{0},r))}^{2}, \nonumber
\end{align}
here the constant $C$ depends on $C_{1}, C_{2}, C_{3}$, $B(x_{0},r)$ and $\Theta$. To establish the estimate on $\Omega_{1}$, we will consider the following three cases, $\alpha \in [1/2, 1)$,  $\alpha \in (1/4, 1/2)$ and  $\alpha \in (0, 1/4]$, respectively.

\textbf{Case 1}. When $\alpha \in [1/2, 1)$. Note that we have $\text{diam} (\Theta_{k}) = \lambda^{-1} $.  If $(x,y) \in \Omega_{1}$ and $|x-y| \ge 2C_{1}C_{3} \lambda^{-1}$,
\[\lambda^{1/2} |x-y|^{1/2\alpha-1/2} \ge \lambda^{1-1/2\alpha} \ge 1.\]
Hence,
\begin{equation}
|K(x,y)| \le  \frac{C\lambda^{1/2}}{|x-y|^{1/2}}.
\end{equation}
Then we have
\begin{align}
&\biggl| \int \int_{ \Omega_{1}}  g(x) \bar{g}(y) K(x,y) dx dy \biggl|  \nonumber\\
&\le \int \int_{\{(x,y) \in \Omega_{1}: |x-y| < 2C_{1}C_{3} \lambda^{-1}\}} | g(x) \bar{g}(y) K(x,y) | dxdy \nonumber\\
& \quad  + \int \int_{\{(x,y) \in \Omega_{1}: |x-y| \ge 2C_{1}C_{3} \lambda^{-1}\}} | g(x) \bar{g}(y) K(x,y)| dxdy \nonumber\\
&\le \lambda \int \int_{\{(x,y)\in \Omega_{1}: |x-y| <2C_{1}C_{3} \lambda^{-1}\}}  |x-y|^{1/2}  |x-y|^{-1/2} |g(x) \bar{g}(y) | dxdy \nonumber\\
&\quad + C\lambda^{1/2} \int \int_{\{(x,y)\in \Omega_{1}: |x-y| \ge 2C_{1}C_{3}\lambda^{-1}\}} |g(x) \bar{g}(y)| |x-y|^{-1/2}dxdy \nonumber\\
&\le C\lambda^{1/2} \int_{\mathbb{R}}\int_{\mathbb{R}} |g(x)| \chi_{B(x_{0},r)}(x) |\bar{g}(y)| \chi_{B(x_{0},r)}(y) |x-y|^{-1/2} dx dy  \nonumber\\
&\le C\lambda^{1/2} \|g\|_{L^{4/3}(B(x_{0},r))}\biggl\||g| \chi_{B(x_{0},r)} \ast |\cdot|^{-1/2}\biggl\|_{L^{4}(\mathbb{R})} \nonumber\\
&\le C\lambda^{1/2} \|g\|_{L^{4/3}(B(x_{0},r))}^{2}.
\end{align}
Here we applied the HLS inequality
\begin{equation}\label{HLS}
\biggl\||g| \chi_{B(x_{0},r)} \ast |\cdot|^{-\gamma}\biggl\|_{L^{p_{1}}} \le A_{p_{1},p_{2}} \|g \chi_{B(x_{0},r)} \|_{L^{p_{2}}},
\end{equation}
$1 <p_{1}, p_{2}< \infty$, $1/p_{1}=1/p_{2}-1+\gamma$, $0 < \gamma <1$.

\textbf{Case 2}. When $\alpha \in (1/4, 1/2)$. Notice that we now have $\text{diam}(\Theta_{k})= \lambda^{-2\alpha}$. For each $(x,y) \in \Omega_{1}$,
\[|x-y| \le 2r,\]
which implies
\[\frac{\lambda^{1/2}}{|x-y|^{1/2}} \le C\frac{\lambda^{1/2}}{|x-y|^{1/4\alpha}},\]
where the constant $C$ depends on $r$.
So for each $x,y \in B(x_{0},r)$ satisfying $|x-y| \ge 2C_{1}C_{3} \lambda^{-2\alpha}$,
\begin{equation}
|K(x,y)| \le C \max\{\frac{\lambda^{1/2}}{|x-y|^{1/2}},|x-y|^{-1/2\alpha}\} \le C \max\{\frac{\lambda^{1/2}}{|x-y|^{1/4\alpha}},|x-y|^{-1/2\alpha}\},
\end{equation}
where the constant $C \ge 1$ depends on $C_{1}$, $C_{2}$ and $r$. When  $|x-y| \ge 2C_{1}C_{3} \lambda^{-2\alpha} $,
\[\frac{C\lambda^{1/2} |x-y|^{1/2\alpha}}{|x-y|^{1/4\alpha}} = C\lambda^{1/2} |x-y|^{1/4\alpha} \ge 1.\]
Thus
\begin{equation}
|K(x,y)| \le \frac{C\lambda^{1/2}}{|x-y|^{1/4\alpha}}.
\end{equation}
Then we have
\begin{align}
&\biggl| \int \int_{ \Omega_{1}}g(x) \bar{g}(y) K(x,y) dx dy \biggl|  \nonumber\\
&\le \int \int_{\{(x,y) \in \Omega_{1}: |x-y| < 2C_{1}C_{3}\lambda^{-2 \alpha}\}} |g(x) \bar{g}(y) K(x,y)| dxdy \nonumber\\
 &\quad  + \int \int_{\{(x,y) \in \Omega_{1}: |x-y| \ge 2C_{1}C_{3}\lambda^{-2\alpha}\}} |g(x) \bar{g}(y) K(x,y)| dxdy \nonumber\\
&\le \lambda \int \int_{\{(x,y) \in \Omega_{1}: |x-y| < 2C_{1}C_{3} \lambda^{-2\alpha}\}}  |x-y|^{1/4\alpha}  |x-y|^{-1/4\alpha} |g(x) \bar{g}(y)| dxdy \nonumber\\
&\quad + C\lambda^{1/2} \int \int_{\{(x,y) \in \Omega_{1}: |x-y| \ge 2C_{1}C_{3} \lambda^{-2\alpha}\}} |g(x) \bar{g}(y) | |x-y|^{-1/4\alpha}dxdy \nonumber\\
&\le C\lambda^{1/2} \int_{\mathbb{R}}\int_{\mathbb{R}} \chi_{B(x_{0},r)}(x) |g(x)| \chi_{B(x_{0},r)}(y) |\bar{g}(y) | \frac{1}{|x-y|^{1/4\alpha}} dx dy  \nonumber\\
&\le C\lambda^{1/2} \|g\|_{L^{8\alpha/(8\alpha-1)}(B(x_{0},r))}\biggl\||g |\chi_{B(x_{0},r)} \ast |\cdot|^{-1/4\alpha}\biggl\|_{L^{8\alpha}(B(x_{0},r))} \nonumber\\
&\le C\lambda^{1/2} \|g\|_{L^{8\alpha/(8\alpha-1)}(B(x_{0},r))}^{2}.
\end{align}
Here we applied the HLS inequality (\ref{HLS}).

\textbf{Case 3}. When $\alpha \in (0, 1/4]$, similar proof with the previous discussion in Case 1 and Case 2, we divide
\begin{align}
\biggl| \int \int_{ \Omega_{1}}g(x) \bar{g}(y) K(x,y) dx dy \biggl| &\le \int \int_{\{(x,y) \in \Omega_{1}: |x-y| < 2C_{1}C_{3} \lambda^{-2 \alpha}\}} |g(x) \bar{g}(y) K(x,y)|dxdy \nonumber\\
& + \int \int_{\{(x,y) \in \Omega_{1}: 2C_{1}C_{3} \lambda^{-2\alpha} \le |x-y| < \lambda^{-\alpha/(1-\alpha)}\}} |g(x) \bar{g}(y) K(x,y)| dxdy \nonumber\\
&+ \int \int_{\{(x,y) \in \Omega_{1}:  |x-y| \ge \lambda^{-\alpha/(1-\alpha)}\}}| g(x) \bar{g}(y) K(x,y)| dxdy.
\end{align}
Let's  estimate these three terms respectively. For the first term, by H\"{o}lder's inequality and $L^{2}$-estimate for the Hardy-Littlewood maximal function, we have
\begin{align}
&\int \int_{\{(x,y) \in \Omega_{1}: |x-y| < 2C_{1}C_{3} \lambda^{-2 \alpha}\}} |g(x) \bar{g}(y) K(x,y)|dxdy \nonumber\\
&\le C\lambda^{1-2\alpha} \int  M(|g| \chi_{B(x_{0},r)})(y) |\bar{g}(y)| \chi_{B(x_{0},r)}(y) dy \nonumber\\
&\le C\lambda^{1-2\alpha} \|M(|g| \chi_{B(x_{0},r)})\|_{L^{2}} \|g\|_{L^{2}(B(x_{0},r))} \nonumber\\
&\le C\lambda^{1-2\alpha}  \|g\|_{L^{2}(B(x_{0},r))}^{2}.
\end{align}
Let's turn to evaluate the second term. By H\"{o}lder's inequality and Schur's lemma,
\begin{align}
 &\int \int_{\{(x,y) \in \Omega_{1}: 2C_{1}C_{3}\lambda^{-2\alpha} \le |x-y| < \lambda^{-\alpha/(1-\alpha)}\}} |g(x) \bar{g}(y) K(x,y)| dxdy \nonumber\\
 &\le C \biggl\|\int_{\{x \in \mathbb{R}: 2C_{1}C_{3}\lambda^{-2\alpha} \le |x-y| < \lambda^{-\alpha/(1-\alpha)}\}}  |x-y|^{-1/2\alpha}\chi_{B(x_{0},r)}(x)|g(x)|dx \biggl\|_{L^{2} (B(x_{0},r))} \|g\|_{L^{2}(B(x_{0},r))} \nonumber\\
 &\le C \lambda^{1-2\alpha}  \|g\|_{L^{2}(B(x_{0},r))}^{2}.
\end{align}
For the last term, we can apply the HLS inequality and H\"{o}lder's inequality to obtain
\begin{align}
&\int \int_{\{(x,y) \in \Omega_{1}:  |x-y| \ge \lambda^{-\alpha/(1-\alpha)}\}}| g(x) \bar{g}(y) K(x,y)|dxdy \nonumber\\
&\le C\lambda^{1/2} \int_{\mathbb{R}}\int_{\mathbb{R}} |g(x)| \chi_{B(x_{0},r)}(x) |\bar{g}(y)| \chi_{B(x_{0},r)}(y) |x-y|^{-1/2} dx dy  \nonumber\\
&\le C\lambda^{1/2} \|g\|_{L^{4/3}(B(x_{0},r))}\biggl\||g| \chi_{B(x_{0},r)} \ast |\cdot|^{-1/2}\biggl\|_{L^{4}(\mathbb{R})} \nonumber\\
&\le C\lambda^{1/2} \|g\|_{L^{4/3}(B(x_{0},r))}^{2} \nonumber\\
&\le C\lambda^{1/2} \|g\|_{L^{2}(B(x_{0},r)}^{2}.
\end{align}
Notice that $g \in L^{2}(B(x_{0},r))$ here, and $1-2\alpha \ge 1/2$ since $\alpha \in (0,1/4]$.
 \end{proof}

\section{Proof of Theorem \ref{necessary bound for p}}\label{Necessary}

\textbf{Proof of Theorem \ref{necessary bound for p}.}
 The original idea of this proof comes from \cite[Proposition 1.5]{CLV}. Put
 \[\hat{f}(\xi)=\chi_{B(0,\lambda^{1/2})}(\xi).\]
 Then
 \[\|f\|_{H^{s}(\mathbb{R})} \le \lambda^{1/4} \lambda^{s/2}.\]
 By rescaling,
 \[|e^{it\Delta}f(\gamma(x,t))|= \lambda^{1/2} \biggl| \int_{B(0,1)}e^{i\lambda^{1/2}(x-t^{\alpha})\eta + i \lambda t \eta^{2}} d\eta \biggl|. \]
 If $t \in (0, \frac{1}{100} \lambda^{-1})$ and $x \in S:=\{y: |y-t^{\alpha}| \le \frac{1}{100} \lambda^{-1/2} \text{ for some } t \in (0, \frac{1}{100} \lambda^{-1})\}$, then
 \[|\lambda^{1/2}(x-t^{\alpha})\eta +  \lambda t \eta^{2}| \le 1/50\]
 and
 \[|e^{it\Delta}f(\gamma(x,t))| \ge \lambda^{1/2}.\]

 When $\alpha \in [1/2,1) $, we have $|S| \sim \lambda^{-1/2}$ and it follows from inequality (\ref{special Lp estimate}) that
 \[\lambda^{1/2-1/2p} \lesssim \lambda^{1/4+s/2}. \]
Apparently, $p$ can not be larger than $4$ when $s$ is sufficiently close to $1/4$, since $\lambda$ can be sufficiently large.

  When $\alpha \in (0,1/2) $, we get $|S| \sim \lambda^{-\alpha}$ and it follows from inequality (\ref{special Lp estimate}) that
 \[\lambda^{1/2-\alpha/p} \lesssim \lambda^{1/4+s/2}. \]
 It is easy to see that if $\alpha \in (1/4, 1/2)$, $p$ can not be larger than $8\alpha$ when $s$ is sufficiently close to $1/4$, since $\lambda$ can be sufficiently large. By the same reason, if $\alpha \in (0,1/4]$, $p$ can not be larger than $2$ when $s$ is sufficiently close to $1/2-\alpha$.
\hfill$\Box$

\section{Appendix}\label{Section 2}
In this appendix, we will prove Theorem \ref{Theorem3.1}. The original idea comes from \cite{Guth2} and \cite{DL}. The proof here looks long and technical, but we write most of the details for completeness.

\subsection{Wave packets decomposition}\label{Section of wave packets ddecomposition}
 We first introduce the wave packets decomposition for $f$.  Let $\varphi $ be a Schwartz function from $\mathbb{R}$ to $\mathbb{R}$, $\hat{\varphi} $ is non-negative and supported in a small neighborhood of the origin, and identically equal to $1$ in another smaller interval. Denote by $\theta  = \prod\nolimits_{j = 1}^2 {\theta _j} $ the rectangle in the frequency space with center $\left( {c\left( {{\theta _1}} \right),c\left( {{\theta _2}} \right)} \right)$ and
\[{\widehat{\varphi _\theta} }\left( {{\xi _1},{\xi _2}} \right) = \prod\limits_{j = 1}^2 {\frac{1}{{{{\left| {{\theta _j}} \right|}^{1/2}}}}\widehat\varphi } \left( {\frac{{{\xi _j} - c\left( {{\theta _j}} \right)}}{{\left| {{\theta _j}} \right|}}} \right).\]
A rectangle $\nu $ in the physical space is said to be dual to $\theta $ if $\left| {{\theta _j}} \right|\left| {{\nu _j}} \right| = 1,j = 1,2,$ and $\left( {\theta ,\nu } \right)$ is said to be a tile. Let $\boldsymbol{\rm T}$ be a collection of all tiles with fixed dimensions and coordinate axes. Define
\[\widehat{{\varphi}_{\theta, \nu}} \left( \xi  \right) = {e^{ - ic\left( \nu  \right) \cdot \xi}}{\widehat{\varphi _\theta }}\left( \xi  \right).\]
It is well-known that a Schwartz function $f$ can be decomposed by
\[f = \sum_{(\theta, \nu) \in \boldsymbol{\rm T}} \langle f, \varphi_{\theta, \nu}  \rangle \varphi_{\theta, \nu},\]
  and
\[\|f\|_{L^{2}}^{2} \sim \sum_{(\theta, \nu) \in \boldsymbol{\rm T}}  |\langle f, \varphi_{\theta, \nu}\rangle|^{2}. \]

Define $\tilde{\varphi}_{\theta, \nu}$ whose Fourier transform is given by
\[\widehat{\tilde{{\varphi}}_{\theta, \nu}} \left( \xi  \right) = e^{-iR\gamma(0) \cdot \xi}\widehat{{\varphi}_{\theta, \nu}} \left( \xi  \right) .\]
We claim that
\begin{equation}\label{wave packet decomposition}
f = \sum\limits_{\left( {\theta ,\nu } \right) \in \boldsymbol{\rm T}} {{f_{\theta ,\nu }}}  = \sum\limits_{\left( {\theta ,\nu } \right) \in \boldsymbol{\rm T}} {\left\langle {f,{\tilde{\varphi}_{\theta ,\nu }}} \right\rangle } {\tilde{\varphi }_{\theta ,\nu }}.
\end{equation}

Indeed, by Plancherel's theorem, for each Schwartz function $f$,
\[\hat{f}(\xi)=  \sum_{(\theta, \nu) \in \boldsymbol{\rm T}} \langle \hat{f}, \widehat{\tilde{\varphi}_{\theta, \nu} }  \rangle \widehat{\tilde{\varphi}_{\theta, \nu} }(\xi). \]
Replacing $\hat{f}$ by $e^{iR\gamma(0) \cdot \xi} \hat{f}$, we have
\begin{align}
e^{iR\gamma(0) \cdot \xi}\hat{f}(\xi)= \sum_{(\theta, \nu) \in \boldsymbol{\rm T}} \langle e^{iR\gamma(0) \cdot \xi} \hat{f}, \widehat{\tilde{\varphi}_{\theta, \nu} }  \rangle \widehat{\tilde{\varphi}_{\theta, \nu} }(\xi) = \sum_{(\theta, \nu) \in \boldsymbol{\rm T}} \langle \hat{f}, e^{-iR\gamma(0) \cdot \xi}\widehat{\tilde{\varphi}_{\theta, \nu} }  \rangle \widehat{\tilde{\varphi}_{\theta, \nu} }(\xi).
\end{align}
Then
\begin{align}
\hat{f}(\xi)= \sum_{(\theta, \nu) \in \boldsymbol{\rm T}} \langle \hat{f}, e^{-iR\gamma(0) \cdot \xi}\widehat{\tilde{\varphi}_{\theta, \nu} }  \rangle e^{-iR\gamma(0) \cdot \xi} \widehat{\tilde{\varphi}_{\theta, \nu} }(\xi).
\end{align}
The claim follows by applying the Plancherel's theorem again. And it is not hard to check that
\[\|f\|_{L^{2}}^{2} \sim \sum_{(\theta, \nu) \in \boldsymbol{\rm T}}  |\langle f, \tilde{\varphi}_{\theta, \nu}\rangle|^{2}. \]

Next, we consider the localization of $e^{it\Delta} \tilde{\varphi}_{\theta, \nu}(x+R\gamma(\frac{t}{R^{2}})) $ in $B(0,R) \times [0,R]$. In fact,
\begin{align}
 &e^{it\Delta} \tilde{\varphi}_{\theta, \nu}(x+R\gamma(\frac{t}{R^{2}})) \nonumber\\
 &= \frac{1}{{\sqrt R }}\int_{\mathbb{R}^2} {{e^{i\left( {\left( {x - c\left( \nu  \right)} \right) \cdot \left( {{R^{ - 1/2}}\eta  + c\left( \theta  \right)} \right) +[R\gamma(\frac{t}{R^{2}})-R\gamma(0)]  \cdot \left( {{R^{ - 1/2}}\eta  + c\left( \theta  \right)} \right) + t{{\left| {{R^{ - 1/2}}\eta  + c\left( \theta  \right)} \right|}^2}} \right)}}\prod\limits_{j = 1}^2 {\widehat\varphi } \left( {{\eta _j}} \right)} d\eta \nonumber.
\end{align}
After simple calculation,
\[\biggl|R\gamma(\frac{t}{R^{2}}) -R\gamma(0) \biggl|  \le C_{\alpha}R^{1/2}.\]
Put
\begin{equation}\label{large tube}
{T_{\theta ,\nu }}:= \left\{ {\left( {x,t} \right):0 \le t \le  R,\left| {x- c\left( \nu  \right) + 2tc\left( \theta  \right)} \right| \le {R^{\frac{1}{2} + \delta }}} \right\},\delta  \ll \varepsilon,
\end{equation}
which is a tube with direction
$G\left( \theta  \right) = \left( { - 2c\left( \theta  \right),1} \right)$. Integration by parts implies that in $B\left( {0,R} \right) \times \left[ {0,R} \right],$
\begin{equation}\label{Eq25}
\left| e^{it\Delta} \tilde{\varphi}_{\theta, \nu}(x+R\gamma(\frac{t}{R^{2}})) \right| \le \frac{1}{{\sqrt R }}{\chi^* _{{T_{\theta ,\nu }}}}\left( {x,t} \right),
\end{equation}
where $\chi^* _{{T_{\theta ,\nu }}}$ denotes a bump function satisfying $\chi^* _{{T_{\theta ,\nu }}}=1$ on the tube ${T_{\theta ,\nu }}$, and $\chi^* _{{T_{\theta ,\nu }}}=O(R^{-1000})$ outside ${T_{\theta ,\nu }}$. So we can essentially treat $\chi^* _{{T_{\theta ,\nu }}}$ by $\chi _{{T_{\theta ,\nu }}}$ which is the indicator function on the tube $T_{\theta ,\nu }$.

We do not know if  the Fourier transform of  ${e^{it\Delta}}{\tilde{\varphi} _{\theta ,\nu }}(x+R\gamma(\frac{t}{R^{2}}))$ is concentrated near a paraboloid. This  brings the main difficulty for us to apply the decoupling method to improve our result in Theorem \ref{Theorem3.1}.

  \subsection{ Proof of Theorem \ref{Theorem3.1}}\label{Section 5}
To prove Theorem \ref{Theorem3.1}, we just consider the case $p = 3.2$, other cases will follow from  H\"{o}lder's inequality. Next we will show inequality (\ref{main reduction estimate}) via induction on both $R$ and $A$.

Fix $\gamma \in \Gamma_{\alpha, R^{-1}}$. We say that we are in the algebraic case if there is a transverse complete intersection $Z\left( P \right)$ of dimension $2$, where $\deg P\left( z \right) \le D = D\left( \varepsilon  \right)$ (we will give the definition about $D\left( \varepsilon  \right)$ later), so that
\begin{align}\label{Eq51}
&\left\| {{e^{it\Delta}}f(x+R\gamma(\frac{t}{R^{2}}))} \right\|_{BL_{k,A}^p{L^q}\left( {B\left( {0,R} \right) \times \left[ {0,R} \right]} \right)}^{} \nonumber\\
&\le C\left\| {{e^{it\Delta}}f(x+R\gamma(\frac{t}{R^{2}}))} \right\|_{BL_{k,A}^p{L^q}(\left( {B\left( {0,R} \right) \times \left[ {0,R} \right]) \cap {N_{{R^{1/2 + \delta }}}}\left( {Z\left( P \right)} \right)} \right)}^{},
\end{align}
here $N_{R^{1/2+\delta}}(Z(P))$ denotes the $R^{1/2+\delta}$-neighborhood of $Z(P)$. Otherwise we are in the cellular case.

We will use polynomial partitioning as in \cite{DL}. There exist a non-zero polynomial $P\left( z \right) = \prod\limits_l {{Q_l}\left( z \right)} $ of degree at most $D$ and 2-dimensional transverse complete intersection $Z\left( P \right)$, such that $\left( {{\mathbb{R}^2} \times \mathbb{R}} \right)\backslash Z\left( P \right)$ is a union of $\sim {D^3}$ disjoint cells ${{\rm O}_i}$, and for each $i$, we have
\[\left\| {{e^{it\Delta}}f(x+R\gamma(\frac{t}{R^{2}}))} \right\|_{BL_{k,A}^p{L^q}\left( {B\left( {0,R} \right) \times \left[ {0,R} \right]} \right)}^p \le C{D^3}\left\| {{e^{it\Delta}}f(x+R\gamma(\frac{t}{R^{2}}))} \right\|_{BL_{k,A}^p{L^q}\left( {\left( {B\left( {0,R} \right) \times \left[ {0,R} \right]} \right) \cap {{\rm O}_i}} \right)}^p.\]

Set
\[W := {N_{{R^{1/2 + \delta }}}}\left( {Z\left( P \right)} \right),\hspace{0.2cm}{\rm O}_i^{'} := {\rm O}_i^{}\backslash W.\]
\textbf{Cellular case.}  In this case, $W \subset  \cup {}_l{N_{{R^{1/2 + \delta }}}}\left( {Z\left( {{Q_l}} \right)} \right),$ the contribution from $W$ is negligible. Hence for each $i$,
\begin{align}\label{Eq52}
&\left\| {{e^{it\Delta}}f(x+R\gamma(\frac{t}{R^{2}}))} \right\|_{BL_{k,A}^p{L^q}\left( {B\left( {0,R} \right) \times \left[ {0,R} \right]} \right)}^p \nonumber\\
&\le C{D^3}\left\| {{e^{it\Delta}}f(x+R\gamma(\frac{t}{R^{2}}))} \right\|_{BL_{k,A}^p{L^q}\left( {\left( {B\left( {0,R} \right) \times \left[ {0,R} \right]} \right) \cap {\rm O}_i^{'}} \right)}^p.
\end{align}

We do wave packets decomposition for $f$ as inequality (\ref{wave packet decomposition}) in $B(0,R) \times [0,R]$. For each cell ${\rm O}_i^{'},$ we put
\[{{\rm T}_i} := \left\{ {\left( {\theta ,\nu } \right) \in \boldsymbol{\rm T}:{T_{\theta ,\nu}} \cap {\rm O}_i^{'} \ne \emptyset} \right\}.\]
For the function $f$, we define
\[{f_i}: = \sum\limits_{\left( {\theta ,\nu} \right) \in {{\rm T}_i}} {{f_{\theta ,\nu}}} .\]
It follows that on ${\rm O}_i^{'},$
\[\biggl| {e^{it\Delta}}f(x+R\gamma(\frac{t}{R^{2}}))\biggl| \sim \biggl|{e^{it\Delta}}{f_i}(x+R\gamma(\frac{t}{R^{2}}))\biggl|.\]
By Fundamental Theorem of Algebra, see \cite{Guth2}, for each $\left( {\theta ,\nu} \right) \in \boldsymbol{\rm T},$ we have
\[Card\left\{ {i:\left( {\theta ,\nu} \right) \in {{\rm T}_i}} \right\} \le D + 1.\]
Hence
\[\sum\limits_i {\left\| {{f_i}} \right\|_{{L^2}}^2}  \le CD\left\| f \right\|_{{L^2}}^2,\]
by pigeonhole principle, there exists ${\rm O}_i^{'}$ such that
\begin{equation}\label{Eq53}
\left\| {{f_i}} \right\|_{{L^2}}^2 \le C{D^{ - 2}}\left\| f \right\|_{{L^2}}^2.
\end{equation}
So for such $i$, by inequality (\ref{Eq52}) we get
\begin{align}\label{decompose via R/2}
&\left\| {{e^{it\Delta}}f}(x+R\gamma(\frac{t}{R^{2}})) \right\|_{BL_{k,A}^p{L^q}\left( {B\left( {0,R} \right) \times \left[ {0,R} \right]} \right)}^p  \nonumber\\
&\le C{D^3}\left\| {{e^{it\Delta}}{f_i}(x+R\gamma(\frac{t}{R^{2}}))} \right\|_{BL_{k,A}^p{L^q}\left( {\left( {B\left( {0,R} \right) \times \left[ {0,R} \right]} \right) \cap {\rm O}_i^{'}} \right)}^p  \nonumber\\
  &\le C{D^3}\sum\limits_{{B_{R/2}}\;{\mathop{\rm cov}} er\;B\left( {0,R} \right) \times [0,R]} {\left\| {{e^{it\Delta}}{f_i}(x+R\gamma(\frac{t}{R^{2}})) } \right\|_{BL_{k, A }^p{L^q}\left( {{B_{R/2}} } \right)}^p}.
    \end{align}
  In order to apply the induction on $R$, we need the following \textbf{observation}. Suppose that the projection of $B_{\frac{R}{2}}$ on $t$-direction is contained in an interval $[t_{0}, t_{0} + R/2]\subset [0,R]$. We write $R\gamma(\frac{t}{R^{2}})=\frac{R}{2}2\gamma(\frac{t}{(R/2)^{2}} \frac{1}{4}):= \frac{R}{2}\bar{\gamma}(\frac{t}{(R/2)^{2}} )$, then check that  the function $\bar{\gamma}(t) :=2\gamma(\frac{t}{4})$ satisfies
  \begin{equation}\label{observation}
  |\bar{\gamma}(t)- \bar{\gamma}(t^{\prime})| \le C_{\alpha} |t-t^{\prime}|^{\alpha},
  \end{equation}
   for each $t,t^{\prime} \in [t_{0}/(\frac{R}{2})^{2}, (t_{0} + \frac{R}{2})/(\frac{R}{2})^{2} ]$. Indeed, for such $t, t^{\prime}$,  it is easy to see that $t/4, t^{\prime}/4 \in [t_{0}/R^{2}, (t_{0} + \frac{R}{2})/R^{2}] \subset [0,R^{-1}]$. Since $\gamma \in \Gamma_{\alpha, R^{-1}}$ and $\alpha \in [1/2,1)$, we obtain
   \[|\bar{\gamma}(t)-\bar{\gamma}(t^{\prime})| =\biggl|2 \gamma(\frac{t}{4}) -2 \gamma(\frac{t^{\prime}}{4})\biggl| \le 2^{1-2\alpha}C_{\alpha} |t-t^{\prime}|^{\alpha} \le C_{\alpha} |t-t^{\prime}|^{\alpha}, \]
   and inequality (\ref{observation}) is established.
    So we can apply the induction for $R/2$ and translation invariance to get
    \begin{equation}\label{induction on R/2}
    {\left\| {{e^{it\Delta}}{f_i}(x+R\gamma(\frac{t}{R^{2}})) } \right\|_{BL_{k, A }^p{L^q}\left( {{B_{R/2}} } \right)}^p} \le C(K, \varepsilon, C_{\alpha}) (\frac{R}{2})^{(log\bar{A} - log A)\delta} (\frac{R}{2})^{\varepsilon}\|f_{i}\|_{L^{2}}^{p}.
    \end{equation}
    Then it follows from inequalities (\ref{decompose via R/2}), (\ref{induction on R/2}), (\ref{Eq53}) that
  \begin{align}
\left\| {{e^{it\Delta}}f}(x+R\gamma(\frac{t}{R^{2}})) \right\|_{BL_{k,A}^p{L^q}\left( {B\left( {0,R} \right) \times \left[ {0,R} \right]} \right)}^p \le C{D^{3 - p}}{\left( {C\left( {K,\varepsilon, C_{\alpha} } \right) R^{\delta (log\bar{ A }- log A)} {R^\varepsilon }{{\left\| f \right\|}_{{L^2}}}} \right)^p}. \nonumber
 \end{align}
We choose $D=D(\varepsilon)$ sufficiently large such that $C{D^{3 - p}} \ll 1,$ this completes the induction in the cellular case.

\textbf{ Algebraic case.} We decompose $B\left( {0,R} \right) \times \left[ {0,R} \right]$ into balls ${B_j}$ of radius $\rho ,$ ${\rho ^{1/2 + {\delta _2}}} = {R^{1/2 + \delta }}.$ Choose ${\delta _2}\gg \delta,$ so that $\rho  \sim {R^{1 - O\left( {{\delta _2}} \right)}}.$ For each $j$ we define
\[{{\rm T}_j}: = \left\{ {\left( {\theta ,\nu} \right) \in \boldsymbol{\rm T}:{T_{\theta ,\nu}} \cap {N_{{R^{1/2 + \delta }}}}\left( {Z\left( P \right)} \right) \cap {B_j} \ne \emptyset } \right\},\]
and
\[  {f_j}: = \sum\limits_{\left( {\theta ,\nu} \right) \in {{\rm T}_j}} {{f_{\theta ,\nu}}} .\]
On each ${B_j} \cap {N_{{R^{1/2 + \delta }}}}\left( {Z\left( P \right)} \right)$, we have
\[ \biggl| {e^{it\Delta}}f(x+R\gamma(\frac{t}{R^{2}}))\biggl| \sim \biggl|{e^{it\Delta}}{f_j}(x+R\gamma(\frac{t}{R^{2}}))\biggl|.\]
Therefore,
\[\left\| {{e^{it\Delta}}f(x+R\gamma(\frac{t}{R^{2}}))} \right\|_{BL_{k,A}^p{L^q}\left( {B\left( {0,R} \right) \times [0,R]} \right)}^p \le \sum\limits_j {\left\| {{e^{it\Delta}}{f_j} (x+R\gamma(\frac{t}{R^{2}}))} \right\|_{BL_{k,A}^p{L^q}\left( {{B_j} \cap {N_{{R^{1/2 + \delta }}}}\left( {Z\left( P \right)} \right)} \right)}^p} .\]

We further divide ${{\rm T}_j}$ into tubes which are tangential to $Z(P)$ and tubes which are transverse to $Z(P)$. We say that ${T_{\theta ,\nu }}$ is tangential to $Z(P)$ in ${B_j}$ if the following two conditions hold:

\textbf{Distance condition:}
\[{T_{\theta ,\nu}} \cap 2{B_j} \subset {N_{{R^{1/2 + \delta }}}}\left( {Z\left( P \right)} \right) \cap 2{B_j} = {N_{{\rho ^{1/2 + {\delta _2}}}}}\left( {Z\left( P \right)} \right) \cap 2{B_j}.\]

\textbf{Angle condition:}
If $z \in Z(P) \cap {N_{O\left( {{R^{1/2 + \delta }}} \right)}}\left( {{T_{\theta ,\nu}}} \right) \cap 2{B_j} = Z(P) \cap {N_{O\left( {{\rho ^{1/2 + {\delta _2}}}} \right)}}\left( {{T_{\theta ,\nu}}} \right) \cap 2{B_j},$ then
\[Angle\left( {G\left( \theta  \right),{T_z}Z(P)} \right) \le C{\rho ^{ - 1/2 + {\delta _2}}}.\]
Here $T_{z}Z(P)$ denotes the tangential plane of $Z(P)$ at point $z$. The tangential wave packets are defined by
\[{{\rm T}_{j,{\rm{tang}}}} := \left\{ {\left( {\theta ,\nu} \right) \in {{\rm T}_j}:{T_{\theta ,\nu }} \text{ is tangential to } Z(P) \text{ in } {B_j}} \right\},\]
and the transverse wave packets
\[{{\rm T}_{j,trans}}: = {{\rm T}_j}\backslash {{\rm T}_{j,{\rm{tang}}}}.\]
Setting
\[{f_{j,{\rm{tang}}}}: = \sum\limits_{\left( {\theta ,\nu} \right) \in {{\rm T}_{j,{\rm{tang}}}}} {{f_{\theta ,\nu}}} , \hspace{0.8cm} {f_{j,{\rm{trans}}}}: = \sum\limits_{\left( {\theta ,\nu} \right) \in {{\rm T}_{j,{\rm{trans}}}}} {{f_{\theta ,\nu }}} ,\]
so
\[{f_j} = {f_{j,{\rm{tang}}}} + {f_{j,{\rm{trans}}}}.\]
Hence, we have
\begin{align}
 \left\| {{e^{it\Delta}}f(x+R\gamma(\frac{t}{R^{2}}))} \right\|_{BL_{k,A}^p{L^q}\left( {B\left( {0,R} \right) \times [0,R]} \right)}^p &\le \sum\limits_j {\left\| {{e^{it\Delta}}{f_j}(x+ R\gamma(\frac{t}{R^{2}})) } \right\|_{BL_{k,A}^p{L^q}\left( {{B_j}} \right)}^p}  \nonumber\\
  &\le \sum\limits_j {\left\| {{e^{it\Delta}}{f_{j,{\rm{tang}}}}(x+ R\gamma(\frac{t}{R^{2}})) } \right\|_{BL_{k,\frac{A}{2}}^p{L^q}\left( {{B_j}} \right)}^p} \nonumber\\
   \:\ & + \sum\limits_j {\left\| {{e^{it\Delta}}{f_{j,{\rm{trans}}}}(x+ R\gamma(\frac{t}{R^{2}})) } \right\|_{BL_{k,\frac{A}{2}}^p{L^q}\left( {{B_j}} \right)}^p}.  \nonumber
 \end{align}
We will treat the tangential term and the transverse term respectively.

\textbf{ Algebraic transverse case.} In this case, the transverse term dominates, by induction on the radius $R$,
\begin{align}
 \left\| {{e^{it\Delta}}{f_{j,{\rm{trans}}}} (x+ R\gamma(\frac{t}{R^{2}}))} \right\|_{BL_{k,\frac{A}{2}}^p{L^q}\left( {{B_j}} \right)}^{} &= \left\| {{e^{it\Delta}}{f_{j,{\rm{trans}}}}(x+ \rho \bar{\gamma}(\frac{t}{\rho^2})) } \right\|_{BL_{k,\frac{A}{2}}^p{L^q}\left( B_{j} \right)} \nonumber\\
  &\le C\left( {K,\varepsilon, C_{\alpha} } \right){\rho^{\delta \left( {\log \overline A  - \log \overline {\frac{A}{2}} } \right)}}{ \rho  ^\varepsilon }{\left\| {{f_{j,{\rm{trans}}}}} \right\|_{{L^2}}} \nonumber\\
  &\le {R^{ O(\delta)  - \varepsilon O\left( {{\delta _2}} \right)}}C\left( {K,\varepsilon, C_{\alpha}} \right) R^{\delta (log \bar{A} - log A)}{R^\varepsilon }{\left\| {{f_{j,{\rm{trans}}}}} \right\|_{{L^2}}}. \nonumber
  \end{align}
Here we used the induction on $R$, and the similar observation for the function $\bar{\gamma}(t)=\frac{R}{\rho} \gamma(t \frac{\rho^{2}}{R^{2}})$ as we did to establish inequality (\ref{induction on R/2}). By Subsection 8.4 in \cite{Guth2} we have
\begin{equation}\label{Eq54}
\sum\limits_j {\left\| {{f_{j,{\rm{trans}}}}} \right\|_{{L^2}}^2}  \le C\left( D \right)\left\| f \right\|_{{L^2}}^2.
\end{equation}
Then
\begin{align}
& \sum\limits_j {\left\| {{e^{it\Delta}}{f_{j,{\rm{trans}}}} (x+R\gamma(\frac{t}{R^{2}}))} \right\|_{BL_{k,\frac{A}{2}}^p{L^q}\left( {{B_j}} \right)}^p}  \nonumber\\
 &\le {R^{O(\delta)  - \varepsilon O\left( {{\delta _2}} \right)}}{\left[ {C\left( {K,\varepsilon, C_{\alpha} } \right) R^{\delta (log \bar{A} - log A)} {R^\varepsilon }} \right]^p}\sum\limits_j {\left\| {{f_{j,{\rm{trans}}}}} \right\|_{{L^2}}^p}  \nonumber\\
  &\le {R^{ O(\delta) - \varepsilon O\left( {{\delta _2}} \right)}}C\left( D \right){\left[ {C\left( {K,\varepsilon, C_{\alpha} } \right) R^{\delta (log \bar{A} - log A)} {R^\varepsilon }\left\| f \right\|_{{L^2}}^{}} \right]^p}. \nonumber
\end{align}
The induction follows by choosing $\varepsilon \delta_{2} \gg \delta$ and the fact that $R$ is sufficiently large.

\textbf{Algebraic tangential case.} In this case, the tangential term dominates, we need to do wave packets decomposition in ${B_j}$ at scale $\rho .$

\textbf{Wave packets decomposition in $\boldsymbol{B_j}$.} Choose tiles $\left( {\overline \theta  ,\overline {\nu}  } \right)$ as in Subsection \ref{Section of wave packets ddecomposition},  where $\overline \theta$ is a ${\rho ^{ - 1/2}}$-cube in frequency space and $\overline \nu$ is a ${\rho ^{1/2}}$-cube in physical space. Let $\bar{\boldsymbol{\rm T}}$ be a collection of all
such tiles with fixed dimensions and coordinate axes. Put
\begin{equation}\label{base for small ball}
\widehat{{\varphi _{\overline \theta  ,\overline \nu  }}}\left( \xi  \right) = {e^{  - i{x_j} \cdot \xi - i{t_j}{{\left| \xi  \right|}^2}- ic\left( {\overline \nu  } \right) \cdot \xi }}{\widehat{\varphi _{\overline \theta  }}}\left( \xi  \right),
\end{equation}
\[{\widehat{\varphi _{\overline \theta  }}}\left( {{\xi _1},{\xi _2}} \right) = \frac{1}{{{\rho ^{ - 1/2}}}}\prod\limits_{j = 1}^2 {\widehat\varphi } \left( {\frac{{{\xi _j} - c\left( {{\theta _j}} \right)}}{{{\rho ^{ - 1/2}}}}} \right).\]
Assume that $(x_{_{j}},t_{j})$ is the center of the ball $B_{j}$.
We set
\[\widehat{{{\tilde{\varphi}}_{\overline \theta  ,\overline \nu  }}}(\xi) = {e^{  -iR\gamma(\frac{t_{j}}{R^{2}}) \cdot \xi  }}\widehat{{\varphi _{\overline \theta  ,\overline \nu  }}}\left( \xi  \right),\]
then
\begin{equation}\label{Eq56}
f = \sum\limits_{\left( {\overline \theta  ,\overline \nu  } \right) \in {\bar{\boldsymbol{\rm T}}}} {\left\langle {f,{{\tilde{\varphi}}_{\overline \theta  ,\overline \nu  }}} \right\rangle } {\tilde{\varphi}_{\overline \theta  ,\overline \nu  }}.
\end{equation}
Therefore,
\[{e^{it\Delta}}f (x+R\gamma(\frac{t}{R^{3}}))= \sum\limits_{\left( {\overline \theta  ,\overline \nu  } \right) \in \bar{\boldsymbol{\rm T}}} {\left\langle {f,{{\tilde{\varphi}}_{\overline \theta  ,\overline \nu  }}} \right\rangle } {\tilde{\psi}_{\overline \theta  ,\overline \nu  }},\]
where
\[{\tilde{\psi}_{\overline \theta  ,\overline \nu  }}(x,t) = {e^{it\Delta}}{\tilde{\varphi}_{\overline \theta  ,\overline \nu  }}(x+R\gamma(\frac{t}{R^{2}})).\]

By similar argument  as in  Subsection \ref{Section of wave packets ddecomposition},  we restrict $(x,t) \in {B_j}$, then
\begin{equation}\label{small tube decay}
\left| {{{\tilde{\psi}}_{\overline \theta  ,\overline \nu  }}\left( {x,t} \right)} \right| \le {\rho ^{ - 1/2}}{\chi _{{T_{\overline \theta  ,\overline \nu  }}}}\left( {x,t} \right),
\end{equation}
where the tube ${T_{\overline \theta  ,\overline \nu  }}$ is defined by
\begin{equation}\label{small tube}
{T_{\overline \theta  ,\overline \nu  }} := \left\{ {\left( {x,t} \right) \in {B_j}:\left| {x - {x_j} - c\left( {\overline \nu  } \right) + 2c\left( {\overline \theta  } \right)\left( {t - {t_j}} \right)} \right| \le {\rho ^{1/2 + \delta }},\left| {t - {t_j}} \right| \le \rho } \right\}.
\end{equation}
For each $\left( {\theta ,\nu } \right) \in {{\rm T}_{j,{\rm{tang}}}},$ we consider the decomposition of ${f_{\theta ,\nu }}$ in $B_{j}$,
\[{f_{\theta ,\nu}} = \sum\limits_{\left( {\overline \theta  ,\overline \nu  } \right) \in {\bar{\boldsymbol{\rm T}}}} {\left\langle {{f_{\theta ,\nu }},{{\tilde{ \varphi}  }_{\overline \theta  ,\overline \nu  }}} \right\rangle } {\tilde{\varphi}_{\overline \theta  ,\overline \nu  }},\]
wave packets $\left( {\overline \theta  ,\overline \nu  } \right)$ which contribute to ${f_{\theta ,\nu }}$ satisfy
 \begin{equation}\label{Eq57}
\left| {c\left( \theta  \right) - c\left( {\overline \theta  } \right)} \right| \le 2{\rho ^{ - 1/2}},
\end{equation}
and
\begin{equation}\label{Eq58}
\left| {c\left( \nu  \right) - c\left( {\overline \nu  } \right) - {x_j}   - 2{t_j}c\left( \theta  \right)} \right| \le {R^{1/2 + \delta }}.
\end{equation}

From inequality (\ref{Eq57}) we know that
 \begin{equation}\label{Eq59}
Angle\left( {G\left( \theta  \right),G\left( {\overline \theta  } \right)} \right) \le 2{\rho ^{ - 1/2}},
 \end{equation}
and inequality (\ref{Eq58}) implies that if $\left( {x,t} \right) \in {T_{\overline \theta  ,\overline \nu  }},$ then
 \begin{equation}\label{Eq510}
\left| {x - c\left( \nu  \right) + 2c\left( \theta  \right)t} \right| \le C{R^{1/2 + \delta }},
 \end{equation}
i.e., ${T_{\overline \theta  ,\overline \nu  }} \subset {N_{{R^{1/2 + \delta }}}}\left( {{T_{\theta ,\nu}\cap B_{j}}} \right).$

We introduce the definition of ${\left( {{R^{'}}} \right)^{ - 1/2 + {\delta _m}}}$-tangent to $Z$ in $B$ with radius $R^{'}$, $\delta_m$ is a small positive constant, $m =1, 2$. Suppose that $Z = Z\left( {{P_1},...,{P_{3 - m}}} \right)$ is a transverse complete intersection in ${\mathbb{R}^2} \times \mathbb{R}.$ We say that ${T_{\bar{\theta} ,\bar{\nu} }}$(with scale ${R^{'}}$) is ${\left( {{R^{'}}} \right)^{ - 1/2 + {\delta _m}}}$-tangent to $Z$ in $B$ if the following two conditions hold:

 \textbf{(1) Distance condition:}
\[{T_{\bar{\theta} ,\bar{\nu} }} \subset {N_{{{\left( {{R^{'}}} \right)}^{1/2 + {\delta _m}}}}}\left( Z \right) \cap B.\]

\textbf{(2) Angle condition:} If $z \in Z \cap {N_{O\left( {{{\left( {{R^{'}}} \right)}^{1/2 + {\delta _m}}}} \right)}}\left( {{T_{\bar{\theta} ,\bar{\nu} }}} \right) \cap B,$ then
\[Angle\left( {G\left( \bar{\theta}  \right),{T_z}Z} \right) \le C{\left( {{R^{'}}} \right)^{ - 1/2 + {\delta _m}}}.\]
Moreover, set
\[{{\rm T}_Z}: = \left\{ {\left( {\bar{\theta} ,\bar{\nu} } \right):{T_{\bar{\theta} ,\bar{\nu} }} \;{\rm{is}}\; {\left( {{R^{'}}} \right)^{ - 1/2 + {\delta _m}}}{\rm{-tangent}}\;{\rm{to}}\;Z\;{\rm{in}}\;B} \right\},\]
we say that $f$ is concentrated in wave packets from ${{\rm T}_{Z}}$ in $B$  if there is some $\gamma \in \Gamma_{\alpha, R^{-1}}$, such that for each $(\bar{\theta}, \bar{\nu}) \in \bar{\boldsymbol{\rm T}}$, $e^{it\Delta} f_{\bar{\theta}, \bar{\nu}} (x + R\gamma(\frac{t}{R^{2}}))$ is essentially supported in $T_{\bar{\theta}, \bar{\nu}}$, and
\[\sum\limits_{\left( {\bar{\theta} ,\bar{\nu} } \right) \notin {\rm{T}_Z}} {\left\| {{f_{\bar{\theta} ,\bar{\nu} }}} \right\|_{{L^2}}^{}}  \le RapDec\left( {{R^{'}}} \right)\left\| f \right\|_{{L^2}}^{}.\]

We claim that new wave packets of ${f_{j,{\rm{tang}}}}$ are ${\rho ^{ - 1/2 + {\delta _2}}}$-tangent to $Z\left( P \right)$ in ${B_j}$ (note that we do not make a separate notation for convenience). In fact, if  $z \in Z \cap {N_{O\left( {{\rho ^{1/2 + {\delta _2}}}} \right)}}\left( {{T_{\overline \theta  ,\overline \nu  }}} \right) \cap {B_j},$ $z \in Z \cap {N_{O\left( {{\rho ^{1/2 + {\delta _2}}}} \right)}}\left( {{T_{\theta ,\nu }}} \right) \cap {B_j},$
\[Angle\left( {G\left( {\overline \theta  } \right),{T_z}Z} \right) \le Angle\left( {G\left( {\overline \theta  } \right),G\left( \theta  \right)} \right) + Angle\left( {G\left( \theta  \right),{T_z}Z} \right) \le C\rho {^{-1/2 + {\delta _2}}}.\]
Also,
\[{T_{\overline \theta  ,\overline \nu  }} \subset {N_{{R^{1/2 + \delta }}}}\left( {{T_{\theta ,\nu }\cap B_{j}}} \right) \cap {B_j} = {N_{{\rho ^{1/2 + {\delta _2}}}}}\left( {{T_{\theta ,\nu }\cap B_{j} }} \right) \cap {B_j} \subset {N_{O\left( {{\rho ^{1/2 + {\delta _2}}}} \right)}}\left( Z \right) \cap {B_j}.\]
Moreover, for each $t\in \{t:(x,t) \in B_{j}\}$,
\begin{align}
&\sum_{(\bar{\theta},\bar{\nu }) \notin \rm{T}_{Z}} \|f_{\bar{\theta}, \bar{\nu} }\|_{L^{2}} \le \sum_{(\theta, \nu) \in T_{j,tang}} \sum_{(\bar{\theta},\bar{\nu }) \notin \rm{T}_{Z}} \biggl|  \langle f_{\theta, \nu}, \tilde{\varphi}_{\bar{\theta}, \bar{\nu} } \rangle \biggl|   \nonumber\\
&\le \sum_{(\theta, \nu) \in T_{j,tang}} \biggl|   \langle f_{\theta, \nu}, \tilde{ \varphi}_{\theta,  \nu } \rangle \biggl|   \sum_{(\bar{\theta},\bar{\nu }) \notin \rm{T}_{Z}} \biggl|   \langle \tilde{\varphi}_{\theta, \nu}, \tilde{\varphi}_{\bar{\theta}, \bar{\nu} } \rangle \biggl|   \nonumber\\
&\le  \biggl( \sum_{(\theta, \nu) \in T_{j,tang}} \biggl| \langle f_{\theta, \nu},  \tilde{\varphi}_{\theta,  \nu } \rangle \biggl|^{2} \biggl)^{1/2} \biggl[ \sum_{(\theta, \nu) \in T_{j,tang}} \biggl( \sum_{(\bar{\theta},\bar{\nu }) \notin \rm{T}_{Z}} \biggl| \langle \tilde{\varphi}_{\theta, \nu}, \tilde{\varphi}_{\bar{\theta}, \bar{\nu} } \rangle \biggl| \biggl)^{2}   \biggl]^{1/2} \nonumber\\
&\le  \biggl( \sum_{(\theta, \nu) \in T_{j,tang}} \biggl| \langle f_{\theta, \nu},  \tilde{\varphi}_{\theta,  \nu } \rangle \biggl|^{2} \biggl)^{1/2} \nonumber\\
 & \times \biggl[ \sum_{(\theta, \nu) \in T_{j,tang}} \biggl( \sum_{(\bar{\theta},\bar{\nu }) \notin \rm{T}_{Z}} \biggl| \langle e^{it\Delta} \tilde{\varphi}_{\theta, \nu}(x +R\gamma(\frac{t}{R^{2}})), e^{it\Delta} \tilde{\varphi}_{\bar{\theta}, \bar{\nu} }(x+R\gamma(\frac{t}{R^{2}})) \rangle \biggl| \biggl)^{2}   \biggl]^{1/2} \nonumber\\
&\le RapDec(R)\|f_{j,tang}\|_{L^{2}}.
\end{align}
Here we used inequalities (\ref{Eq25}) and (\ref{small tube decay}).
Therefore, $f_{j,tang}$ is concentrated in wave packets from $\rm{T}_{Z}$ in $B_{j}$.
If Theorem \ref{Theorem6.1} below holds true, using the similar observation for the function $\frac{R}{\rho} \gamma(t \frac{\rho^{2}}{R^{2}})$ as we did to establish (\ref{induction on R/2}), we have
\begin{align}
 &\left\| {{e^{it\Delta}}{f_{j,{\rm{tang}}}}(x+R\gamma(\frac{t}{R^{2}}) } \right\|_{BL_{k,\frac{A}{2}}^p{L^q}\left( {{B_j}} \right)} \nonumber\\
 &\le { {{\rho ^{\left( {2 + 1/q} \right)\left( {1/p - 1/(4 + \delta) } \right)}}\left\| {{e^{it\Delta}}{f_{j,{\rm{tang}}}}(x+R\gamma(\frac{t}{R^{2}})) } \right\|_{BL_{k,\frac{A}{2}}^{4 + \delta }{L^q}\left( {{B_j}} \right)}^{}} } \nonumber\\
  &\le { {{\rho ^{\left( {2 + 1/q} \right)\left( {1/p - 1/(4 + \delta) } \right)}}C\left( {K,D,\frac{\varepsilon }{2}, C_{\alpha}} \right){\rho^{\delta \left( {\log \overline A  - \log   A /2} \right)}}  {{\rho  }^{\frac{1}{{2\left( {4 + \delta } \right)}} - \frac{1}{4} + \frac{\varepsilon }{2}}}{{\left\| {{f_{j,{\rm{tang}}}}} \right\|}_{{L^2}}}} } \nonumber\\
  &\le {{{R^{O\left( \delta  \right) - \varepsilon /2}}C\left( {K,D,\frac{\varepsilon }{2}, C_{\alpha}} \right)R^{\delta(log\bar{A} -logA )} {R^\varepsilon }{{\left\| {{f_{j,{\rm{tang}}}}} \right\|}_{{L^2}}}} } \nonumber\\
  &\le {R^{O\left( \delta  \right) - \varepsilon /2}}{ {C\left( {K,\varepsilon, C_{\alpha} } \right)R^{\delta(log\bar{A} -logA )}  {R^\varepsilon }{{\left\| f \right\|}_{{L^2}}}} }, \nonumber
 \end{align}
where we choose $C\left( {K,\varepsilon, C_{\alpha} } \right) \ge C\left( {K,D,\frac{\varepsilon }{2}}, C_{\alpha} \right).$ Thus,
\[\sum\limits_j {\left\| {{e^{it\Delta}}{f_{j,{\rm{tang}}}}(x+R\gamma(\frac{t}{R^{2}}))} \right\|_{BL_{k,\frac{A}{2}}^p{L^q}\left( {{B_j}} \right)}^p}  \le {R^{O\left( {{\delta _2}} \right)}}{R^{O\left( \delta  \right) - p\varepsilon /2}}{\left[ {C\left( {K,\varepsilon, C_{\alpha} } \right)R^{\delta(log\bar{A} -logA )}  {R^\varepsilon }{{\left\| f \right\|}_{{L^2}}}} \right]^p}.\]
The induction closes for the fact that ${\delta} \ll {\delta _2} \ll \varepsilon $ and $R$ is sufficiently large. Note that the constants throughout the proof are all independent of the choice of $\gamma$.

\begin{theorem}\label{Theorem6.1}
 Let $\gamma \in \Gamma_{\alpha, R^{-1}}$. Suppose that $Z=Z\left( P \right) \subset {\mathbb{R}^2} \times \mathbb{R}$ is a transverse complete intersection determined by some $P\left( z \right)$ with $degP\left( z \right) \le D_{Z}.$ For all $f$ with $ supp \hat{f} \subset B\left( {0,1} \right)$, $f$ can be decomposed by
\[f= \sum_{(\theta, \nu) \in \boldsymbol{\rm T}}f_{\theta, \nu} : = \sum_{(\theta, \nu) \in \boldsymbol{\rm T}} \langle f, \tilde{\varphi}_{\theta, \nu} \rangle \tilde{\varphi}_{\theta, \nu}. \]
In $B(0,R)\times [0,R]$, $e^{it\Delta}f_{\theta, \nu} (x + R\gamma(\frac{t}{R^{2}})) $ is essentially supported in $T_{\theta, \nu}$, $(\theta, \nu) \in \boldsymbol{\rm T}$ (given by equality (\ref{large tube})), and satisfies
\[\sum\limits_{\left( {\theta ,\nu} \right) \notin {\rm{T}_Z}} {\left\| {{f_{\theta ,\nu }}} \right\|_{{L^2}}^{}}  \le RapDec\left( {{R}} \right)\left\| f \right\|_{{L^2}}^{},\]
where
\[{{\rm T}_Z}: = \left\{ {\left( {\theta ,\nu} \right):{T_{\theta ,\nu }} \;is\; { {{R}} ^{ - 1/2 + {\delta _2}}}{\rm{-tangent}}\;to\;Z(P)\;in\;B(0,R) \times [0,R]} \right\}.\]
Then  for any $\varepsilon  > 0$ and $p > 4,$ there exist positive constants $\overline A  = \overline A \left( \varepsilon  \right)$ and $C\left( {K,D_{Z},\varepsilon, C_{\alpha} } \right)$ such that
\begin{align}\label{Eq61}
&\left\| {{e^{it\Delta}}f(x+R\gamma(\frac{t}{R^{2}}))} \right\|_{BL_{k,A}^p{L^q}\left( {B\left( {0,{R}} \right) \times [0,R]} \right)}^{} \nonumber\\
&\le C\left( {K,D_{Z},\varepsilon, C_{\alpha} } \right){R^{\delta \left( {\log \overline A  - \log A} \right)}} { R ^{\frac{1}{{2p}} - \frac{1}{4} + \varepsilon }}{\left\| f \right\|_{{L^2}}}
\end{align}
holds for all $1 \le A \le \overline A .$ The constant here does not depend on the choice of $\gamma$.
\end{theorem}

We first show that Theorem \ref{Theorem6.1} is translation invariance both in $x$ and $t$. Suppose that $\gamma$ satisfies $\alpha$-H\"{o}lder condition with constant $C_{\alpha}$ on an interval $[t_{0}/R^{2}, t_{0}/R^{2} + R^{-1}]$, $t_{0} \in \mathbb{R}$. $Z(P)$ is an algebraic surface required by Theorem \ref{Theorem6.1}. If $f$ with $ supp \hat{f} \subset B\left( {0,1} \right)$ can be decomposed by
\[f=\sum_{(\theta, \nu) \in \boldsymbol{\rm T}}f_{\theta, \nu}^{x_{0}, t_{0}}= \sum_{(\theta, \nu) \in \boldsymbol{\rm T}}\langle f, \tilde{\varphi}^{x_{0},t_{0}, \gamma}_{\theta, \nu} \rangle \tilde{\varphi}^{x_{0},t_{0},\gamma}_{\theta, \nu}, \]
where
\[\widehat{\tilde{\varphi}_{\theta, \nu}^{x_{0},t_{0},\gamma}} \left( \xi  \right) = e^{-i x_{0} \cdot \xi -it_{0}|\xi|^{2}-iR\gamma(\frac{t_{0}}{R}) \cdot \xi}\widehat{{\varphi}_{\theta, \nu}} \left( \xi  \right) ,\]
and $\varphi_{\theta, \nu}$ was given in Subsection \ref{Section of wave packets ddecomposition}.
In $B(x_0, R)\times [t_0, t_0+R]$, by previous analysis in Subsection \ref{Section of wave packets ddecomposition}, ${{e^{it\Delta}}f_{\theta, \nu}^{x_{0}, t_{0}}(x+R\gamma(\frac{t}{R^{2}}))}$ is essentially supported in
\[{T^{x_{0},t_{0}}_{\theta ,\nu }} := \left\{ {\left( {x,t} \right):t_{0} \le t \le t_{0}+  R,\left| {x-x_{0}  - c\left( \nu  \right) + 2(t-t_{0}) c\left( \theta  \right)} \right| \le {R^{\frac{1}{2} + \delta }}} \right\},\delta  \ll \varepsilon.\]
Denote
\[{{\rm T}_Z}: = \left\{ {\left( {\theta,\nu} \right):{T^{x_{0},t_{0}}_{\theta,\nu }} \;{\rm{is}}\; { {{R}} ^{ - 1/2 + {\delta _2}}}{\rm{-tangent}}\;{\rm{to}}\;Z(P)\;{\rm{in}}\;B(x_{0},R) \times [t_{0},t_{0} + R]} \right\}.\]
Assume that $f$ satisfies
\[\sum\limits_{\left( {\theta ,\nu} \right) \notin {\rm{T}_Z}} {\left\| {{f^{x_{0},t_{0}}_{\theta ,\nu}}} \right\|_{{L^2}}^{}}  \le RapDec\left( {{R}} \right)\left\| f \right\|_{{L^2}}.\]
Changing variables implies,
\begin{align}
&\left\| {{e^{it\Delta}}f(x+R\gamma(\frac{t}{R^{2}}))} \right\|_{BL_{k,A}^p{L^q}\left( {B\left( {x_{0},  {R}} \right) \times [t_{0},t_{0}+ R]} \right)}^{} \nonumber\\
&= \left\| {{e^{it\Delta}}g(x+R\gamma(\frac{t+t_{0}}{R^{2}})-R\gamma(\frac{ t_{0}}{R^{2}}) )} \right\|_{BL_{k,A}^p{L^q}\left( {B\left( {0,  {R}} \right) \times [0, R]} \right)}^{},
\end{align}
where $\hat{g}(\xi) = e^{i x_{0} \cdot \xi + it_{0}|\xi|^{2} + iR\gamma(\frac{t_{0}}{R^{2}}) \cdot \xi  }\hat{f}(\xi)$.
From the decomposition of $f$, we get
\[g = \sum_{(\theta,\nu) \in \boldsymbol{\rm T}}g_{\theta, \nu} = \sum_{(\theta,\nu) \in \boldsymbol{\rm T}} \langle g, \varphi_{\theta, \nu} \rangle \varphi_{\theta, \nu}.\]
The function $e^{it\Delta}g_{\theta, \nu}(x + R\gamma(\frac{t+t_{0}}{R^{2}}) -R\gamma(\frac{t_{0}}{R^{2}}) ) $ is essentially supported in $T_{\theta, \nu}$ defined by (\ref{large tube}).
It is easy to see that $T_{\theta, \nu} =T^{x_{0},t_{0}}_{\theta, \nu} - (x_{0},t_{0})$. Therefore,  $(\theta, \nu) \in {\rm{T}}_{Z}$ if and only if $(\theta, \nu) \in {\rm{T}}_{Z-(x_{0},t_{0})}$  which is defined by
\[{{\rm T}_{Z -(x_{0},t_{0})}}: = \left\{ {\left( {\theta ,\nu} \right):{T_{\theta ,\nu}} \;{\rm{is}}\; { {{R}} ^{ - 1/2 + {\delta _2}}}{\rm{-tangent}}\;{\rm{to}}\;Z(P)-(x_{0},t_{0})\;{\rm{in}}\;B(0,R) \times [0,  R]} \right\}.\]
Moreover,
\[\langle f, \tilde{\varphi}^{x_{0},t_{0}, \gamma}_{\theta, \nu} \rangle =\langle g, \varphi_{\theta, \nu} \rangle, \]
then
\[\sum\limits_{\left( {\theta ,\nu} \right) \notin {\rm{T}_{Z-(x_{0},t_{0})}}} {\left\| {{g_{\theta ,\nu}}} \right\|_{{L^2}}^{}}  = \sum\limits_{\left( {\theta , \nu} \right) \notin {\rm{T}_Z}} {\left\| {{f^{x_{0},t_{0}}_{\theta ,\nu}}} \right\|_{{L^2}}^{}}  \le RapDec\left( {{R}} \right)\left\| f \right\|_{{L^2}}^{}.\]
So it is proved that $g$ is concentrated in wave packets which are tangential to  $Z(P)-(x_{0},t_{0})$ in $B(0,R) \times [0,R]$. Also notice that $\gamma(t+t_{0}/R^{2}) - \gamma(t_{0}/R^{2}) \in \Gamma_{\alpha, R^{-1}}$.
Thus we can apply  Theorem \ref{Theorem6.1} to $g$ to get
\begin{align}
&\left\| {{e^{it\Delta}}g(x+R\gamma(\frac{t+t_{0}}{R^{2}}) -R\gamma(\frac{t_{0}}{R^{2}}) )} \right\|_{BL_{k,A}^p{L^q}\left( {B\left( {0,{R}} \right) \times [0,R]} \right)}^{} \nonumber\\
&\le C\left( {K,D_{Z},\varepsilon, C_{\alpha}  } \right){R^{\delta \left( {\log \overline A  - \log A} \right)}} { R ^{\frac{1}{{2p}} - \frac{1}{4} + \varepsilon }}{\left\| g\right\|_{{L^2}}}.
\end{align}
Because of  $\|g\|_{L^{2}}=\|f\|_{L^{2}}$, we have
\begin{align}
&\left\| {{e^{it\Delta}}f(x+R\gamma(\frac{t}{R^{2}}))} \right\|_{BL_{k,A}^p{L^q}\left( {B\left( {x_{0},{R}} \right) \times [x_{0},t_{0} + R]} \right)}^{} \nonumber\\
&\le C\left( {K,D_{Z},\varepsilon,  C_{\alpha} } \right){R^{\delta \left( {\log \overline A  - \log A} \right)}} { R ^{\frac{1}{{2p}} - \frac{1}{4} + \varepsilon }}{\left\| f\right\|_{{L^2}}}.
\end{align}
Theorem \ref{Theorem6.1} remains valid under translation.

  \subsection{Proof of Theorem \ref{Theorem6.1}}\label{Section 6}
We will again use the induction on $R$ and $A$ to prove Theorem \ref{Theorem6.1}. The base of the induction is the fact that Theorem \ref{Theorem6.1} is trivial when $R \lesssim 1$ or $\bar{A}$ large enough and $A=1$.

Let $D = D\left( {\varepsilon, {D_Z}} \right)$, we will define it later. We say  we are in algebraic case if there is transverse complete intersection ${Y} \subset {Z}$ of dimension $1$ defined by polynomials of degree no more than $D$, such that
\begin{align}
&\left\| {{e^{it\Delta}}f (x+R\gamma(\frac{t}{R^{2}}))} \right\|_{BL_{k,A}^p{L^q}\left( {B\left( {0,R} \right) \times \left[ {0,R} \right]} \right)}^{} \nonumber\\
&\le C\left\| {{e^{it\Delta}}f(x+R\gamma(\frac{t}{R^{2}})) } \right\|_{BL_{k,A}^p{L^q}\left( ({B\left( {0,R} \right) \times \left[ {0,R} \right]) \cap {N_{{R^{1/2 + \delta_{2} }}}}\left( Y \right)} \right)}^{}.\nonumber
\end{align}
Otherwise we are in the cellular case.

\textbf{Cellular case.} For fixed  $\omega \in \Lambda^{2}\mathbb{R}^{3}$, let $Z_{\omega}:=\{z\in Z(P): \nabla P(z) \wedge \omega =0\}$ be a transverse complete intersection of dimension $1$. As the similar argument in  \cite{Guth2}, we can choose a finite set of $\omega \in \Lambda^{2}\mathbb{R}^{3}$ such that the angle of $T_{z}(Z(P))$ changes smaller than $1/1000$ on each connected exponent of $Z(P) \backslash \cup_{\omega}Z_{\omega}$. By pigeonhole principle, we can identify a significant piece $N_{1}$ of $(B\left( {0,R} \right) \times \left[ 0,R \right]) \cap {N_{{R^{1/2 + {\delta _2}}}}}\left( {Z\left( P \right)} \right)$,  where locally $Z\left( P \right)$ behaves like a $2$-plane $V$. Notice that for each $\omega$, $Z_{\omega}$ is a transverse complete intersection determined by polynomial with degree less than $D$ and we are in the cellular case, so we obtain that
\begin{align}\label{Eq62}
& \left\| {{e^{it\Delta}}f(x+R\gamma(\frac{t}{R^{2}})) } \right\|_{BL_{k,A}^p{L^q}\left( {B\left( {0,R} \right) \times \left[ {0,R} \right]} \right)}^{} \nonumber\\
 &\le C\left\| {{e^{it\Delta}}f(x+R\gamma(\frac{t}{R^{2}})) } \right\|_{BL_{k,A}^p{L^q}\left( ({B\left( {0,R} \right) \times \left[ {0,R} \right]) \cap {N_{{R^{1/2 + {\delta _2}}}}}\left( {Z\left( P \right)} \right)} \right)}^{} \nonumber\\
  &\le C\left\| {{e^{it\Delta}}f(x+R\gamma(\frac{t}{R^{2}})) } \right\|_{BL_{k,A}^p{L^q}\left( {{N_1}} \right)}^{}.
 \end{align}
  By Theorem 5.5 in \cite{Guth2}, there exists a polynomial $Q\left( z \right): = \prod\limits_{l = 1}^s {{Q_l}} $ with $\deg Q \le D$, where polynomials
\[{Q_l}\left( z \right) = {Q_{V,l}}\left( {\pi \left( z \right)} \right),\hspace{0.2cm}l = 1,2,...,s,\]
$\pi $ is the orthogonal projection from ${\mathbb{R}^2} \times \mathbb{R}$ to $V,$ then ${\mathbb{R}^2} \times \mathbb{R}\backslash Z\left( Q \right)$ is divided into $\sim {D^2}$ cells ${{\rm O}_i}$ such that
 \begin{equation}\label{Eq63}
\left\| {{e^{it\Delta}}f(x+R\gamma(\frac{t}{R^{2}})) } \right\|_{BL_{k,A}^p{L^q}\left( {{N_1}} \right)}^{p
} \le C{D^2}\left\| {{e^{it\Delta}}f(x+R\gamma(\frac{t}{R^{2}})) } \right\|_{BL_{k,A}^p{L^q}\left( {{N_1} \cap {{\rm O}_i}} \right)}^p.
\end{equation}
 For each $l$, the variety ${Y_l} = Z\left( {P,{Q_l}} \right)$ is a transverse complete intersection of dimension $1$.  Define $W: = {N_{{R^{1/2 + \delta }}}}\left( {Z\left( Q \right)} \right),$ ${\rm O}_i^{'} := {\rm O}_i^{}\backslash W.$ From the analysis in \cite{Guth2}, we have
 \[W \cap {N_1} \subset { \cup _l}{N_{O\left( {{R^{1/2 + \delta_{2} }}} \right)}}\left( {{Y_l}} \right),\]
 since we are in the cellular case, the contribution from $W \cap N_{1}$ is negligible. So we have
\begin{equation}\label{Eq64}
\left\| {{e^{it\Delta}}f(x+R\gamma(\frac{t}{R^{2}})) } \right\|_{BL_{k,A}^p{L^q}\left( {{N_1}} \right)}^{p} \le C{D^2}\left\| {{e^{it\Delta}}f(x+R\gamma(\frac{t}{R^{2}})) } \right\|_{BL_{k,A}^p{L^q}\left( {{N_1} \cap {\rm O}_i^{'}} \right)}^p.
\end{equation}
From inequalities (\ref{Eq62})-(\ref{Eq64}) we actually obtain
 \begin{align}\label{Eq65}
&\left\| {{e^{it \Delta}}f(x+R\gamma(\frac{t}{R^{2}})) } \right\|_{BL_{k,A}^p{L^q}\left( {B\left( {0,R} \right) \times [0,R]} \right)}^{p} \nonumber\\
&\le C{D^2}\left\| {{e^{it\Delta}}f(x+R\gamma(\frac{t}{R^{2}})) } \right\|_{BL_{k,A}^p{L^q}\left( ({B\left( {0,R} \right) \times [0,R]) \cap {\rm O}_i^{'}} \right)}^p.
\end{align}

 For each cell ${\rm O}_i^{'},$ we set
\[{{\rm T}_i} := \left\{ {\left( {\theta ,\nu} \right) \in {\boldsymbol{\rm T}}:{T_{\theta ,\nu}} \cap {\rm O}_i^{'} \ne \emptyset} \right\}.\]
 For the function $f$, we define
\[{f_i}: = \sum\limits_{\left( {\theta , \nu} \right) \in {{\rm T}_i}} {{f_{\theta ,\nu}}} .\]
It follows that on ${\rm O}_i^{'},$
 \begin{equation}\label{Eq66}
\biggl| {e^{it\Delta}}f(x+R\gamma(\frac{t}{R^{2}}))\biggl| \sim \biggl| {e^{it\Delta}}{f_i}(x+R\gamma(\frac{t}{R^{2}}))\biggl|.
\end{equation}
 By the Fundamental Theorem of Algebra, for each $\left( {\theta ,\nu} \right) \in {\boldsymbol{\rm T}},$
we have
\[\text{Card}\left\{ {i:\left( {\theta ,\nu} \right) \in {{\rm T}_i}} \right\} \le D + 1.\]
Hence
\[\sum\limits_i {\left\| {{f_i}} \right\|_{{L^2}}^2}  \le CD\left\| f \right\|_{{L^2}}^2,\]
by pigeonhole principle, there exists ${\rm O}_i^{'}$ such that
\begin{equation}\label{Eq67}
\left\| {{f_i}} \right\|_{{L^2}}^2 \le C{D^{ - 1}}\left\| f \right\|_{{L^2}}^2.
\end{equation}
So by inequalities (\ref{Eq65}), (\ref{Eq66}), the induction on ${R},$ and inequality (\ref{Eq67}), we have
\begin{align}
& \left\| {{e^{it\Delta}}f(x+R\gamma(\frac{t}{R^{2}})) } \right\|_{BL_{k,  A }^p{L^q}\left( {B\left( {0,R} \right) \times [0,R]} \right)}^p \nonumber\\
 &\le C{D^2}\left\| {{e^{it\Delta}}{f_i}(x+R\gamma(\frac{t}{R^{2}})) } \right\|_{BL_{k, A }^p{L^q}\left( {B\left( {0,R} \right) \times [0,R]} \right)}^p \nonumber\\
  &\le C{D^2}\sum\limits_{{B_{R/2}}\;{\mathop{\rm cov}} er\;B\left( {0,R} \right) \times [0,R]} {\left\| {{e^{it\Delta}}{f_i}(x+R\gamma(\frac{t}{R^{2}})) } \right\|_{BL_{k,  A }^p{L^q}\left( {{B_{R/2}}} \right)}^p}  \nonumber\\
  &\le C{D^{2 - \frac{p}{2}}}{\left( {C\left( {K,D_{Z},\varepsilon, C_{\alpha} } \right)\biggl({\frac{R}{2}\biggl)^{\delta \left( {\log \overline A  - \log A} \right)}}{{\left( {\frac{R}{2}} \right)}^{\frac{1}{{2p}} - \frac{1}{4} + \varepsilon }}{{\left\| f \right\|}_{{L^2}}}} \right)^p} \nonumber\\
  &\le C{D^{2 - \frac{p}{2}}}{\left( {C\left( {K,D_{Z},\varepsilon, C_{\alpha}} \right){R^{\delta \left( {\log \overline A  - \log A} \right)}} {R^{\frac{1}{{2p}} - \frac{1}{4} + \varepsilon }}{{\left\| f \right\|}_{{L^2}}}} \right)^p}, \nonumber
 \end{align}
choosing $D:=D(\varepsilon, D_Z)$ sufficiently large such that $C{D^{2 - \frac{p}{2}}} \ll 1,$ this completes the induction.

\textbf{Algebraic case. }In the algebraic case, there exists a transverse complete intersection $Y \subset Z\left( P \right)$ of dimension $1$, determined by polynomial with degree no more than  $D = D\left( {\varepsilon, {D_Z}} \right),$ so that
\begin{align}
&\left\| {{e^{it\Delta}}f (x+R\gamma(\frac{t}{R^{2}}))} \right\|_{BL_{k,A}^p{L^q}\left( {B\left( {0,R} \right) \times [0,R]} \right)}^{} \nonumber\\
&\le C\left\| {{e^{it\Delta}}f(x+R\gamma(\frac{t}{R^{2}})) } \right\|_{BL_{k,A}^p{L^q}\left( ({B\left( {0,R} \right) \times  [0,R]) \cap {N_{{R^{1/2 + {\delta _2}}}}}\left( Y \right)} \right)}^{}.
\end{align}
We decompose $B\left( {0,R} \right) \times \left[ {0,R} \right]$ into balls ${B_j}$ of radius $\rho $, ${\rho ^{1/2 + {\delta _1}}} = {R^{1/2 + {\delta _2}}},$ ${\delta _2} \ll {\delta _1},$ in fact $\rho  \sim {R^{1 - O\left( {{\delta _1}} \right)}}.$ For each $j$, we define
\[{{\rm T}_j} := \left\{ {\left( {\theta ,\nu} \right) \in {\boldsymbol{\rm T}}:{T_{\theta ,\nu }} \cap {N_{{R^{1/2 + {\delta _2}}}}}\left( {Y} \right) \cap {B_j} \ne \emptyset } \right\},\]
and
\[{f_j}: = \sum\limits_{\left( {\theta ,\nu} \right) \in {{\rm T}_j}} {{f_{\theta ,\nu}}} .\]
On each ${B_j}: B_{j} \cap {N_{{R^{1/2 + {\delta _2}}}}}(Y) \neq \emptyset$, we have
\[ \biggl| {e^{it\Delta}}f(x+R\gamma(\frac{t}{R^{2}}))\biggl| \sim \biggl|{e^{it\Delta}}{f_j}(x+R\gamma(\frac{t}{R^{2}}))\biggl|.\]
Therefore,
\[\left\| {{e^{it\Delta}}f (x+R\gamma(\frac{t}{R^{2}}))} \right\|_{BL_{k,A}^p{L^q}\left( {B\left( {0,R} \right) \times [0,R]} \right)}^p   \le \sum\limits_j {\left\| {{e^{it\Delta}}{f_j}(x+R\gamma(\frac{t}{R^{2}})) } \right\|_{BL_{k,A}^p{L^q}\left( {{B_j}} \right)}^p} .\]
We further divide ${{\rm T}_j}$ into tubes that are tangential to $Y$ and tubes that are transverse to $Y$. We say that ${T_{\theta ,\nu}}$ is tangential to $Y$ in ${B_j}$ if the following two conditions hold:

\textbf{Distance condition:}
\begin{equation}\label{Eq68}
{T_{\theta ,\nu}} \cap 2{B_j} \subset {N_{{R^{1/2 + {\delta _2}}}}}\left( Y \right) \cap 2{B_j} = {N_{{\rho ^{1/2 + {\delta _1}}}}}\left( Y \right) \cap 2{B_j}.
\end{equation}

\textbf{Angle condition:}
If $z \in Y \cap {N_{O\left( {{R^{1/2 + {\delta _2}}}} \right)}}\left( {{T_{\theta ,\nu}}} \right) \cap 2{B_j} = Y \cap {N_{O\left( {{\rho ^{1/2 + {\delta _1}}}} \right)}}\left( {{T_{\theta ,\nu}}} \right) \cap 2{B_j},$ then
\begin{equation}\label{Eq69}
Angle\left( {G\left( \theta  \right),{T_z}Y} \right) \le C{\rho ^{ - 1/2 + {\delta _1}}}.
\end{equation}
The tangential wave packets are defined by
\[{{\rm T}_{j,{\rm{tang}}}} := \left\{ {\left( {\theta ,\nu} \right) \in {{\rm T}_j}:{T_{\theta ,\nu}} \text{ is tangent to } Y \text{ in } {B_j}} \right\},\]
and the transverse wave packets
\[{{\rm T}_{j,trans}}: = {{\rm T}_j}\backslash {{\rm T}_{j,{\rm{tang}}}}.\]
Set
\[{f_{j,{\rm{tang}}}}: = \sum\limits_{\left( {\theta ,\nu} \right) \in {{\rm T}_{j,{\rm{tang}}}}} {{f_{\theta ,\nu}}} , \hspace{0.2cm} {f_{j,{\rm{trans}}}}: = \sum\limits_{\left( {\theta ,\nu} \right) \in {{\rm T}_{j,{\rm{trans}}}}} {{f_{\theta ,\nu}}} ,\]
so
\[{f_j} = {f_{j,{\rm{tang}}}} + {f_{j,{\rm{trans}}}}.\]
Therefore, we have
\begin{align}
 \left\| {{e^{it\Delta}}f(x+R\gamma(\frac{t}{R^{2}})) } \right\|_{BL_{k,A}^p{L^q}\left( {B\left( {0,R} \right) \times [0,R]} \right)}^p &\le \sum\limits_j {\left\| {{e^{it\Delta}}{f_j} (x+R\gamma(\frac{t}{R^{2}}))} \right\|_{BL_{k,A}^p{L^q}\left( {{B_j}} \right)}^p}  \nonumber\\
  &\le \sum\limits_j {\left\| {{e^{it\Delta}}{f_{j,{\rm{tang}}}}(x+R\gamma(\frac{t}{R^{2}})) } \right\|_{BL_{k,\frac{A}{2}}^p{L^q}\left( {{B_j}} \right)}^p} \nonumber\\
   &+ \sum\limits_j {\left\| {{e^{it\Delta}}{f_{j,{\rm{trans}}}}(x+R\gamma(\frac{t}{R^{2}})) } \right\|_{BL_{k,\frac{A}{2}}^p{L^q}\left( {{B_j}} \right)}^p}.  \nonumber
 \end{align}
We will treat the tangential term and the transverse term respectively. Again, we need to use wave packets decomposition in ${B_j}.$

\textbf{ Algebraic tangential case.} In this case, the tangential term dominates. We claim that the new wave packets ${T_{\overline \theta  ,\overline \nu  }}$ of ${f_{j,{\rm{tang}}}}$ are ${\rho ^{ - 1/2 + {\delta _1}}}$-tangent to $Y$ in ${B_j}$. In fact, by (\ref{Eq59}) and (\ref{Eq510}),
if $z \in Y\cap {N_{O\left( {{\rho ^{1/2 + {\delta _1}}}} \right)}}\left( {{T_{\overline \theta  ,\overline \nu  }}} \right) \cap {B_j},$ then $z \in Y \cap {N_{{O(R^{1/2 + {\delta _2}})}}}\left( {{T_{\theta ,\nu}}} \right) \cap {B_j},$ we have
\begin{equation}\label{Eq610}
Angle\left( {G\left( {\overline \theta  } \right),{T_z}Y} \right) \le Angle\left( {G\left( {\overline \theta  } \right),G\left( \theta  \right)} \right) + Angle\left( {G\left( \theta  \right),{T_z}Y} \right) \le C{\rho ^{ - 1/2 + {\delta _1}}}.
\end{equation}
Also,
\begin{equation}\label{Eq611}
{T_{\overline \theta  ,\overline \nu  }} \subset {N_{{R^{1/2 + {\delta _2}}}}}\left( {{T_{\theta ,\nu}\cap B_{j}}} \right) \cap {B_j} = {N_{{\rho ^{1/2 + {\delta _1}}}}}\left( {{T_{\theta ,\nu}\cap B_{j}}} \right) \cap {B_j} \subset {N_{O\left( {{\rho ^{1/2 + {\delta _1}}}} \right)}}\left( Y \right) \cap {B_j}.
\end{equation}
 Consider ${B_K} \times I_K^j$ such that
\[\left[ {{N_{O\left( {{\rho ^{1/2 + {\delta _1}}}} \right)}}\left( Y \right) \cap {B_j}} \right] \cap \left( {{B_K} \times I_K^j} \right) \ne \emptyset ,\]
there exists ${z_0} \in Y \cap {B_j} \cap {N_{O\left( {{\rho ^{1/2 + {\delta _1}}}} \right)}}\left( {{B_K} \times I_K^j} \right)$, for each ${T_{\overline \theta  ,\overline \nu  }}$ such that ${T_{\overline \theta  ,\overline \nu  }} \cap \left( {{B_K} \times I_K^j} \right) \ne \emptyset $, we have that ${z_0} \in Y \cap {B_j} \cap {N_{O\left( {{\rho ^{1/2 + {\delta _1}}}} \right)}}\left( {{T_{\overline \theta  ,\overline \nu  }}} \right)$, it holds
\[Angle\left( {G\left( {\overline \theta  } \right),{T_{{z_0}}}Y} \right) \le C{\rho ^{ - 1/2 + {\delta _1}}}.\]
Then for each $\tau $ with such a $\bar{\theta} $ in it, it follows
\[Angle\left( {G\left( \tau  \right),{T_{{z_0}}}Y} \right)  \le {\left( {KM} \right)^{ - 1}}.\]

Note that $T_{z_{0}}Y$ only depends on $B_{K} \times I_{K}^{j}$, so it can be chosen as one of the vectors in $V_{1},V_{2},...,V_{A}$, by the definition of the broad norm, such $\tau $ does not contribute to \[\left\| {{e^{it\Delta}}{f_{j,{\rm{tang}}}}(x+R\gamma(\frac{t}{R^{2}})) } \right\|_{BL_{k,\frac{A}{2}}^p{L^q}\left( {{B_j}} \right)}^p.\]
 Since ${f_{j,{\rm{tang}}}}$ is concentrated in wave packets from $\rm{T}_{Y}$ in ${B_j},$
\[\left\| {{e^{it\Delta}}{f_{j,{\rm{tang}}}} (x+R\gamma(\frac{t}{R^{2}}))} \right\|_{BL_{k,\frac{A}{2}}^p{L^q}\left( {{B_j}} \right)}^p \le RapDec(\rho )\left\| f \right\|_{{L^2}}^p,\]
which can be negligible. So we only need to consider the transverse case.

\textbf{Algebraic transverse case.} In this case, the transverse term dominates. So we need to estimate
\[\sum\limits_j {\left\| {{e^{it\Delta}}{f_{j,{\rm{trans}}}}(x+R\gamma(\frac{t}{R^{2}})) } \right\|_{BL_{k,\frac{A}{2}}^p{L^q}\left( {{B_j}} \right)}^p} .\]
Consider the new wave packets decomposition of ${f_{j,{\rm{trans}}}}$ in ${B_j},$ by (\ref{Eq59}) and (\ref{Eq510}), the new wave packets ${T_{\overline \theta  ,\overline \nu  }}$ satisfy
\begin{equation}\label{Eq612}
{T_{\overline \theta  ,\overline \nu  }} \subset {N_{{R^{1/2 + {\delta}}}}}\left( {{T_{\theta ,\nu}\cap B_{j}}} \right) \cap {B_j} \subset {N_{{R^{1/2 + {\delta _2}}}}}\left( Z \right) \cap {B_j}.
\end{equation}
And if $z \in Z \cap {N_{O\left( {{\rho^{1/2 + {\delta _2}}}} \right)}}\left( {{T_{\overline \theta  ,\overline \nu  }}} \right) \cap {B_j} \subset Z \cap {N_{O\left( {{R ^{1/2 + {\delta _2}}}} \right)}}\left( {{T_{\theta ,\nu}}} \right) \cap {B_j},$ then
\begin{equation}\label{Eq613}
Angle\left( {G\left( {\overline \theta  } \right),{T_z}Z} \right) \le Angle\left( {G\left( \theta  \right),{T_z}Z} \right) + Angle\left( {G\left( \theta  \right),G\left( {\overline \theta  } \right)} \right) \le C{\rho ^{ - 1/2 + {\delta _2}}}.
\end{equation}
 ${T_{\overline \theta  ,\overline \nu  }}$ is no longer $\rho^{-1/2+\delta_{2}}$-tangent to $Z$ in ${B_j}$ because the distance condition is not satisfied.

For each vector $b$ with $\left| b \right| \le {R^{1/2 + {\delta _2}}},$ define
\[{\overline {\rm T} _{Z + b}} := \left\{ {\left( {\overline \theta  ,\overline \nu} \right):{T_{\overline \theta  ,\overline \nu }} \text{ is } \rho^{-1/2+\delta_{2}} \text{-tangent to }Z + b \text{ in }{B_j}} \right\}.\]
By the angle condition, it turns out that each ${T_{\overline \theta  ,\overline \nu  }} \in {\overline {\rm T} _{Z + b}}$ for some $b$. We set
\[{f_{j,{\rm{trans,b}}}}: = \sum\limits_{\left( {\overline \theta  ,\overline \nu } \right) \in {{\overline {\rm T} }_{Z + b}}} {{f_{\overline \theta  ,\overline \nu}}} .\]
Then on ${B_j},$ it holds
\begin{equation}\label{Eq614}
\biggl| {{e^{it\Delta}}{f_{j,{\rm{trans,b}}}}}  (x+R\gamma(\frac{t}{R^{2}})) \biggl| \sim {\chi _{{N_{{\rho ^{1/2 + {\delta _2}}}}}\left( {Z + b} \right)}}\left( {x,t} \right)\biggl| {{e^{it\Delta}}{f_{j,{\rm{trans}}}}} (x+R\gamma(\frac{t}{R^{2}}))\biggl| .
\end{equation}

Next we choose a family of vectors $b \in {B_{R^{1/2 + {\delta _2}}}}.$ We cover ${N_{{R^{1/2 + {\delta _2}}}}}\left( Z \right) \cap {B_j}$ with disjoint balls of radius ${R^{1/2 + {\delta _2}}}.$  In each ball $B$,   we will dyadically pigeonhole the volume of ${N_{{\rho ^{1/2 + {\delta _2}}}}}\left( Z \right) \cap B.$

For each $s\in \mathbb{Z}$, denote
\[{\mathcal{B}_s}: = \left\{ {B\left( {\omega,{R^{1/2 + {\delta _2}}}} \right) \subset {N_{{R^{1/2 + {\delta _2}}}}}\left( Z \right) \cap {B_j}: |B\left( {\omega,{R^{1/2 + {\delta _2}}}} \right) \cap {N_{{\rho^{1/2 + {\delta _2}}}}}\left( Z \right)| \sim {2^s}} \right\}.\]
We take a value of $s$ so that
\[\left\| {{e^{it\Delta}}{f_{j,{\rm{trans}}}} (x+R\gamma(\frac{t}{R^{2}}))} \right\|_{BL_{k,\frac{A}{2}}^p{L^q}\left( {{B_j}} \right)}^p \le \left( {\log R} \right)\left\| {{e^{it\Delta}}{f_{j,{\rm{trans}}}} (x+R\gamma(\frac{t}{R^{2}}))} \right\|_{BL_{k,\frac{A}{2}}^p{L^q}\left( {{ \cup _{B \in {{\mathcal{B}}_s}}B}} \right)}^p.\]
Therefore, we only consider $\left( {\theta ,\nu } \right)$ such that ${T_{\theta ,\nu}}$ meets at least one of the balls in $\mathcal{B}_s$. We choose a random set of $\left| {{B_{{R^{1/2 + {\delta _2}}}}}} \right|/{2^s}$ vectors $b \in {B_{{R^{1/2 + {\delta _2}}}}}$. For a typical ball $B\left( {\omega,{R^{1/2 + {\delta _2}}}} \right) \in {\mathcal{B}_s},$ the union ${ \cup _b}{N_{{\rho ^{1/2 + {\delta _2}}}}}\left( {Z + b} \right) \cap {B_j}$ covers a definite fraction of the ball with high probability.  It follows
\begin{align}\label{Eq615}
&\left\| {{e^{it\Delta}}{f_{j,{\rm{trans}}}}(x+R\gamma(\frac{t}{R^{2}})) } \right\|_{BL_{k,\frac{A}{2}}^p{L^q}\left( {{B_j}} \right)}^p \nonumber\\
&\le \left( {\log R} \right)\sum\limits_b {\left\| {{e^{it\Delta}}{f_{j,{\rm{trans,b}}}} (x+R\gamma(\frac{t}{R^{2}}))} \right\|_{BL_{k,\frac{A}{2}}^p{L^q}\left( {{N_{{\rho ^{1/2 + {\delta _2}}}}}\left( {Z + b} \right) \cap {B_j}} \right)}^p}.
 \end{align}

 By the induction on $R$, we have
\begin{align}
 &\left\| {{e^{it\Delta}}{f_{j,{\rm{trans,b}}}}(x+R\gamma(\frac{t}{R^{2}})) } \right\|_{BL_{k,\frac{A}{2}}^p{L^q}\left( {{N_{{\rho ^{1/2 + {\delta _2}}}}}\left( {Z + b} \right) \cap {B_j}} \right)}^p \nonumber\\
 &\le \left\| {{e^{it\Delta}}{f_{j,{\rm{trans,b}}}} (x+R\gamma(\frac{t}{R^{2}}))} \right\|_{BL_{k,\frac{A}{2}}^p{L^q}\left( {{B_j}} \right)}^p \nonumber\\
  &\le {\left[ {C\left( {K,D_Z,\varepsilon, C_{\alpha} } \right){\rho^{\delta \left( {\log \overline A  - \log {\frac{A}{2}} } \right)}}  {{ \rho  }^{\frac{1}{{2p}} - \frac{1}{4} + \varepsilon }}{{\left\| {{f_{j,{\rm{trans,b}}}}} \right\|}_{{L^2}}}} \right]^p} \nonumber\\
  &\le {\left[ {C\left( {K,D_Z,\varepsilon, C_{\alpha} } \right){R^{O\left( \delta  \right)}}R^{-\varepsilon O(\delta_{1})} R^{\delta(log\bar{A} -log A )} {{ R  }^{\frac{1}{{2p}} - \frac{1}{4} + \varepsilon }}{{\left\| {{f_{j,{\rm{trans,b}}}}} \right\|}_{{L^2}}}} \right]^p}. \nonumber
  \end{align}
  Here we used the induction on the radius of $B_{j}$  and the observation for $\frac{R}{\rho} \gamma(t \frac{\rho^{2}}{R^{2}})$ as we did in order to get inequality  (\ref{induction on R/2}).
 If
\begin{equation}\label{Eq621}
\sum\limits_j {\sum\limits_b {\left\| {{f_{j,{\rm{trans,b}}}}} \right\|_{{L^2}}^2} }  \le \sum\limits_j {\left\| {{f_{j,{\rm{trans}}}}} \right\|_{{L^2}}^2}  \le D\left\| f \right\|_{{L^2}}^2,
\end{equation}
\begin{equation}\label{Eq620}
\mathop {\max }\limits_b \left\| {{f_{j,{\rm{trans,b}}}}} \right\|_{{L^2}}^2 \le {R^{O\left( {{\delta _2}} \right)}}{\left( {\frac{R}{\rho }} \right)^{ - 1/2}}\left\| {{f_{j,{\rm{trans}}}}} \right\|_{{L^2}}^2,
\end{equation}
then we have
\begin{align}
 &\sum\limits_j {\sum\limits_b {\left\|  e^{it\Delta}f_{j,trans,b} (x+R\gamma(\frac{t}{R^{2}})) \right\|_{BL_{k,\frac{A}{2}}^p{L^q}\left( {{B_j}} \right)}^p} }  \nonumber\\
&\le {\left[ {C\left( {K,D_Z,\varepsilon, C_{\alpha} } \right){R^{O\left( \delta  \right)}} R^{-\varepsilon O(\delta_{1})}  R^{\delta(log\bar{A} -log A )} {{ R }^{\frac{1}{{2p}} - \frac{1}{4} + \varepsilon }}} \right]^p}\sum\limits_j {\sum\limits_b {\left\| {{f_{j,{\rm{trans,b}}}}} \right\|_{{L^2}}^p} }  \nonumber\\
&\le {\left[ {C\left( {K,D_Z,\varepsilon, C_{\alpha} } \right){R^{O\left( \delta  \right)}} R^{-\varepsilon O(\delta_{1})} R^{\delta(log\bar{A} -log A )} {{ R  }^{\frac{1}{{2p}} - \frac{1}{4} + \varepsilon }}} \right]^p}\sum\limits_j {\sum\limits_b {\left\| {{f_{j,{\rm{trans,b}}}}} \right\|_{{L^2}}^2\mathop {\mathop {\max }\limits_b \left\| {{f_{j,{\rm{trans,b}}}}} \right\|_{{L^2}}^{p - 2}}\limits_{} } }  \nonumber\\
&\le {\left[ {C\left( {K,D_Z,\varepsilon, C_{\alpha} } \right){R^{O\left( \delta  \right)}} R^{-\varepsilon O(\delta_{1})} R^{\delta(log\bar{A} -log A )} {{ R  }^{\frac{1}{{2p}} - \frac{1}{4} + \varepsilon }}} \right]^p}D{R^{O\left( {{\delta _2}} \right)}}{\left( {\frac{R}{\rho }} \right)^{ - (p/2 - 1)}}\left\| f \right\|_{{L^2}}^p \nonumber\\
&=D{R^{O\left( \delta  \right)}}{R^{O\left( {{\delta _2}} \right)}}{R^{ -\varepsilon  O\left( {{\delta _1}} \right)  }}\;{\left[ {C\left( {K,D_Z,\varepsilon, C_{\alpha} } \right) R^{\delta(log\bar{A} -log A )} {R^{\frac{1}{{2p}} - \frac{1}{4} + \varepsilon }}\left\| f \right\|_{{L^2}}^{}} \right]^p} \nonumber\\
&\le {R^{O\left( {{\delta_2 }} \right)}}{R^{ - \varepsilon O\left( {{\delta _1}} \right) }}{\left[ {C\left( {K,D_Z,\varepsilon, C_{\alpha}} \right) R^{\delta(log\bar{A} -log A )} {R^{\frac{1}{{2p}} - \frac{1}{4} + \varepsilon }}\left\| f \right\|_{{L^2}}^{}} \right]^p}, \nonumber
\end{align}
so the induction closes by taking ${\delta _2} \ll   \varepsilon {\delta _1}$ and the fact that $R$ is sufficiently large. This completes the proof of Theorem \ref{Theorem6.1}.

Next we will prove inequalities (\ref{Eq621}) and (\ref{Eq620}). It is easy to check that  $f_{j,trans, b}$ and $f_{j,trans}$ are concentrated in wave packets from $(\bar{\theta}, \bar{\nu}) \in \bar{\rm{T}}_{Z+b}$ and $(\theta, \nu) \in  {\rm{T}_{Z}}$ respectively, and  the tubes  $T_{\bar{\theta}, \bar{\nu}}$'s and $T_{\theta, \nu}$'s are distributed as required by \cite[Lemma 7.4 and Lemma 7.5]{Guth2}. But Guth's results were built for the extension operator for the paraboloid, here we identity it as  the free Schr\"{o}dinger operator $e^{it\Delta}$ by Plancherel's theorem. In order to apply \cite[Lemma 7.4 and Lemma 7.5]{Guth2} to Schr\"{o}dinger operator along tangential curves, we construct a new function $\tilde{f}$ whose Fourier transform is given by $e^{iR\gamma(\frac{t_{j}}{R^{2}}) \cdot \xi}\hat{f}$. We recall that $(x_j, t_j)$ is the center of $ B_j$.

 We first do wave packets decomposition for $\tilde{f}$ in $B(0,R) \times [0,R]$. Choosing $\tilde{\tilde{\varphi}}_{\theta, \nu}$ whose Fourier transform is given by $e^{iR\gamma(\frac{t_{j}}{R^{2}})-iR\gamma(0)}\widehat{\varphi_{\theta,\nu}}$, and $\widehat{\varphi_{\theta,\nu}}$ was given in Subsection \ref{Section of wave packets ddecomposition}.
 Decompose
 \[\tilde{f}=\sum_{(\theta, \nu) \in \boldsymbol{\rm T}}\tilde{f}_{\theta, \nu} =\sum_{(\theta, \nu) \in \boldsymbol{\rm T}} \langle \tilde{f}, \tilde{\tilde{\varphi}}_{\theta, \nu} \rangle \tilde{\tilde{\varphi}}_{\theta, \nu}. \]
Also we will make the wave packet decomposition for $\tilde{f}$ in the ball $B_{j}$ with radius $\rho$. We decompose
  \[\tilde{f}=\sum_{(\bar{\theta}, \bar{\nu}) \in \bar{\boldsymbol{\rm T}}}\tilde{f}_{\bar{\theta}, \bar{\nu}} = \sum_{(\bar{\theta}, \bar{\nu}) \in \bar{\boldsymbol{\rm T}}} \langle \tilde{f}, \varphi_{\bar{\theta}, \bar{\nu}} \rangle  \varphi_{\bar{\theta}, \bar{\nu}}, \]
where $\varphi_{\bar{\theta}, \bar{\nu}}$ was defined by inequality (\ref{base for small ball}).

Now we check the following facts.

\textbf{(a)} By the same arguments as in Subsection \ref{Section of wave packets ddecomposition}, it is not hard to see that $e^{it\Delta}\tilde{f}_{\theta,\nu}$ is essentially supported in $T_{ \theta,  \nu }$ which is defined by inequality (\ref{large tube}).

\textbf{(b)} For each $(\bar{\theta}, \bar{\nu} ) \in \bar{\boldsymbol{\rm T}}$,  $e^{it\Delta}\tilde{f}_{\bar{\theta},\bar{\nu}}$ is essentially supported in $T_{\bar{\theta}, \bar{\nu}}$ (see (\ref{small tube})).

\textbf{(c)} For each $(\theta, \nu) \in \boldsymbol{\rm T}$, we have $\|f_{\theta, \nu}\|_{L^{2}} = \|\tilde{f}_{\theta, \nu}\|_{L^{2}}.$

\textbf{(d)} For each $(\bar{\theta}, \bar{\nu} ) \in \bar{\boldsymbol{\rm T}}$, it holds $\| f_{\bar{\theta}, \bar{\nu}}\|_{L^{2}}=\|\tilde{f}_{\bar{\theta}, \bar{\nu}}\|_{L^{2}}.$

 Denote
 \[\tilde{f}_{j,trans} = \sum_{(\theta, \nu) \in T_{j,trans}}\tilde{f}_{\theta, \nu}; \quad \quad \tilde{f}_{j,trans,b} = \sum_{(\bar{\theta}, \bar{\nu}) \in \bar{T}_{Z+b}}\tilde{f}_{\bar{\theta}, \bar{\nu}}.\]
By facts (a) and (b), we can apply the result of \cite[Lemma 7.4 and Lemma 7.5]{Guth2} to obtain that
 \begin{align}\label{result1}
 \sum_{b}\|\tilde{f}_{j,trans,b}\|_{L^{2}}^{2} \le \| \tilde{f}_{j,trans}\|_{L^{2}}^{2},
 \end{align}
 and
  \begin{align}\label{result2}
\mathop {\max }_{b}\|\tilde{f}_{j,trans,b}\|_{L^{2}}^{2} \le {R^{O\left( {{\delta _2}} \right)}}{\left( {\frac{R}{\rho }} \right)^{ - 1/2}} \|\tilde{f}_{j,trans}\|_{L^{2}}^{2}.
 \end{align}

Inequalities (\ref{Eq621}) and (\ref{Eq620}) can be implied by combining  inequalities (\ref{result1}), (\ref{result2}) and (\ref{Eq54}) with
\[ \quad \quad \|f_{j,trans}\|_{L^{2}} = \|\tilde{f}_{j,trans}\|_{L^{2}}, \quad \quad \|f_{j,trans,b}\|_{L^{2}} = \|\tilde{f}_{j,trans,b}\|_{L^{2}},\]
which follow from facts (c) and (d).

\section*{Conflict of interest}
 On behalf of all authors, the corresponding author states that there is no conflict of interest.

\begin{flushleft}
\vspace{0.3cm}\textsc{Wenjuan Li\\School of Mathematics and Statistics\\Northwestern Polytechnical University\\710129\\Xi'an, People's Republic of China}

\vspace{0.3cm}\textsc{Huiju Wang (Corresponding author)\\School of Mathematical Sciences\\University of Chinese Academy of Sciences\\100049\\Beijing, People's Republic of China}

\end{flushleft}

\end{document}